\documentclass[aos]{imsart}

\RequirePackage{amsthm,amsmath,amsfonts,amssymb}
\RequirePackage[numbers]{natbib}
\usepackage{hyperref}       
\usepackage{url}            
\usepackage{booktabs}       
\usepackage{nicefrac}       
\usepackage{microtype}      
\usepackage{graphicx}
\usepackage{bm}
\usepackage{xcolor}
\usepackage{lineno}
\usepackage{tikz}
\usetikzlibrary{patterns}
\usepackage{subfigure}
\usepackage{xcolor}
\usepackage{multirow}
\usepackage{algorithm}
\usepackage{algorithmic}
\usepackage{soul}

\usepackage{bigstrut}

\hypersetup{
    colorlinks = true,
    citecolor = blue,
}

\renewcommand{\vec}{\bm}

\newcommand\cB{\mathcal{B}}
\newcommand\cT{\mathcal{T}}
\newcommand\cO{\mathcal{O}}

\newcommand\cM{\mathcal{M}}
\newcommand\cN{\mathcal{N}}
\newcommand\bP{\mathbb{P}}

\newcommand\bE{\mathbb{E}}
\newcommand\bB{\mathbb{B}}
\newcommand\bD{\mathbb{D}}
\newcommand\bV{\mathbb{V}}
\newcommand{\bR}{\mathbb{R}}
\newcommand{\reach}{\rm{reach}}
\newcommand{\bfemph}[1]{{\textit{#1}}}

\startlocaldefs
\numberwithin{equation}{section}
\theoremstyle{plain}
\newtheorem{theorem}{Theorem}[section]
\newtheorem{corollary}{Corollary}[theorem]
\newtheorem{lemma}[theorem]{Lemma}
\newtheorem{proposition}[theorem]{Proposition}
\newtheorem*{remark}{Remark}
\newtheorem{Example}[theorem]{Example}

\newtheorem{definition}[theorem]{Definition}

\endlocaldefs

\begin{document}

\begin{frontmatter}
\title{Manifold Fitting}
\runtitle{Manifold Fitting}

\begin{aug}
\author[A]{\fnms{Zhigang} \snm{Yao}\thanksref{t1,t2}},
\author[A]{\fnms{Jiaji} \snm{Su}}
\author[A]{\fnms{Bingjie} \snm{Li}}
\and
\author[C,D]{\fnms{Shing-Tung} \snm{Yau}} 
\address[A]{Department of Statistics and Data Science, National University of Singapore, 117546 Singapore}


\address[C]{Department of Mathematics, Harvard University, 02138 Cambridge USA}
\address[D]{Yau Mathematical Sciences Center, Jingzhai, Tsinghua University, Haidian District, Beijing, 100084 China}

\thankstext{t1}{Research are supported by MOE Tier 2 grant A-0008520-00-00 and Tier 1 grant A8000987-00-00 at the National University of Singapore.}

\thankstext{t2}{ZY thanks the support from the Center of Mathematical Sciences and Applications (CMSA) at Harvard University during his visit since 2022. ZY thanks Professor Charles Fefferman for his helpful discussions. Part of the work has been done during the Harvard Conference on Geometry and Statistics, supported by CMSA during Feb 27-March 1, 2023.}

\end{aug}

\begin{abstract}
While classical data analysis has addressed observations that are real numbers or elements of a real vector space, at present many statistical problems of high interest in the sciences address the analysis of data that consist of more complex objects, taking values in spaces that are naturally not (Euclidean) vector spaces but which still feature some geometric structure. Manifold fitting is a long-standing problem, and has finally been addressed in recent years by  Fefferman et. al
(\cite{fefferman2020reconstruction,fefferman2021reconstruction}).
We develop a method with a theory guarantee that fits a $d$-dimensional underlying manifold from noisy observations sampled in the ambient space $\bR^D$. The new approach uses geometric structures to obtain the manifold estimator in the form of image sets via a two-step mapping approach. 
We prove that, under certain mild assumptions and with a sample size $N=\cO(\sigma^{-(d+3)})$, these estimators are true $d$-dimensional smooth manifolds whose estimation error, as measured by the Hausdorff distance, is bounded by $\cO(\sigma^2\log(1/\sigma))$ with high probability. Compared with the existing approaches proposed in \cite{fefferman2018fitting, fefferman2021fitting, genovese2014nonparametric, yao2019manifold}, our method exhibits superior efficiency while attaining very low error rates with a significantly reduced sample size, which scales polynomially in $\sigma^{-1}$ and exponentially in $d$. Extensive simulations are performed to validate our theoretical results. Our findings are relevant to various fields involving high-dimensional data in machine learning. Furthermore, our method opens up new avenues for existing non-Euclidean statistical methods in the sense that it has the potential to unify them to analyze data on manifolds in the ambience space domain.
\end{abstract}



\end{frontmatter} 

\section{Introduction}
The Whitney extension theorem, named after Hassler Whitney, is a partial converse to Taylor's theorem.
Broadly speaking, it states that any smooth function defined on a closed subset of a smooth manifold can be extended to a smooth function defined on the entire manifold. 
This question can be traced back to H. Whitney's work in the early 1930s (\cite{whitney1992analytic}), and has finally been answered in recent years by Charles Fefferman \cite{fefferman2006whitney,fefferman2005sharp}. The solution to the Whitney extension problem led to new insights into data interpolation and inspired the formulation of the Geometric Whitney Problems (\cite{fefferman2020reconstruction,fefferman2021reconstruction}): 
\begin{itemize}
    \item [Problem I.] Assume that we are given a set $\mathcal{A}\subset \bR^D$. When can we construct a smooth $d$-dimensional submanifold $\widehat\cM\subset \bR^D$ to approximate $\mathcal{A}$, and how well can $\widehat\cM$ estimate $\mathcal{A}$ in terms of distance and smoothness?
    \item [Problem II.] If $(\mathcal{A},d_\mathcal{A})$ is a metric space, when does there exist a Riemannian manifold $(\widehat{\cM},g_{\widehat{\cM}})$ that approximates $(\mathcal{A},d_\mathcal{A})$ well?
\end{itemize}


To address these problems, various mathematical approaches have been proposed (see \cite{fefferman2018fitting,fefferman2020reconstruction,fefferman2021reconstruction, fefferman2016testing,fefferman2021fitting}). However, many of these methods rely on restrictive assumptions, making it challenging to implement them as efficient algorithms. As the manifold hypothesis continues to be a foundational element in statistical research, the Geometric Whitney Problems, particularly Problem I, merit further exploration and discussion within the statistical community.

The manifold hypothesis posits that high-dimensional data typically lie close to a low-dimensional manifold. The genesis of the manifold hypothesis stems from the observation that numerous physical systems possess a limited number of underlying variables that determine their behavior, even when they display intricate and diverse phenomena in high-dimensional spaces. For instance, while the motion of a body can be expressed as high-dimensional signals, the actual motion signals comprise a low-dimensional manifold, as they are generated by a small number of joint angles and muscle activations. Analogous phenomena arise in diverse areas, such as speech signals, face images, climate models, and fluid turbulence. The manifold hypothesis is thus essential for efficient and accurate high-dimensional data analysis in fields such as computer vision, speech analysis, and medical diagnosis.

In early statistics, one common approach for approximating high-dimensional data was to use a lower-dimensional linear subspace. One widely used technique for identifying the linear subspace of high-dimensional data is Principal Component Analysis (PCA). Specifically, PCA involves computing the eigenvectors of the sample covariance matrix and then employing these eigenvectors to map the data points onto a lower-dimensional space. One of the principal advantages of methods like this is that they can yield a simplified representation of the data, facilitating visualization and analysis. Nevertheless, linear subspaces can only capture linear relationships in the data and may fail to represent non-linear patterns accurately. To address these limitations, it is often necessary to employ more advanced manifold-learning techniques that can better capture non-linear relationships and preserve key information in the data. These algorithms can be grouped into three categories based on their purpose: manifold embedding, manifold denoising, and manifold fitting. The key distinction between them is depicted in Figure \ref{Fig:Embed_Denoise_Fit}.

\begin{figure}[htb]
\tikzset{every picture/.style={line width=0.75pt}}
        \begin{minipage}{.47\textwidth}
            \resizebox{0.95\textwidth}{!}{
                \subfigure[Embedding]{
                \begin{tikzpicture}[x=0.75pt,y=0.75pt,yscale=-1,xscale=1]
                    \draw [densely dotted]   (6,55) .. controls (105,6) and (205,7) .. (306,56) ;
        
                    \fill [black] (26.08,29.67) circle (1pt);\fill [black] (45.08,49.17) circle (1pt);\fill [black] (65.58,17.42) circle (1pt);\fill [black] (73.58,40.17) circle (1pt);
                    \fill [black] (92.58,15.67) circle (1pt);\fill [black] (106.33,34.67) circle (1pt);\fill [black] (118.08,10.17) circle (1pt);\fill [black] (131.08,30.17) circle (1pt);
                    \fill [black] (154.58,28.67) circle (1pt);\fill [black] (148.83,9.17) circle (1pt);\fill [black] (176.33,32.17) circle (1pt);\fill [black] (178.33,9.67) circle (1pt);
                    \fill [black] (198.58,33.42) circle (1pt);\fill [black] (206.33,13.92) circle (1pt);\fill [black] (217.08,35.67) circle (1pt);\fill [black] (232.08,45.67) circle (1pt);
                    \fill [black] (245.83,42.42) circle (1pt);\fill [black] (265.58,50.67) circle (1pt);\fill [black] (285.33,57.67) circle (1pt);\fill [black] (227.58,14.67) circle (1pt);
                    \fill [black] (243.83,24.42) circle (1pt);\fill [black] (263.83,26.42) circle (1pt);\fill [black] (278.83,35.92) circle (1pt);\fill [black] (298.33,44.67) circle (1pt);
                    
                    \draw (145,91.37) -- (147.83,91.37) -- (147.83,70.67) -- (153.5,70.67) -- (153.5,91.37) -- (156.33,91.37) -- (150.67,105.17) -- cycle ;
                    \draw [red, densely dotted]   (0,140) -- (310,140) ;
                    
                    \fill [red] (26.08,140) circle (1pt);\fill [red] (45.08,140) circle (1pt);\fill [red] (65.58,140) circle (1pt);\fill [red] (73.58,140) circle (1pt);
                    \fill [red] (92.58,140) circle (1pt);\fill [red] (106.33,140) circle (1pt);\fill [red] (118.08,140) circle (1pt);\fill [red] (131.08,140) circle (1pt);
                    \fill [red] (154.58,140) circle (1pt);\fill [red] (148.83,140) circle (1pt);\fill [red] (176.33,140) circle (1pt);\fill [red] (178.33,140) circle (1pt);
                    \fill [red] (198.58,140) circle (1pt);\fill [red] (206.33,140) circle (1pt);\fill [red] (217.08,140) circle (1pt);\fill [red] (232.08,140) circle (1pt);
                    \fill [red] (245.83,140) circle (1pt);\fill [red] (265.58,140) circle (1pt);\fill [red] (285.33,140) circle (1pt);\fill [red] (227.58,140) circle (1pt);
                    \fill [red] (243.83,140) circle (1pt);\fill [red] (263.83,140) circle (1pt);\fill [red] (278.83,140) circle (1pt);\fill [red] (298.33,140) circle (1pt);
                \end{tikzpicture}    
            }    
            }
        \end{minipage}
        \begin{minipage}{.47\textwidth}            
            \resizebox{0.95\textwidth}{!}{
            \subfigure[Denoising]{
                \begin{tikzpicture}[x=0.75pt,y=0.75pt,yscale=-1,xscale=1]
                    \draw [densely dotted]   (6,55) .. controls (105,6) and (205,7) .. (306,56) ;
        
                    \fill [black] (26.08,29.67) circle (1pt);\fill [black] (45.08,49.17) circle (1pt);\fill [black] (65.58,17.42) circle (1pt);\fill [black] (73.58,40.17) circle (1pt);
                    \fill [black] (92.58,15.67) circle (1pt);\fill [black] (106.33,34.67) circle (1pt);\fill [black] (118.08,10.17) circle (1pt);\fill [black] (131.08,30.17) circle (1pt);
                    \fill [black] (154.58,28.67) circle (1pt);\fill [black] (148.83,9.17) circle (1pt);\fill [black] (176.33,32.17) circle (1pt);\fill [black] (178.33,9.67) circle (1pt);
                    \fill [black] (198.58,33.42) circle (1pt);\fill [black] (206.33,13.92) circle (1pt);\fill [black] (217.08,35.67) circle (1pt);\fill [black] (232.08,45.67) circle (1pt);
                    \fill [black] (245.83,42.42) circle (1pt);\fill [black] (265.58,50.67) circle (1pt);\fill [black] (285.33,57.67) circle (1pt);\fill [black] (227.58,14.67) circle (1pt);
                    \fill [black] (243.83,24.42) circle (1pt);\fill [black] (263.83,26.42) circle (1pt);\fill [black] (278.83,35.92) circle (1pt);\fill [black] (298.33,44.67) circle (1pt);
                    
                    \fill [red] (26.08,40) circle (1pt);\fill [red] (46.08,41.5) circle (1pt);\fill [red] (63.58,30.42) circle (1pt);\fill [red] (74.58,31.17) circle (1pt);
                    \fill [red] (91.58,25.67) circle (1pt);\fill [red] (106.33,24.67) circle (1pt);\fill [red] (117.08,20.17) circle (1pt);\fill [red] (131.08,21.17) circle (1pt);
                    \fill [red] (148.83,19.17) circle (1pt);\fill [red] (154.58,19.67) circle (1pt);\fill [red] (176.33,21.67) circle (1pt);\fill [red] (178,22.17) circle (1pt);
                    \fill [red] (198.58,23.42) circle (1pt);\fill [red] (206.33,23.92) circle (1pt);\fill [red] (217.08,25.67) circle (1pt);\fill [red] (227.58,23.67) circle (1pt);
                    \fill [red] (232.08,27.67) circle (1pt);\fill [red] (243.83,34.42) circle (1pt);\fill [red] (245.83,32.42) circle (1pt);\fill [red] (261.83,36.42) circle (1pt);
                    \fill [red] (265.58,38.67) circle (1pt);\fill [red] (278.83,41.92) circle (1pt);\fill [red] (291.33,48.67) circle (1pt);\fill [red] (298.33,51.67) circle (1pt);
                \end{tikzpicture}    
            }}           
            
                \vspace{20pt}
            
            \resizebox{0.95\textwidth}{!}{          
            \subfigure[Fitting]{
                \begin{tikzpicture}[x=0.75pt,y=0.75pt,yscale=-1,xscale=1]
                    \draw [densely dotted]   (6,55) .. controls (105,6) and (205,7) .. (306,56) ;
        
                    \fill [black] (26.08,29.67) circle (1pt);\fill [black] (45.08,49.17) circle (1pt);\fill [black] (65.58,17.42) circle (1pt);\fill [black] (73.58,40.17) circle (1pt);
                    \fill [black] (92.58,15.67) circle (1pt);\fill [black] (106.33,34.67) circle (1pt);\fill [black] (118.08,10.17) circle (1pt);\fill [black] (131.08,30.17) circle (1pt);
                    \fill [black] (154.58,28.67) circle (1pt);\fill [black] (148.83,9.17) circle (1pt);\fill [black] (176.33,32.17) circle (1pt);\fill [black] (178.33,9.67) circle (1pt);
                    \fill [black] (198.58,33.42) circle (1pt);\fill [black] (206.33,13.92) circle (1pt);\fill [black] (217.08,35.67) circle (1pt);\fill [black] (232.08,45.67) circle (1pt);
                    \fill [black] (245.83,42.42) circle (1pt);\fill [black] (265.58,50.67) circle (1pt);\fill [black] (285.33,57.67) circle (1pt);\fill [black] (227.58,14.67) circle (1pt);
                    \fill [black] (243.83,24.42) circle (1pt);\fill [black] (263.83,26.42) circle (1pt);\fill [black] (278.83,35.92) circle (1pt);\fill [black] (298.33,44.67) circle (1pt);
                    
                    \draw [red]   (7.75,55.5) .. controls (31,37.25) and (36.5,36.5) .. (48.75,33) .. controls (61,29.5) and (64.75,33.75) .. (82.75,30) .. controls (100.75,26.25) and (107,19.75) .. (122,18.25) .. controls (137,16.75) and (157.5,20.75) .. (175,21.75) .. controls (192.5,22.75) and (207.75,19.25) .. (224.25,23) .. controls (240.75,26.75) and (301.75,55.5) .. (307.5,59.25) ;
                \end{tikzpicture}    
            }}
        \end{minipage}       
    \caption{Illustrations for (a) manifold embedding, (b) manifold denoising, and (c) manifold fitting.}
    \label{Fig:Embed_Denoise_Fit}
\end{figure}
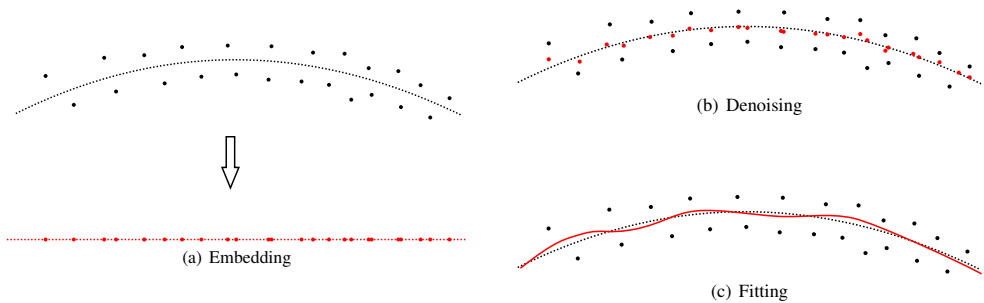


\textit{Manifold embedding}, a technique that aims to find a low-dimensional representation of high-dimensional data sets sampled near an unknown low-dimensional manifold, has gained significant attention and contributed to the development of dimensionality reduction, visualization, and clustering techniques since the beginning of the 21st century. 
This technique seeks to preserve the distances between points on the manifold. Thus the Euclidean distance between each pair of low-dimensional points is similar to the manifold distance between the corresponding high-dimensional points. Manifold embedding tries to learn a set of points in a low-dimensional space with a similar local or global geometric structure to the manifold data. The resulting low-dimensional representation usually has better aggregation and clearer demarcation between classes. Many scholars have performed various types of research on manifold-embedding algorithms, such as Isometric Mapping (\cite{tenenbaum2000global}), Locally Linear Embedding (\cite{roweis2000nonlinear, donoho-hessian-lle}), Laplacian Eigenmaps (\cite{belkin2003laplacian}), Local Tangent Space Alignment (\cite{zhang2004principal}), and Uniform Manifold Approximation Map (\cite{mcinnes2018umap}). Although these algorithms achieve useful representations of real-world data, few of them provide theoretical guarantees. Furthermore, these algorithms typically do not consider the geometry of the original manifold or provide any illustration of the smoothness of the embedding points.

\textit{Manifold denoising} aims to address outliers in data sets distributed along a low-dimensional manifold. Because of disturbances during collection, storage, and transportation, real-world manifold-distributed data often contain noise. Manifold denoising methods are designed to reduce the effect of noise and produce a new set of points closer to the underlying manifold. There are two main approaches to achieving this: feature-based and expectation-based methods. Feature-based methods extract features using techniques such as wavelet transformation (\cite{deutsch2016manifold,yang2021manifold}) or neural networks (\cite{luo2020differentiable}) and then drop non-informative features to recover denoised points via inverse transformations. However, such methods are typically validated only through simulation studies, lacking theoretical analysis. On the other hand, expectation-based methods can achieve manifold denoising by shifting the local sample mean (\cite{wang2010manifold}) or by fitting a local mean function (\cite{sober2020manifold}). However, these methods lack a solid theoretical basis or require overly restrictive assumptions.

\textit{Manifold fitting} is a crucial and challenging problem in manifold learning. It aims to reconstruct a smooth manifold that closely approximates the geometry and topology of a hidden low-dimensional manifold, using only a data set that lies on or near it. Unlike manifold embedding or denoising, manifold fitting strongly emphasizes the local and global properties of the approximation. It seeks to ensure that the generated manifold's geometry, particularly its curvature and smoothness, is precise. The application of manifold fitting can significantly enhance data analysis by providing a deeper understanding of data geometry. A key benefit of manifold fitting is its ability to uncover the shape of the hidden manifold by projecting the samples onto the learned manifold.
For example, when reproducing the three-dimensional structure of a given protein molecule, the molecule must be photoed from different angles several times via cryo-electron microscopy (cryo-EM). Although the orientation of the molecule is equivalent to the Lie group $SO(3)$, the cryo-EM images are often buried by a high-dimensional noise because of the scale of the pixels.
Manifold fitting helps recover the underlying low-dimensional Lie group of protein-molecule images and infer the structure of the protein from it. In a similar manner, manifold fitting can also be used for light detection and ranging (\cite{kim2021nanophotonics}), as well as wind-direction detection (\cite{dang2015wind}). In addition, manifold fitting can generate manifold-valued data with a specific distribution. This capability is potentially useful in generative machine-learning models, such as Generative Adversarial Network (GAN, \cite{goodfellow2014generative}).

\subsection{Main Contribution}
The main objective of this paper is to address the problem of manifold fitting by developing a smooth manifold estimator based on a set of noisy observations in the ambient space. Our goal is to achieve a state-of-the-art geometric error bound while preserving the geometric properties of the manifold. To this end, we employ the Hausdorff distance to measure the estimation error and reach to quantify the smoothness of manifolds. Further details and definitions of these concepts are provided in Section \ref{Section:Notation}.

Specifically, we consider a random vector $Y \in \mathbb{R}^D$ that can be expressed as
\begin{equation}
    \label{eq:Add_model}
    Y = X + \xi,
\end{equation}
where $X \in \bR^D$ is an unobserved random vector following a distribution $\omega$ supported on the latent manifold $\cM$, and $\xi \sim \phi_\sigma$ represents the ambient-space observation noise, independent of $X$, with a standard deviation $\sigma$. The distribution of $Y$ can be viewed as the convolution of $\omega$ and $\phi_\sigma$, whose density at point $y$ can be expressed as
\begin{equation}
    \label{eq:def:nu}
    \nu(y) = \int_\cM \phi_\sigma(y-x)\omega(x)d x.
\end{equation}
Assume $\mathcal{Y} = \{y_i\}_{i=1}^N \subset \bR^D$ is the collection of observed data points,
also in the form of
\begin{equation}
    y_i = x_i + \xi_i, \quad \text{ for } i = 1,\cdots,N, 
\end{equation}
with $(y_i, x_i,\xi_i)$ being $N$ independent and identical realizations of $(Y,X,\xi)$. Based on $\mathcal{Y}$, we construct an estimator $\widehat{\cM}$ for $\cM$ and provide theoretical justification for it under the following main assumptions:

\begin{itemize}
    \item The latent manifold $\cM$ is a compact and twice-differentiable $d$-dimensional sub-manifold, embedded in the ambient space $\bR^D$. Its volume with respect to the $d$-dimensional Hausdorff measure is upper bounded by $V$, and its reach is lower bounded by a fixed constant $\tau$.
    
    \item The distribution $\omega$ is a uniform distribution, with respect to the $d$-dimensional Hausdorff measure, on $\mathcal{M}$.
    
    \item The noise distribution $\phi_\sigma$ is a Gaussian distribution supported on $\mathbb{R}^D$ with density function 
        $$
         \phi_\sigma (\xi)= (\frac{1}{2\pi \sigma^2})^{\frac{D}{2}}\exp{(-\frac{\|\xi\|_2^2}{2\sigma^2})}.
        $$
    \item The intrinsic dimension $d$ and noise standard deviation $\sigma$ are known.
\end{itemize}

In general, $\widehat{\cM}$ is constructed by estimating the projection of points. For a point $y$ in the domain $\Gamma = \{y: d(y,\mathcal{M})\leq C\sigma\}$, we estimate its projection on $\cM$ in a two-step manner: determining the direction and moving $y$ in that direction.
The estimation has both theoretical and algorithmic contributions. From the theoretical perspective:
\begin{itemize}
    \item On the population level, given the observation distribution $\nu$ and the domain $\Gamma$, we are able to obtain a smoothly bordered set $\mathcal{S}\in \mathbb{R}^D$ such that the Hausdorff distance satisfies
    $$d_H(\mathcal{S}, \mathcal{M})<c\sigma^2{\log(1/\sigma)}.$$

    \item On the sample level, given a sample set $\mathcal{Y}$, with sample size $N = \cO(\sigma^{-(d+3)})$ and $\sigma$ being sufficiently small, we are able to obtain an estimator $\widehat{\cM}$ as a smooth $d$-dimensional manifold such that
    \begin{itemize}
        \item For any point $y \in \widehat{\cM}$, $d(y, \mathcal{M})$ is less than $C\sigma^2{\log(1/\sigma)}$;
        \item For any point $x \in \mathcal{M}$, $d(x, \widehat{\cM})$ is less than $C\sigma^2{\log(1/\sigma)}$;
        \item For any two points $y_1$, $y_2$, we have $\|y_1 - y_2\|_2^2/d(y_2, T_{y_1}\widehat{\cM}) \geq c\sigma\tau$,
    \end{itemize}
    with probability $1 - C_1\exp(-C_2\sigma^{-c_1})$, for some positive constant $c$, $c_1$, $C$, $C_1$, and $C_2$.
\end{itemize}

In summary, given a set of observed samples, we can provide a smooth $d$-dimension manifold $\widehat{\cM}$ which is higher-order closer to $\mathcal{M}$ than $\mathcal{Y}$.  Meanwhile, the approximate reach of $\widehat{\cM}$ is no less than $c\sigma\tau$.

In addition to its theoretical contributions, our method has practical benefits for some applications. This paper diverges from previous literature in its motivation, as other works often define output manifolds through the roots or ridge set of a complicated mapping $f$. In contrast, we estimate the orthogonal projection onto $\cM$ for each point near $\cM$. Compared with previous manifold-fitting methods, our framework offers three notable advantages:
\begin{itemize}
    \item Our framework yields a definitive solution to the output manifold, which can be calculated in two simple steps without iteration. This results in greater efficiency than existing algorithms.
    \item Our method requires only noisy samples and does not need any information about the latent manifold, such as its dimension, thereby broadening the applicability of our framework.
    \item Our framework computes the approximate projection of an observed point onto the hidden manifold, providing a clear relationship between input and output. In comparison, previous algorithms used multiple iterative operations, making it difficult to understand the relationship between input samples and the corresponding outputs.
\end{itemize}

\subsection{Related Works}
One main source of manifold fitting would be the Delaunay triangulation \cite{lee1980two} from the 1980s. Given a sample set, a Delaunay triangulation is a meshing
in which no samples are inside the circumcircle of any triangle in the triangulation.  Based on this technique, the early manifold-fitting approaches \cite{cheng2005manifold,boissonnat2009manifold} consider dense samples without noise. In other words, the given data set constitutes $(\epsilon,\delta)$-net of the hidden manifold.  Both \cite{cheng2005manifold} and \cite{boissonnat2009manifold} generate a piecewise linear manifold by triangulation that is geometrically and topologically similar to the hidden manifold. However, the generated manifold is not smooth and the noise-free and densely distributed assumption of the given data prevents the algorithm from being widespread. 

In recent years, manifold fitting has been more intensively studied and developed, the research including the accommodation of multiple types of noise and sample distributions, as well as the smoothness of the resulting manifolds.
Genovese et al. have obtained a sequence of results from the perspective of minimax risk under Hausdorff distance (\cite{genovese2012minimax,genovese2012manifold}) with Le Cam's method. Their work starts from \cite{genovese2012minimax}, where noisy sample points are also modeled as the summation of latent random variables from the hidden manifold and additive noise, but the noise term is assumed to be bounded and perpendicular to the manifold.
The optimal minimax estimation rate is lower bounded by $\cO(N^{-2/(2+d)})$ with properly constructed extreme cases, and upper bounded by $\cO((\frac{\log N}{N})^{2/(2+d)})$ with a sieve maximum likelihood estimator (MLE).
Hence, they conclude the rate is tight up to logarithmic factors, and the optimal rate of convergence is $\cO(N^{-2/(2+d)})$. This result is impressive since the rate only depends on the intrinsic dimension $d$ instead of the ambient dimension $D$.
However, the noise assumption is not realistic, and the sieve MLE is not computationally tractable.
Their subsequent work \cite{genovese2012manifold} considers the noiseless model, clutter noise model, and additive noise model. In the additive model, the noise assumption is relaxed to general Gaussian distributions. They view the distribution of samples as a convolution of a manifold-valued distribution and a distribution of noise in ambient space, and the fitting problem is treated as a deconvolution problem. They find a lower bound for the optimal estimation rate, $\cO(\frac{1}{\log N})$, with the same methodology in \cite{genovese2012minimax}, and an upper bound as a polynomial of $\frac{1}{\log N}$ with a standard deconvolution density estimator.
Nevertheless, their output is not necessarily a manifold, and they claim that this method requires a known noise distribution, which is also unrealistic. Meanwhile, to guarantee a small minimax risk, the required sample size should be in exponential form, which is unsatisfactory.

Since a consistent estimation of the manifold requires a very large sample size, Genovese et al. avoid this difficulty by studying the ridge of the sample distribution as a proxy \cite{genovese2014nonparametric}. They begin by showing that the Hausdorff distance between the ridge of the kernel density estimator (KDE) and the ridge of the sample density is $\cO_P((\frac{\log N}{N})^{2/(D+8)})$, and then prove that the ridge of the sample density is $\cO(\sigma^2\log(1/\sigma))$ in the Hausdorff distance with their model. Consequently, the ridge of the KDE density is shown to be an estimator with rate $\cO_P((\frac{\log N}{N})^{2/(D+8)}) + \cO(\sigma^2\log(1/\sigma))$, and they adopt the mean-shift algorithm \cite{ozertem2011locally} to estimate it.
In two similar works, \cite{chen2015asymptotic,mohammed2017manifold}, ridge estimation is implemented by two other approaches with convergence guarantee. 
While these methods yield favorable results in terms of minimax risk, evaluating the smoothness of their estimators presents a challenge. Despite claims that some methods require only a small sample size, their complex algorithms may prove impractical even for toy examples. Furthermore, the feasibility of the KDE-based algorithm in high-dimensional cases remains unverified. As noted by \cite{dunson2022graph}, kernel-based methodologies which fail to consider the intrinsic geometry of the domain may lead to sub-optimal outcomes, such as convergence rates that are dependent on the ambient dimensionality, $D$, rather than the intrinsic dimensionality, $d$. Although \cite{dunson-wu-2022} introduce a local-covariance-based approach that transforms the global manifold reconstruction problem into a local Gaussian process regression problem, thereby facilitating interpolation of the estimated manifold between fitted data points, their resulting output estimator is still in the form of discrete point sets.


The manifold generated with the above methods may have a very small reach, resulting in small twists and turns that do not align with the local geometry of the hidden manifold. To address this, some new research has aimed to ensure a lower-bounded reach of the output manifold, such as \cite{fefferman2018fitting}, \cite{ yao2019manifold} and \cite{fefferman2021fitting}.
Together with \cite{mohammed2017manifold}, all four papers design smooth mappings to capture some spatial properties and depict the output manifold as its root set or ridge.
Despite the different techniques used, all these papers provide estimators, which are close to $\cM$ and have a lower-bounded reach, with high probability.
Their required sample size depends only on $\sigma$ and $d$, which is noteworthy and instructive.
The main difference is that \cite{mohammed2017manifold}, \cite{fefferman2018fitting}, and \cite{yao2019manifold} estimate the latent manifold with accuracy $\cO(\sigma)$, measured in terms of Hausdorff distance, while \cite{fefferman2021fitting} achieves a higher approximation rate $\cO(\sigma^2)$.
However, the method in \cite{fefferman2021fitting} requires more knowledge of the manifold, which conflicts with the noisy observation assumption, and the restriction of sample size and the immature algorithms for estimating the projection direction hinder the implementation of the idea.
On the other hand, obtaining a manifold defined as the ridge or root set of a function requires additional numerical algorithms.
These algorithms can be computationally expensive and affect the accuracy of the estimate. A detailed technical comparison of these approaches is provided in Section \ref{section:Detail_Review} for completeness.

\subsection{Detailed review of existing fitting algorithms}
\label{section:Detail_Review}

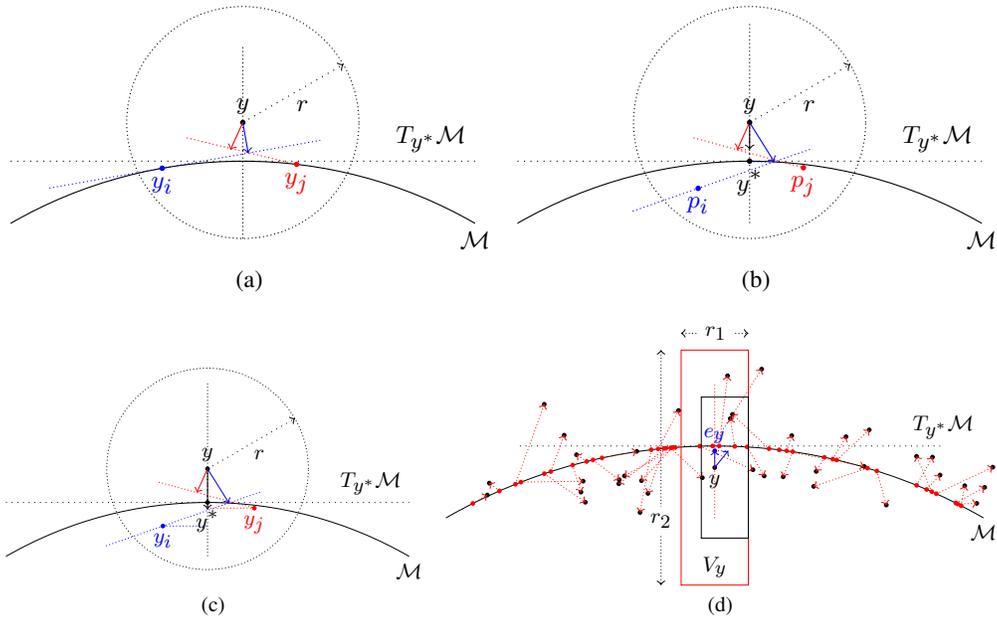
\begin{figure}[htbp]
    \centering
    \resizebox{.95\textwidth}{!}{
        \subfigure[]{
            \begin{tikzpicture}
                \draw (3,-0.8040) node[below] {$\cM$} arc (60:120:6);
                \draw [dotted] (-3,0) --(3,0) node[above left] {$T_{y^*}\cM$};
                \draw [densely dotted] (0,0.5) circle (1.5);
                \fill (0,0.5) node[above]{$y$} circle (1pt);
                \draw [densely dotted] (0,-1) -- (0,1.5);
                \draw [dotted,->] (0,0.5) -- (1.299,1.25) node[left = 15pt,below=10pt]{$r$};
                
                \fill [blue] (-1.041889, -0.09115348) node[below]{$y_i$} circle (1pt);
                \draw [blue,densely dotted] (-2.5, -0.3482578) -- (1,0.2688866);
                \draw [blue,->](0,0.5) -- (0.0696764,0.1048455);

                \fill [red] (0.695935336485, -0.04049717) node[below]{$y_j$} circle (1pt);
                \draw [red,densely dotted] (-0.7245861814802, 0.3016111038197) -- (0.695935336485, -0.04049717);
                \draw [red,->](0,0.5) -- (-0.1603345777298,0.1481173725313);
            \end{tikzpicture}
            \label{Fig:PreviousWork:Mohammed2017}
        }
        \subfigure[]{
            \begin{tikzpicture}
                \draw (3,-0.8040) node[below] {$\cM$} arc (60:120:6);
                \draw [dotted] (-3,0) --(3,0) node[above left] {$T_{y^*}\cM$};
                \draw [densely dotted] (0,0.5) circle (1.5);
                \fill (0,0) node[below]{$y^*$} circle (1pt);
                \fill (0,0.5) node[above]{$y$} circle (1pt);
                \draw [dotted] (0,0) -- (0,0.5);
                \draw [dotted,->] (0,0.5) -- (1.299,1.25) node[left = 15pt,below=10pt]{$r$};
                
                \fill [blue] (-0.6647, -0.3515) node[below]{$p_i$} circle (1pt);
                \draw [blue,densely dotted] (-1.510, -0.6469) -- (0.8,0.1603);
                \draw [blue,->](0,0.5) -- (0.3183,-0.0080);
                \draw [blue,densely dotted] (0,-0.0080)-- (0.3183,-0.0080);

                \fill [red] (0.695935336485, -0.0848142899562) node[below]{$p_j$} circle (1pt);
                \draw [red,densely dotted] (-0.7245861814802, 0.3016111038197) -- (0.695935336485, -0.0848142899562);
                \draw [red,->](0,0.5) -- (-0.1603345777298,0.1481173725313);
                \draw [red,densely dotted] (0,0.148117) -- (-0.1603345777298,0.1481173725313);

                \draw [->] (0,0.5) -- (0,0.15);
                \draw [densely dotted] (0,-0.8) -- (0,1.8);
            \end{tikzpicture}            
            \label{Fig:PreviousWork:Fefferman2018}
        }
    }    
    \resizebox{.95\textwidth}{!}{
        \subfigure[]{
            \begin{tikzpicture}
                \draw (3,-0.8040) node[below] {$\cM$} arc (60:120:6);
                \draw [dotted] (-3,0) --(3,0) node[above left] {$T_{y^*}\cM$};
                \draw [densely dotted] (0,0.5) circle (1.5);
                \fill (0,0) node[below]{$y^*$} circle (1pt);
                \fill (0,0.5) node[above]{$y$} circle (1pt);
                \draw [dotted] (0,0) -- (0,0.5);
                \draw [dotted,->] (0,0.5) -- (1.299,1.25) node[left = 15pt,below=10pt]{$r$};
                
                \fill [blue] (-0.6647, -0.3515) node[below]{$y_i$} circle (1pt);
                \draw [blue,densely dotted] (-1.510, -0.6469) -- (0.8,0.1603);
                \draw [blue,->](0,0.5) -- (0.3183,-0.0080);
                \draw [blue,densely dotted] (-0.6647, -0.3515) -- (0, -0.3515);
                \draw [red,densely dotted] (0.69593, -0.0848) -- (0, -0.0848);

                \fill [red] (0.695935336485, -0.0848142899562) node[below]{$y_j$} circle (1pt);
                \draw [red,densely dotted] (-0.7245861814802, 0.3016111038197) -- (0.695935336485, -0.0848142899562);
                \draw [red,->](0,0.5) -- (-0.1603345777298,0.1481173725313);
                \draw [->] (0,0.5) -- (0,-0.1);
                \draw [densely dotted] (0,-0.8) -- (0,1.8);
            \end{tikzpicture}            
            \label{Fig:PreviousWork:Yao2019}
        }
        \subfigure[]{
            \begin{tikzpicture}    
                \draw (0 ,0) node[below] {$\cM$} arc (60:120:8);
                \draw [dotted] (-6.9,1.0718) --(-4,1.0718) --(-0.6,1.0718) node[above] {$T_{y^*}\cM$};
                \draw (-4,-1) node[above] {$V_y$};
                \draw[red] (-4.5,-1) rectangle (-3.5,2.5);
                \draw[densely dotted, <->] (-4.5,2.75) --(-3.5,2.75) node[left = 6pt,fill=white] {$r_1$};
                \draw[densely dotted, <->] (-4.8,2.5) --(-4.8,-1) node[above = 21pt, fill=white] {$r_2$};
                \fill[red] (-2.8390,0.9871) circle (1pt);\fill (-2.9811,0.6126) circle (1pt);\draw[densely dotted, ->,red] (-2.8390,0.9871) --(-2.9811,0.6126);
                \fill[red] (-0.7018,0.3603) circle (1pt);\fill (0.1380,0.7562) circle (1pt);\draw[densely dotted, ->,red] (-0.7018,0.3603) --(0.1380,0.7562);
                \fill[red] (-4.9451,1.0158) circle (1pt);\fill (-5.1115,0.0816) circle (1pt);\draw[densely dotted, ->,red] (-4.9451,1.0158) --(-5.1115,0.0816);
                \fill[red] (-4.6281,1.0471) circle (1pt);\fill (-5.0900,0.5550) circle (1pt);\draw[densely dotted, ->,red] (-4.6281,1.0471) --(-5.0900,0.5550);
                \fill[red] (-1.5871,0.6993) circle (1pt);\fill (-1.4464,0.1926) circle (1pt);\draw[densely dotted, ->,red] (-1.5871,0.6993) --(-1.4464,0.1926);
                \fill[red] (-2.1841,0.8630) circle (1pt);\fill (-1.7007,1.3119) circle (1pt);\draw[densely dotted, ->,red] (-2.1841,0.8630) --(-1.7007,1.3119);
                \fill[red] (-6.5362,0.6591) circle (1pt);\fill (-5.9252,0.2342) circle (1pt);\draw[densely dotted, ->,red] (-6.5362,0.6591) --(-5.9252,0.2342);
                \fill[red] (-4.7549,1.0361) circle (1pt);\fill (-4.9937,0.3653) circle (1pt);\draw[densely dotted, ->,red] (-4.7549,1.0361) --(-4.9937,0.3653);
                \fill[red] (-0.3844,0.2081) circle (1pt);\fill (-0.1001,0.2172) circle (1pt);\draw[densely dotted, ->,red] (-0.3844,0.2081) --(-0.1001,0.2172);
                \fill[red] (-0.9954,0.4861) circle (1pt);\fill (-0.7597,0.9058) circle (1pt);\draw[densely dotted, ->,red] (-0.9954,0.4861) --(-0.7597,0.9058);
                \fill[red] (-3.9336,1.0715) circle (1pt);\fill (-3.8094,2.1171) circle (1pt);\draw[densely dotted, ->,red] (-3.9336,1.0715) --(-3.8094,2.1171);
                \fill[red] (-5.9049,0.8417) circle (1pt);\fill (-6.0577,0.9415) circle (1pt);\draw[densely dotted, ->,red] (-5.9049,0.8417) --(-6.0577,0.9415);
                \fill[red] (-4.6563,1.0448) circle (1pt);\fill (-4.1842,0.5993) circle (1pt);\draw[densely dotted, ->,red] (-4.6563,1.0448) --(-4.1842,0.5993);
                \fill[red] (-3.0402,1.0140) circle (1pt);\fill (-1.7753,0.7675) circle (1pt);\draw[densely dotted, ->,red] (-3.0402,1.0140) --(-1.7753,0.7675);
                \fill[red] (-0.7695,0.3905) circle (1pt);\fill (-0.9972,0.9182) circle (1pt);\draw[densely dotted, ->,red] (-0.7695,0.3905) --(-0.9972,0.9182);
                \fill[red] (-6.9658,0.5017) circle (1pt);\fill (-6.5370,1.6932) circle (1pt);\draw[densely dotted, ->,red] (-6.9658,0.5017) --(-6.5370,1.6932);
                \fill[red] (-3.7018,1.0662) circle (1pt);\fill (-3.7693,1.4801) circle (1pt);\draw[densely dotted, ->,red] (-3.7018,1.0662) --(-3.7693,1.4801);
                \fill[red] (-4.5874,1.0502) circle (1pt);\fill (-5.3435,0.6265) circle (1pt);\draw[densely dotted, ->,red] (-4.5874,1.0502) --(-5.3435,0.6265);
                \fill[red] (-4.2242,1.0687) circle (1pt);\fill (-3.4319,0.7174) circle (1pt);\draw[densely dotted, ->,red] (-4.2242,1.0687) --(-3.4319,0.7174);
                \fill[red] (-6.8839,0.5339) circle (1pt);\fill (-5.9580,0.5476) circle (1pt);\draw[densely dotted, ->,red] (-6.8839,0.5339) --(-5.9580,0.5476);
                \fill[red] (-2.9293,0.9998) circle (1pt);\fill (-2.8541,1.2264) circle (1pt);\draw[densely dotted, ->,red] (-2.9293,0.9998) --(-2.8541,1.2264);
                \fill[red] (-0.3336,0.1822) circle (1pt);\fill (-0.2795,0.4803) circle (1pt);\draw[densely dotted, ->,red] (-0.3336,0.1822) --(-0.2795,0.4803);
                \fill[red] (-0.8468,0.4241) circle (1pt);\fill (-0.6471,0.6389) circle (1pt);\draw[densely dotted, ->,red] (-0.8468,0.4241) --(-0.6471,0.6389);
                \fill[red] (-3.5170,1.0572) circle (1pt);\fill (-3.2643,0.6250) circle (1pt);\draw[densely dotted, ->,red] (-3.5170,1.0572) --(-3.2643,0.6250);
                \fill[red] (-2.3676,0.9035) circle (1pt);\fill (-2.4899,0.6930) circle (1pt);\draw[densely dotted, ->,red] (-2.3676,0.9035) --(-2.4899,0.6930);
                \fill[red] (-2.2436,0.8766) circle (1pt);\fill (-2.0460,1.2104) circle (1pt);\draw[densely dotted, ->,red] (-2.2436,0.8766) --(-2.0460,1.2104);
                \fill[red] (-4.0323,1.0717) circle (1pt);\fill (-3.3017,2.2170) circle (1pt);\draw[densely dotted, ->,red] (-4.0323,1.0717) --(-3.3017,2.2170);
                \fill[red] (-1.7272,0.7421) circle (1pt);\fill (-2.5288,0.4695) circle (1pt);\draw[densely dotted, ->,red] (-1.7272,0.7421) --(-2.5288,0.4695);
                \fill[red] (-5.6780,0.8938) circle (1pt);\fill (-5.1595,0.9918) circle (1pt);\draw[densely dotted, ->,red] (-5.6780,0.8938) --(-5.1595,0.9918);
                \fill[red] (-4.8404,1.0275) circle (1pt);\fill (-4.5391,1.6005) circle (1pt);\draw[densely dotted, ->,red] (-4.8404,1.0275) --(-4.5391,1.6005);
                \fill[red] (-7.5996,0.2162) circle (1pt);\fill (-7.3842,0.3415) circle (1pt);\draw[densely dotted, ->,red] (-7.5996,0.2162) --(-7.3842,0.3415);
                \fill[red] (-0.4140,0.2231) circle (1pt);\fill (0.0170,0.4563) circle (1pt);\draw[densely dotted, ->,red] (-0.4140,0.2231) --(0.0170,0.4563);
                \fill[red] (-7.2041,0.4021) circle (1pt);\fill (-7.3544,0.5188) circle (1pt);\draw[densely dotted, ->,red] (-7.2041,0.4021) --(-7.3544,0.5188);
                \fill[red] (-6.1143,0.7874) circle (1pt);\fill (-5.9907,0.3842) circle (1pt);\draw[densely dotted, ->,red] (-6.1143,0.7874) --(-5.9907,0.3842);
                \fill[red] (-4.7259,1.0388) circle (1pt);\fill (-5.4281,0.5091) circle (1pt);\draw[densely dotted, ->,red] (-4.7259,1.0388) --(-5.4281,0.5091);
                \fill[red] (-6.4395,0.6908) circle (1pt);\fill (-6.2409,1.2324) circle (1pt);\draw[densely dotted, ->,red] (-6.4395,0.6908) --(-6.2409,1.2324);
                \fill[red] (-3.1933,1.0310) circle (1pt);\fill (-3.7367,1.5424) circle (1pt);\draw[densely dotted, ->,red] (-3.1933,1.0310) --(-3.7367,1.5424);
                \fill[red] (-5.8119,0.8639) circle (1pt);\fill (-5.4203,0.5675) circle (1pt);\draw[densely dotted, ->,red] (-5.8119,0.8639) --(-5.4203,0.5675);

                \fill (-4,0.75) node[below]{$y$} circle (1pt);
                \fill [blue] (-4,1) node[above]{$e_y$} circle (1pt);
                \draw [blue,densely dotted] (-4,1) -- (-3.8,1);
                \draw [blue,->] (-4,0.75) -- (-4,1);       
                \draw [blue,->] (-4,0.75) -- (-3.8,1);
                \draw [red,densely dotted] (-4,0) -- (-4,2);     
                \draw (-4.2,-0.3) rectangle (-3.5, 1.8);
            \end{tikzpicture}
            \label{Fig:PreviousWork:Fefferman2021}
        }
    }
    \caption{A toy example to illustrate the methodologies in \cite{mohammed2017manifold, fefferman2018fitting, yao2019manifold,fefferman2021fitting}.}
    \label{Fig:PreviousWork}
\end{figure}

This subsection presents a review of the technical details of the previously mentioned work \cite{mohammed2017manifold, fefferman2018fitting, yao2019manifold,fefferman2021fitting}. These papers relax the requirement for sample size by exploiting the geometric properties of the data points. For ease of understanding, we introduce some common geometric notations here, while more detailed notations can be found in Section \ref{Section:Notation}. For a point $x\in \cM$, $T_x\cM$ denotes the tangent space of $\cM$ at $x$, 
and $\Pi_{x}^{\perp}$ is the orthogonal projection matrix onto the normal space of $\cM$ at $x$. For a point $y$ off $\cM$, $y^* = \arg\min_{x\in\cM} \|y-x\|_2$ denotes the projection of $y$ on $\cM$, and $\widehat{\Pi}_{y}^\perp$ is the estimator of ${\Pi}_{y^*}^\perp$. For an arbitrary matrix $A$, $\Pi_{hi}(A)$ represents its projection on the span of the eigenvectors corresponding to the largest $D-d$ eigenvalues. We use the notation $\cB_D(z,r)$ to denote a $D$-dimensional ball with center $z$ and radius $r$. To be consistent with the papers subsequently referred to, we frequently use upper- and lower-case letters (such as $c$, $c_1$, $c_2$, $C$, $C_1$, and $C_2$) to represent absolute constants. The upper and lower cases represent constants greater or less than one, respectively, and their values may vary from line to line.

\subsubsection*{An early work without noise}
One early work on manifold fitting is \cite{mohammed2017manifold}, which only focuses on the case of noiseless sample $\mathcal{X}=\{x_i\in\cM\}_{i=1}^N$. To reconstruct an $\widehat{\cM}$ with $\mathcal{X}$, the authors construct a function $f(y)$ to approximate the squared distance from an arbitrary point $y$ to $\cM$, and the ridge set of $f(y)$ is a proper estimator of $\cM$.

As stated in \cite{mohammed2017manifold}, $f(y)$ can be estimated by performing local Principal Components Analysis (PCA). The procedure is shown in Fig. \ref{Fig:PreviousWork:Mohammed2017}. For an arbitrary point $y$ close to $\cM$, its $r$ neighborhood index set is defined as 
$$I_y = \{i:\|x_i-y\|_2\leq r\}.$$

For each $i\in I_y$, $\widehat\Pi_{x_i}^\perp$ can be obtained via local PCA, and the square distance between $y$ and $T_{x_i}\cM$ is approximated by 
$$f_i(y)= \|\widehat\Pi_{x_i}^\perp(y-x_i)\|_2^2.$$
Then, $f(y)$ is designed as the weighted average of $f_i(y)$'s; that is, 
$$f(y) = \sum_{i\in I_y} \alpha_i(y) f_i(y),$$
with the weights defined as
$$\tilde{\alpha}_i(y)=\theta(\frac{\sqrt{f_i(y)}}{2r}),\quad \tilde{\alpha}(y) =  \sum_{i\in I_y}\tilde{\alpha}_i(y), \quad \alpha_i(y) = \frac{\tilde{\alpha}_i(y)}{\tilde{\alpha}(y)},$$
and $\theta(t)$ is an indicator function such that $\theta(t)=1$ for $t\leq 1/4$ and $\theta(t) = 0$ for $t\geq 1$.

The estimator $\widehat{\cM}$ is given as the ridge set of $f(y)$; that is,
$$\widehat{\cM} = \{y\in\bR^D:~d(y,\cM)\leq cr,~\Pi_{hi}(H_f(y))\partial f(y) = 0\},$$
where $H_f(y)$ is the Hessian matrix of $f$ at point $y$. Such an $\widehat{\cM}$ is claimed to have a reach bounded below by $cr$ and $\cO(r^2)$-close to $\cM$ in terms of Hausdorff distance.

Although this paper does not consider the ambient space noise and relies heavily on a well-estimated projection direction $\widehat\Pi_{x_i}^\perp$, the idea of approximating the distance function with projection matrices is desirable and provides a good direction for subsequent work.

\subsubsection*{An attempt with noise}
In the follow-up work \cite{fefferman2018fitting}, noise from the ambient space is considered.  Similar to \cite{mohammed2017manifold}, the main aim of \cite{fefferman2018fitting} is to estimate the bias from an arbitrary point to the hidden manifold with local PCA. The collection of all zero-bias points can be interpreted as an estimator for $\cM$.

To construct the bias function $f(y)$, the authors assume there is a sample set $\mathcal{Y}_0 = \{y_i\}_{i=1}^{N}$, with the sample size satisfying
$$N/\ln(N)>\frac{CV}{\omega_{min}\beta_d(r^2/\tau)^d},\quad N\leq e^D,$$
where $V$ is the volume of $\cM$, $\beta_d$ is the volume of a Euclidean unit ball in $\bR^d$, and $\omega_{min}$ is the lower bound of $\omega$ on $\cM$. Under such conditions, $\mathcal{Y}_0$ is $Cr^2/\tau$-close to $\cM$ in Hausdorff distance with probability $1-N^{-C}$. Then, a subset  $\mathcal{Y}_1 = \{p_i\} \subset \mathcal{Y}_0$ is selected greedily as a minimal $cr/d$-net of $\mathcal{Y}_0$.
For each $p_i\in \mathcal{Y}_1$, there exists a $D$-dimensional ball $U_i=\cB_D(p_i,r)$ and a $d$-dimensional ball $D_i=\cB_d(p_i,r)$, where $D_i$ can be viewed as a disc cut from $U_i$.
In the ideal case, $D_i$ should be parallel to $T_{p_i^*}\cM$. Thus, the authors provide a new algorithm to estimate the basis of $D_i$ with the sample points falling in $U_i$. The basis of $D_i$ leads to an estimator of $\Pi_{p_i}^\perp$, which is denoted by $\widehat{\Pi}_{p_i}^\perp$.

For $y$ near $\cM$, let $I_y = \{i:\|p_i-y\|_2\leq r \}$, and
$$f_i(y) = \widehat{\Pi}_{p_i}^\perp(y-p_i),\quad \text{for } i \in I_y.$$
Then, $f(y)$ can be constructed as 
\begin{align}\label{fy:fefferman18}
 f(y) = \sum_{i\in I_y} \alpha_i(y)(\widehat{\Pi}_y^\perp\widehat{\Pi}_{p_i}^\perp)(y-p_i),    
\end{align}
with $\widehat{\Pi}_y^\perp = \Pi_{hi}(\sum_{i\in I_y} \alpha_i(y)\widehat{\Pi}_{p_i}^\perp)$, and the weights defined as 
$$\tilde{\alpha}_i(y)=(1 - \frac{\|y - p_i\|_2^2}{r^2})^{d+2},\quad \tilde{\alpha}(y) =  \sum_{i\in I_y}\tilde{\alpha}_i(y), \quad \alpha_i(y) = \frac{\tilde{\alpha}_i(y)}{\tilde{\alpha}(y)},$$
for $y$ satisfying $\|y-p_i\|_2\leq r$ and $0$ otherwise. Subsequently, there is
$$\widehat{\cM} = \{y\in\bR^D:~d(y,\cM)\leq cr,\quad f(y) = 0\}.$$

By setting $r = \cO(\sigma)$, the authors prove $\widehat{\cM}$ is $\cO(r^2) = \cO(\sigma)$-close to $\cM$ and its reach is bounded below by $c\tau$ with probability $1-N^{-C}$. However, it is notable that the algorithm for disc-orientation estimation is not proved theoretically in the paper, and the accuracy of $f(y)$ is limited by the subsequent successive projections $\widehat{\Pi}_y^\perp\widehat{\Pi}_{p_i}^\perp$ and the lack of accuracy in estimating $\widehat{\Pi}_y^\perp$. Moreover, because of the limitation of the sample size $N$, the estimation error of the manifold has a non-zero lower bound and the practical application is very limited.

\subsubsection*{A better estimation for noisy data}
To address the issues in \cite{fefferman2018fitting}, the authors of \cite{yao2019manifold} propose an improved method that avoids the continuous projections and estimates ${\Pi}_{y^*}^\perp$ better. The authors claim that fitting the manifold is enough to estimate the projection direction and the local mean well, because the manifold can be viewed as a linear subspace locally, and the local sample mean is a good reference point for the hidden manifold.
They assume there is a sample set $\mathcal{Y} = \{y_i\}_{i=1}^N$. For each $y_i$, $\widehat{\Pi}_{y_i}^\perp$ is obtained by local PCA with $r = \cO(\sqrt{\sigma})$. Then, for an arbitrary point $y$ with $I_y = \{i:\|y_i-y\|_2\leq r \}$, the bias function can be constructed as 
\begin{align}\label{fy:yao2019}
f(y) = \widehat{\Pi}_y^\perp(y - \sum_{i\in I_y}\alpha_i(y)y_i),    
\end{align}
where $\widehat{\Pi}_y^\perp = \Pi_{hi}(\sum_{i\in I_y}\alpha_i(y)\widehat\Pi_{y_i}^\perp)$. The weights are defined as 
$$\tilde{\alpha}_i(y)=(1 - \frac{\|y - y_i\|_2^2}{r^2})^{\beta},\quad \tilde{\alpha}(y) =  \sum_{i\in I_y}\tilde{\alpha}_i(y), \quad \alpha_i(y) = \frac{\tilde{\alpha}_i(y)}{\tilde{\alpha}(y)},$$
for $y$ satisfying $\|y-y_i\|\leq r$ and $0$ otherwise, with $\beta\geq 2$ being a fixed integer which guarantees $f(y)$ to be derivable in the second order. With such a bias function, the output manifold can be given as 
$$\widehat{\cM} = \{y\in\bR^D:~d(y,\cM)\leq cr,\quad f(y) = 0\},$$
which is shown to be $\cO(\sigma)$-close to $\cM$ in Hausdorff distance and have a reach no less than $c\tau$ with probability $1 - c\exp(-Cr^{d+2}N)$.
Although the theoretical error bound in \cite{yao2019manifold} remains the same as that in \cite{fefferman2018fitting}, the method in \cite{yao2019manifold} vastly simplifies the computational process and outperforms the previous works numerically in many cases.


\subsubsection*{The necessity of noise reduction and an attempt}
Based on the result mentioned above, the error in the manifold fitting can be attributed to two components: sampling bias and learning error, namely
$$d_H(\cM,\widehat{\cM}) \leq d_H(\cM,\mathcal{Y}) + d_H(\mathcal{Y},\widehat{\cM}),$$
where $\mathcal{Y}$ is the generic sample set. Usually, the first term can be regarded as $\cO(\sigma)$, as the Gaussian noise will die out within several $\sigma$, and the second term is bounded by $Cr^2$ with the PCA-based algorithms listed above. The optimal radius of local PCA, which balances the overall estimation error and the computation complicity, should be $r = \cO(\sqrt{\sigma})$, and leads to a fitting error such that 
$$d_H(\cM,\widehat{\cM})\leq C\sigma.$$
Since the sampling bias $d_H(\mathcal{Y}, \cM)= \cO(\sigma)$ prevents us from moving closer to $\cM$, denoising is necessary for a better $\widehat{\cM}$.

On the basis of \cite{fefferman2018fitting}, the same group of authors provides better results in \cite{fefferman2021fitting} with refined points and a net. They refine the points by constructing a mesh grid on each disc $D_i$. As illustrated in Figure \ref{Fig:PreviousWork:Fefferman2021}, each hyper-cylinder of the mesh is much longer in the direction perpendicular to the manifold than parallel.
Subsequently, in each hyper-cylinder, a subset of $\mathcal{Y}_0$ is selected with a complicated design, and their average is denoted by $e_y$.
The collection of such $e_y$ in all hyper-cylinders is denoted by $\mathcal{Y}_1$, which is shown to be $Cd\sigma^2/\tau$-close to $\cM$.

The authors take $\mathcal{Y}_1$ as the input data set of \cite{fefferman2018fitting} to perform subsampling and construct a new group of discs $\{D_i^\prime\}$. With the refined points in $\mathcal{Y}_1$ and refined discs $\{D_i^\prime\}$, the same function $f(y)$ will lead to an $\widehat{\cM}$ which is $\cO(\sigma^2)$-close to $\cM$ and has a reach no less than $c\tau$ with probability $1-N^{-C}$.

 To the best of our knowledge, the result presented in \cite{fefferman2021fitting} constitutes a state-of-the-art error bound for manifold fitting. However, some challenges exist in implementing the method described in that paper:
\begin{itemize}
    \item The refinement step for $e_y$ involves sampling directly from the latent manifold, which contradicts the initial assumption of noisy data.
    \item The algorithms for refining points and determining the orientation of discs are only briefly described and may not be directly applied to real-world data.
    \item The sample-size requirement is similar to that described in \cite{fefferman2018fitting}, further limiting the practical implementation of the algorithm.
\end{itemize}

\subsection{Organization}
This paper is organized as follows. Section 2 presents the model settings, assumptions, preliminary results, and mathematical preliminaries. Section 3 introduces a novel contraction direction-estimation method. The workflow and theoretical results of our local contraction methods are included in Section 4, and the output manifold is analyzed in Section 5. Numerical studies are presented in Section 6, to demonstrate the effectiveness of our approach. Finally, Section 7 provides a summary of the key findings and conclusions of our study, as well as several directions for future research.

\section{Proposed method}
In this section, we present some necessary notations and fundamental concepts, then formally state our primary result regarding the fitting of a manifold. Finally, we introduce several lemmas and propositions crucial for further elaboration.
\subsection{Notations and important concepts}
\label{Section:Notation}
Throughout this paper, we use both upper- and lower-case $C$ to represent absolute constants. The distinction between upper and lower-case letters represents the magnitude of the constants, with the former being greater than one and the latter being less than one. The values of these constants may vary from line to line. In our notation, $x$ represents a point on the latent manifold $\cM$, $y$ represents a point related to the distribution $\nu$, and $z$ represents an arbitrary point in the ambient space. The symbol $r$ is used to denote the radius in some instances.
Capitalized math calligraphy letters, such as $\cM$, $\mathcal{Y}$, and $\mathcal{B}_D(z,r)$, represent concepts related to sets. This last symbol denotes a $D$-dimensional Euclidean ball with center $z$ and radius $r$.

The distance between a point $a$ and a set $\mathcal{A}$ is represented as $d(a,\mathcal{A}) = \min_{a^\prime\in\mathcal{A}} \|a-a^\prime\|_2$, where $\|\cdot\|_2$ is the Euclidean distance. To measure the distance between two sets, we adopt the \emph{Hausdorff distance}, a commonly used metric in evaluating the accuracy of estimators. This distance will be used to measure the distance between the latent manifold $\cM$ and its estimate $\widehat{\cM}$ throughout this paper. Formally, this metric can be defined as follows:
\begin{definition}
    [Hausdorff distance] Let $\mathcal A_1$ and $\mathcal A_2$ be two non-empty subsets of $\bR^D$. We define their Hausdorff distance $d_H(\mathcal A_1,\mathcal A_2)$ induced by Euclidean distance as
    $$
    d_H(\mathcal A_1,\mathcal A_2) = \max \{\sup_{a\in \mathcal A_1} \inf_{b\in \mathcal A_2}\|a-b\|_2,~\sup_{b\in \mathcal A_2} \inf_{a\in \mathcal A_1}\|a-b\|_2\}.
    $$
\end{definition}
\begin{remark}
    For any $\mathcal A_1,~\mathcal A_2 \subset \bR^D$,  $d_{H}(\mathcal A_1,\mathcal A_2)<\epsilon$ is equivalent to the fact that, for $\forall a\in \mathcal A_1$ and $\forall b\in \mathcal A_2$,
    $$
        d(a,\mathcal A_2) <\epsilon \text{ and } d(b,\mathcal A_1) <\epsilon.
    $$
\end{remark}

In the context of geometry, the Hausdorff distance provides a measure of the proximity between two manifolds. It is commonly acknowledged that a small Hausdorff distance implies a high level of alignment between the two manifolds, with controlled discrepancies.

We also require some basic geometrical concepts related to manifolds, more of which can be found in the supplementary material. Given a point $x$ in the manifold $\cM$, the tangent space at $x$, denoted by $T_x\cM$, is a $d$-dimensional affine space containing all the vectors tangent to $\cM$ at $x$. To facilitate our analysis, we introduce the projection matrices $\Pi_x^-$ and $\Pi_x^\perp$, which project any vector $v\in\mathbb{R}^D$ onto the tangent space $T_x\cM$ and its normal space, respectively. These two projection matrices are closely related as $\Pi_x^\perp = I_D - \Pi_x^-$, where $I_D$ is the identity mapping of $\mathbb{R}^D$. Furthermore, given an arbitrary point $z$ not in $\cM$, its projection onto the manifold is defined as $z^* = \arg\min_{x\in\cM} \|x-z\|_2$, and we use $\widehat{\Pi}_z^\perp$ and $\widehat{\Pi}_z^-$ as estimators for $\Pi_{z^*}^\perp$ and $\Pi_{z^*}^-$, respectively.

The concept of \emph{Reach}, first introduced by Federer \cite{federer1959curvature}, plays a crucial role in measuring the regularity of manifolds embedded in Euclidean space. Reach has proven to be a valuable tool in various applications, including signal processing and machine learning, making it an indispensable concept in the study of manifold models. It can be defined as follows:
\begin{definition}
    [Reach] Let $\mathcal A$ be a closed subset of $\bR^D$. The reach of $\mathcal A$, denoted by $\reach(\mathcal A)$, is the largest number $\tau$ to have the following property: any point at a distance less than $\tau$ from $\mathcal A$ has a unique nearest point in $\mathcal A$.
\end{definition}
\begin{remark}
    The value of $\reach(\cM)$ can be interpreted as a second-order differential quantity if $\cM$ is treated as a function. Namely, let $\gamma$ be an arc-length parameterized geodesic of $\cM$; then, according to \cite{niyogi2008finding}, $\|\gamma^{\prime\prime}(t)\|_2\leq \reach(\cM)^{-1}$ for all $t$.
\end{remark}
For example, the reach of a circle is its radius, and the reach of a linear subspace is infinite. Intuitively, a large reach implies that the manifold is locally close to the tangent space. This phenomenon can be explained by the following lemma in \cite{federer1959curvature}:
\begin{lemma}
    \label{Lemma:ReachCond}
    [Federer's reach condition] Let $\cM$ be an embedded sub-manifold of $\bR^{D}$ with reach $\tau$. Then,
    $$\tau^{-1} = \sup \left\{\frac{2d(b,T_a\cM)}{\|a-b\|_2^2}|a,b\in\cM,~a\neq b\right\}.$$
\end{lemma}

\subsection{Overview of the main results}
As stated in the introduction, the fundamental objective of this paper is to develop an estimator $\widehat{\cM}$ for the latent manifold $\cM$, using the available sample set $\mathcal{Y}$. To this end, we employ a two-step procedure for each $y \in \Gamma =\{y: d(y,\cM)\leq C\sigma\}$, involving (i) identification of the contraction direction and (ii) estimation of the contracted point. It should be noted that contraction is distinct from projection, as the former entails movement in a singular direction in normal space.

\subsubsection*{Determining the contraction direction}
To enhance the accuracy of our algorithm, we introduce a novel approach for estimating the direction of $y^*-y$ for each $y$, instead of estimating the basis of $T_{y^*}\cM$.

On the population level, consider a $D$-dimensional ball $\cB(y,r_0)$ with $r_0 = C\sigma$. Let 
$$\mu_y^\mathbb{B} = \bE_{Y\sim \nu}(Y|Y\in \cB_D(y, r_0)).$$
$\mu_y^\mathbb{B} - y$ estimates the direction of $y^* - y$ with an error upper bounded by $C\sigma\sqrt{\log(1/\sigma)}$ (Theorem \ref{Thm:AngleCondition}). To make the estimation continuous with respect to $\mathcal Y$ and $y$, we let
$$
    F(y) = \sum\alpha_i(y) y_i,
$$
with the weights $\alpha_i$'s given in Section \ref{Sec:F(z)}. When the total sample size $N = C_1r_0^{-d}\sigma^{-3}$, $F(y) - y$ estimates the direction of $y^* - y$ with an error upper bounded by $C\sigma\sqrt{\log(1/\sigma)}$ with probability no less than $1 - C_1\exp(-C_2\sigma^{-c})$ (Theorem \ref{Thm:Angle:F(z)}).

\subsubsection*{Estimating the contracted point} The estimation of the projection points is discussed in three distinct scenarios in Section 4, the most notable of which is using $F(y)$ to estimate the contraction direction.

Let $\widetilde{U}$ be the projection matrix onto the direction of $\mu_y^\mathbb{B}-y$. Consider a cylinder region
$$
    {\bV}_{y} = \cB_{D-1}(y,r_1) \times \cB_1(y,r_2),
$$
where the second ball is an open interval in the direction of $\mu_y^\mathbb{B} - y$, and the first ball is in the complement of it in $\bR^D$, with $r_1 = c\sigma$ and $r_2 = C\sigma\sqrt{\log(1/\sigma)}$. On the population level, let the contracted version of $y$ be denoted by
$$\mu_y^\bV = y + \widetilde{U} \bE_{Y\sim \nu}\left(Y-y|Y\in {\bV}_{y}\right);$$
then, $\|\mu_y^\bV - y^*\|_2 \leq C\sigma^2\sqrt{\log(1/\sigma)}$ (Theorem \ref{Thm:ContractWithEstDir2}). For the sake of continuity, we let
$$
    \widehat{U} = \frac{(F(y)-y)(F(y)-y)^T}{\|(F(y)-y)\|_2^2},
$$
and construct another smooth map
$$
G(y) = \sum\beta_i(y) y_i, 
$$
where the weights $\beta_i$'s are related to $\widehat{U}$; their definition can be found in Section \ref{Sec:G(z)}. Then, the distance between $G(y)$ and $y^*$ is upper bounded by $C\sigma^2{\log(1/\sigma)}$ with probability at least $1 - C_1\exp(-C_2\sigma^{-c})$ (Theorem \ref{Thm:appro G(z)}).

\subsubsection*{Constructing the manifold estimator}
In Section 5, we propose a variety of methods to construct the manifold estimator for various scenarios. We begin by considering the case where the distribution $\nu$ is known, and demonstrate that the set
$$\mathcal{S}=\{\mu_y^\bV:y\in\Gamma\}$$
has a Hausdorff distance of $\cO(\sigma^2\log(1/\sigma))$ to $\cM$ (Theorem \ref{Thm:out_manifold_1}).

Next, we use the sample set $\mathcal{Y}$ to obtain an estimated version,
$$\widehat{\mathcal{S}}=\{G(y):y\in\Gamma\},$$
which has an approximate Hausdorff distance to $\cM$ in the order of $\cO(\sigma^2\log(1/\sigma))$ with high probability (Theorem \ref{Thm:out_manifold_2}).

Finally, we consider the scenario in which there exists a $d$-dimensional preliminary estimation $\widetilde{\cM}$ that is $\cO(\sigma)$ close to $\cM$. In this case, we show that, with high probability, $G(\widetilde{\cM})$ is a $d$-dimensional manifold having an approximate Hausdorff error of $\cO(\sigma^2\log(1/\sigma))$ and a reach no less than $c\sigma\reach(\widetilde{\cM})$ (Theorem \ref{Thm:out_manifold_global_d}).

\subsection{Lemmas and propositions} In this subsection, we present some propositions and lemmas for reference. Their proofs are omitted from the main content and can be found in the supplementary material.

A notable phenomenon when analyzing the distribution in the vicinity of the manifold is the prevalence of quantities contingent upon $d$ rather than $D$. This phenomenon is particularly evident in the subsequent lemma and its corollary.
\begin{lemma}
\label{Lemma:prob_in_a_ball}
    For any arbitrary point $z$ such that its neighborhood
     $\cB_D(z,r) \cap \cM \neq \varnothing$ with $r = C\sqrt{2d}\sigma\sqrt{\log(1/\sigma)}$, the probability that $Y\sim \nu$ falls in $\cB_D(z,r)$ is 
    $$
        \bP(Y\in \cB_D(z,r)) = c r^d
    $$
    for some small constant $c$.
\end{lemma}
\begin{corollary}
\label{Col:Local_sample_size}
    Let $n$ be the number of observed points that fall in $\cB_D(z,r)$. Assume the total sample size is $N=CD\sigma^{-3}r^{-d}$. Then,
    $$\bP(C_1 D\sigma^{-3}\leq n \leq C_2 D\sigma^{-3})\geq 1 - 2\exp\left( -C_3\sigma^{-3}\right),$$
    for some constant $C_1$, $C_2$, and $C_3$.
\end{corollary}

Since the Gaussian distribution can be approximated to vanish within a few standard deviations ($\sigma$), adopting a radius that is marginally larger than $\sigma$ can result in polynomial benefits for local estimation. For instance, when computing the conditional expectation within a ball near the origin, we have the following proposition:

\begin{proposition}
    \label{Prop:TruncateNormalConcentration}
    Let $\xi$ be a $D$-dimensional normal random vector with mean $0$ and covariance matrix $\sigma^2 I_D$. Assume there is a $D$-dimensional ball $\cB_D(z,r)$ centered at point $z$ with radius $r = C_1\sigma\sqrt{\log(1/\sigma)}$, and $\|z\|_2 = C_2\sigma$. Then, the truncated version of $\xi$ satisfies
    $$\|\bE(\xi|\xi\in\cB_D(z,r))\|_2\leq C_3\sigma^2,$$
    for some constants $C_1$, $C_2$, and $C_3$.
\end{proposition}

Analogously, it is sufficient to focus on a subset of $\cM$ when studying certain local structures. For instance, in analyzing the conditional moments of $\nu$ within a $D$-dimensional ball $\cB_D(z,r)$, the submanifold $\cM_R = \cM \cap \cB_D(z,R)$ with $R\gg r$ exerts a significant influence. By incorporating $\cM_R$, $\nu$ can be approximated with
\begin{align*}
    \nu_R(y) = \int_{\cM_R} \phi_\sigma(y-x)\omega(x)dx.
\end{align*}
If we normalize them within $\cB_D(z,r)$, the two densities $\tilde{\nu}$ and $\tilde{\nu}_R$ should be close, and it is sufficient to work with $\tilde{\nu}_R(y)$ directly. These can be summarized as the following lemma:

\begin{lemma}
    \label{Lemma:PartManifold}
    Let $\tilde{\nu}(y)$ be the conditional density function within $\cB_D(z,r)$, and $\tilde{\nu}_R(y)$ be its estimator based on $\cM_R$. By setting 
    $$R = r + C_1\sigma\sqrt{(d+\eta)\log(1/\sigma)},$$
    we have
    \begin{equation}
        \label{eq:Part_Manifold}
        \left| \tilde{\nu}(y) - \tilde{\nu}_R(y) \right| \leq C_2\sigma^\eta \tilde{\nu}_R(y)
    \end{equation}
    for some constant $C_1$ and $C_2$.
\end{lemma}
\begin{corollary}
    If \eqref{eq:Part_Manifold} holds for all $y\in \cB_D(z,r)$, we have
    \begin{align*}
        \left\|\bE_{\tilde{\nu}} Y - \bE_{\tilde{\nu}_R} Y\right\|_2
        \leq C\sigma^\eta \int_{\cB_D(z,r)} \|y-z\|_2\tilde{\nu}_R(y)\,dy
        \leq Cr\sigma^\eta.
    \end{align*}
\end{corollary}

\section{Estimation of contraction direction}
\label{section:Direction_Estimation}
This section presents a novel method for estimating the contraction direction and provides an error bound. Our approach is underpinned by the fact that, in the denoising step, the goal is to ``push'' a point $z$, which is within a distance of $\Delta = C\sigma$ to $\cM$, toward its projection on $\cM$, i.e., $z^*$. Therefore, it is sufficient to estimate the direction of $z^* - z$ instead of estimating the entire basis of $T_{z^*}\cM$. To determine this direction, we focus on a ball $\cB_D(z, r_0)$ centered at $z$ with radius $r_0 = C\sigma$ and provide population-level and sample-level estimators.

\subsection{Population level}
Let the conditional expectation of $\nu$ within the ball be $\mu_z^\mathbb{B}$, namely
\begin{equation}
\label{eq:mu_z^B}
    \mu_z^\mathbb{B} = \bE_{Y\sim \nu}(Y|Y\in \cB_D(z, r_0)).
\end{equation}
The accuracy of the vector $\mu_z^\mathbb{B} - z$ in estimating the direction of $z^* - z$ is reported in Theorem \ref{Thm:AngleCondition}. The proof of this theorem is presented in the remainder of this subsection. This result demonstrates that the vector $\mu_z^\mathbb{B} - z$ performs well in estimating the contraction direction, providing further support for its use in the denoising step. The proof of lemmas and propositions is omitted here and can be found in the supplementary material.

\begin{theorem}
    \label{Thm:AngleCondition}
    For a point $z$ such that $d(z,\cM) = \cO(\sigma)$, we can estimate the direction of $z^* - z$ with $$\mu_z^\mathbb{B} - z = \bE_{Y\sim \nu}(Y - z|Y\in \cB_D(z, r_0)).$$
    The estimation error can be bounded as
    \begin{equation}
        \label{eq:AngleCondition}
        \sin \{\Theta\left(\mu_z^\mathbb{B} - z, \ z^*-z\right)\} \leq C\sigma\sqrt{\log(1/\sigma)}.
    \end{equation}
\end{theorem}

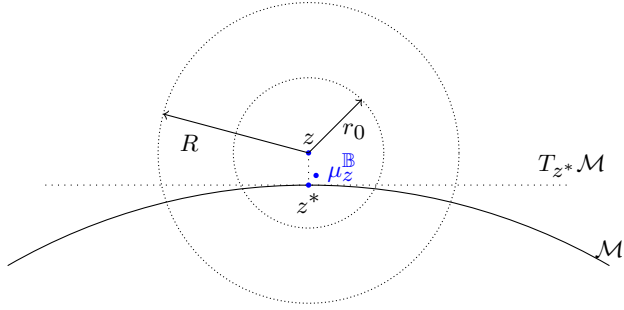
\begin{figure}[htbp]
    \begin{tikzpicture}    
        \draw (0 ,0) node[above] {$\cM$} arc (60:120:8);
        \draw [dotted] (-7.5,1.0718)--(-4,1.0718) node[below] {$z^*$}--(-0.5,1.0718)  node[above] {$T_{z^*}\cM$};
        \draw [dotted] (-4,1.5) node[above] {$z$} --(-4,1.0718);
        \fill [blue] (-4,1.0718) circle (1pt);
        \fill [blue] (-4,1.5) circle (1pt);
        \fill [blue] (-3.9,1.2) node[above = 3pt, right=1pt]{$\mu_z^\mathbb{B}$} circle (1pt);
        \draw [densely dotted] (-4,1.5) circle (1);
        \draw [densely dotted] (-4,1.5) circle (2);
        \draw [->] (-4,1.5) -- (-3.2929,2.2071) node[below = 12pt, left = -5pt]{$r_0$};
        \draw [->] (-4,1.5) -- (-5.9319,2.0176) node[below = 11pt, right = 3pt]{$R$};
    \end{tikzpicture}
    \caption{An illustration for estimating the contraction direction.}
\end{figure}

Without loss of generality, we assume that $z^*$ is the origin, $T_{z^*}\cM$ is the span of the first $d$ Cartesian-coordinate directions, $z^*-z$ is the $(d+1)$-th direction, and the remaining directions constitute the complement in $\bR^D$. To prove Theorem \ref{Thm:AngleCondition}, we first provide a sufficient statement for the error bound in \eqref{eq:AngleCondition}: 
\begin{proposition}    
    \label{Prop:Angle-Cord}
    Let $\mu_z^\mathbb{B} = (\mu^{(1)},\cdots,\mu^{(D)})$, to show \eqref{eq:AngleCondition}
    is sufficient to show 
    \begin{equation*}        
        \left\{
            \begin{array}{cll}
                |\Delta-\mu^{(i)}| &\geq c_1\sigma, \quad &\text{for } i = d+1;\\
                |\mu^{(i)}|   &\leq c_2 \sigma^2\sqrt{\log(1/\sigma)}, \quad &\text{for } i \neq d+1.\\
            \end{array}
        \right.
        \label{eq:Angle-Cord}
    \end{equation*}
\end{proposition}

To prove Proposition \ref{Prop:Angle-Cord}, we employ a strategy of locally approximating the manifold to the whole and using discs to approximate the local neighborhood of the manifold. Specifically, we use a disc $\bD = T_{z^*}\cM\cap\cB(z, R)$ to approximate $\cM_R$ and generalize the result to the entire manifold. The final error bound is achieved by combining the following lemmas:
\begin{lemma}    
    \label{Lemma:LocalGeneralDiff}
    Let $\tilde{\nu}_R(y)$ be the conditional density function within $\cB_D(z,r_0)$ induced by $\cM_R$, and $\tilde{\nu}_{\bD}(y)$ be its estimator with $\bD$. By setting $R = r_0 + C_1\sigma\sqrt{\log(1/\sigma)}$, we have
    $$\left| \tilde{\nu}_R(y) - \tilde{\nu}_{\bD}(y) \right| \leq C_2\sigma\sqrt{\log(1/\sigma)} \tilde{\nu}_{\bD}(y)$$
    for some constant $C_1$ and $C_2$.
\end{lemma}

\begin{lemma}
    \label{Lemma:LocalDisc}
    Let the conditional expectation of $Y\sim\tilde{\nu}_{\bD}$ within $\cB_D(z,r)$ be 
    $$\mu_{z,\bD}^\mathbb{B} = (\mu^{(1)}_\bD,\cdots,\mu^{(D)}_\bD).$$
    Then, there is 
    $$
    \left\{
        \begin{array}{cll}
            |\Delta - \mu_\bD^{(i)}| &\geq c\sigma \quad &\text{for } i = d+1\\
            |\mu_\bD^{(i)}|   &= 0 , \quad &\text{for } i \neq d+1\\
        \end{array}
    \right.,
    $$
\end{lemma}
%

According to Lemma \ref{Lemma:PartManifold} and Lemma \ref{Lemma:LocalGeneralDiff}, by setting $R = r_0 + C_1\sigma\sqrt{\log(1/\sigma)}$, we have
\begin{equation*}
    \left| \tilde{\nu}_R(z) - \tilde{\nu}_{\bD}(z) \right| \leq C_2\sigma\sqrt{\log(1/\sigma)} \tilde{\nu}_{\bD}(z),
\end{equation*}
and thus the conditional expectations within $\cB_D(z,r_0)$ should also be close, namely
\begin{equation*}
    \|\mu_{z,\bD}^\mathbb{B} - \mu_{z}^\mathbb{B}\|_2 \leq C \sigma^2\sqrt{\log(1/\sigma)},
\end{equation*}
for some constant $C$.

Therefore, together with Lemma \ref{Lemma:LocalDisc}, the statement in Proposition \ref{Prop:Angle-Cord} is fulfilled, and hence the proof of Theorem \ref{Thm:AngleCondition} is completed.

\subsection{Estimation with finite sample}
\label{Sec:F(z)}
In practice, we typically have access to only the data point collection $\mathcal{Y}$, which is sampled from the distribution $\nu(y)$. To construct an estimator for $\mu_z^{\mathbb{B}}$ as defined in \eqref{eq:mu_z^B}, a natural approach is to use the local average, defined as
$$
\tilde{\mu}_z^{\mathbb{B}} = \frac{1}{|I_z|}\sum_{i\in I_z} y_i,
$$
where $I_z$ is the index of $y_i$'s that lie in $\mathcal{Y} \cap\mathcal{B}_D(z,r_0)$.
Although $\tilde{\mu}_z^{\mathbb{B}}$ converges to $\mu_z^{\mathbb{B}}$ as the size of $I_z$ goes to infinity, it is not a continuous mapping of $y$ because of the discontinuity introduced by the change in the neighborhood. The discontinuity can adversely affect the smoothness of $\widehat{\cM}$. To address this issue, we need a smooth version of $\tilde{\mu}_z^{\mathbb{B}}$.

Let the local weighted average at point $z$ be
\begin{equation}
    \label{eq:def:F(z)}
    F(z) = \sum_{i}\alpha_i(z) y_i,
\end{equation}
with the weights being defined as
\begin{equation}
    \label{eq:def:weight_F}
\tilde{\alpha}_i(z)=
\left\{\begin{array}{cc}
\left(1 - \frac{\|z - y_i\|_2^2}{r_0^2}\right)^{k}, & \|z - y_i\|_2\leq r_0;\\
0, &  {\rm otherwise};\\ 
\end{array}\right.
\quad \tilde{\alpha}(z) =  \sum_{i\in I_z}\tilde{\alpha}_i(z),
\quad \alpha_i(z) = \frac{\tilde{\alpha}_i(z)}{\tilde{\alpha}(z)},
\end{equation}
with $k>2$ being a fixed integer guaranteeing a twice-differentiable smoothness. Similar to $\mu_z^\mathbb{B} - z$, the direction of $F(z) - z$ approximates the direction $z$ to $z^*$ well:
\begin{theorem}
\label{Thm:Angle:F(z)}
    If the sample size $N = C_1\sigma^{-(d+3)}$, for a point $z$ such that $d(z,\cM) = \cO(\sigma)$, $F(z)$ as defined in \eqref{eq:def:F(z)} provides an estimation of the contraction direction, whose error can be bounded by 
    $$\sin \{\Theta\left(F(z) - z, \ z^*-z\right)\} \leq C_2\sigma\sqrt{\log(1/\sigma)},$$
    with probability at least $1 - C_3\exp(-C_4\sigma^{-c})$, for some constant $c$, $C_1$, $C_2$, $C_3$, and $C_4$.
\end{theorem}

\section{Local contraction}
This section presents the theoretical results of the local contraction process.
Let $z$ be a point within a distance of $C\sigma$ to $\cM$, and let $\bV_{z}$ be a neighborhood of $z$.
The conditional expectation of $\nu$ within $\bV_{z}$ can be viewed as a denoised version of $z$, namely
$$\bE_{Y\sim \nu}\left(Y|Y\in \bV_{z}\right).$$
To minimize noise and avoid distortion by the manifold, $\bV_z$ should be narrow in the directions tangent to the manifold and broad in the direction perpendicular to it, like inserting a straw into a ball.
Thus, determining the orientation of $\bV_{z}$ and the scale in two directions is crucial.
In the following sub-sections, we analyze the population-level denoising result for three different orientation settings and provide a smooth estimator for the last case.

\subsection{Contraction with known projection direction}
In the simplest scenario, we assume the direction of $T_{z^*} \cM$, i.e., $\Pi^\perp_{z^*}$, is known. Then, $\bV_{z}$ can be constructed as the Cartesian product of two balls. Specifically,
\begin{equation}
\label{eq:V_z:1}
\begin{aligned}
    \bV_{z} 
    &= \cB_d(z,r_1) \times \cB_{D-d}(z,r_2)\\
    &=\Pi^-_{z^*}\cB_D(z,r_1) \times \Pi^\perp_{z^*}\cB_D(z,r_2),
\end{aligned}
\end{equation}
where the first ball is $d$-dimensional, lying in $\bR^{d} = T_{z^*}\cM$, while the second one is in the orthogonal complement of $\bR^{d}$ in $\bR^{D}$ with a radius $r_2\gg r_1$. Let $\mu_z^{\mathbb V}$ be the denoised point, calculated with the conditional expectation within $\bV_{z}$; precisely,
\begin{equation}
\label{eq:mu_z:1}
    \mu_z^{\mathbb V} = z + \Pi^\perp_{z^*}\bE_{Y\sim \nu}\left(Y-z|Y\in \bV_{z}\right),
\end{equation}
where $Y$ is a random vector with density function $\nu(y)$. The refined point $\mu_z^{\mathbb V}$ is much closer to $\cM$. This result can be summarized as the following theorem:

\begin{theorem}
    \label{Thm:ContractWithKnownDir}
    Consider a point $z$ such that $d(z,\cM)<C \sigma$. Let its neighborhood $\bV_{z}$ be defined as \eqref{eq:V_z:1} with radius 
    $$r_1 = c\sigma \quad \text{and} \quad r_2 = C\sigma\sqrt{\log(1/\sigma)}.$$ The refined point $\mu_z^{\mathbb V}$ given by \eqref{eq:mu_z:1} satisfies
    $$d(\mu_z^{\mathbb V},\cM)\leq C\sigma^2\log(1/\sigma),$$
    for some constant $C$.
\end{theorem}
\begin{proof}
    Recall that $Y = X + \xi$ in our model setting, and $\Pi^-_{z^*}$ is the orthogonal projection onto $T_{z^*}\cM$. If we analogously write $z$ as 
    $$z = z^* + (z-z^*) := z^* +\delta_z, $$
    $\mu_z^{\mathbb V}$ in \eqref{eq:mu_z:1} can be decomposed as 
    \begin{equation}
    \label{eq:decompose_mu_z}
        \begin{aligned}
            \mu_z^{\mathbb V} &= z + \Pi^\perp_{z^*}\bE_{Y\sim \nu}\left(Y-z|Y\in \bV_{z}\right)\\
            &= z^* + \delta_z + \Pi^\perp_{z^*} \bE_\nu\left((X + \xi) - (z^* + \delta_z)|Y\in \bV_{z}\right)\\
            &= z^* + \Pi^-_{z^*}\delta_z + \bE_\nu\left(\Pi^\perp_{z^*}(X - z^*)|Y\in \bV_{z} \right) + \bE_\nu\left(\Pi^\perp_{z^*}\xi|Y\in \bV_{z} \right).
        \end{aligned}
    \end{equation}
    With such an expression, $\mu_z^{\mathbb V} - z^*$ can be decomposed into three terms. The next step is to show that the norms of these terms are upper bounded by $\cO(\sigma^2\log(1/\sigma))$. According to Lemma \ref{Lemma:PartManifold}, to get a bound in the order of $\cO(\sigma^2\log(1/\sigma))$, we only need to consider a local part of $\cM$, i.e., $\cM_R$ with $R = C\sigma\sqrt{\log(1/\sigma)}$, and thus it is safe to assume $\|X- z\|_2 \leq C\sigma\sqrt{\log(1/\sigma)}$ for some constant $C$.
    
    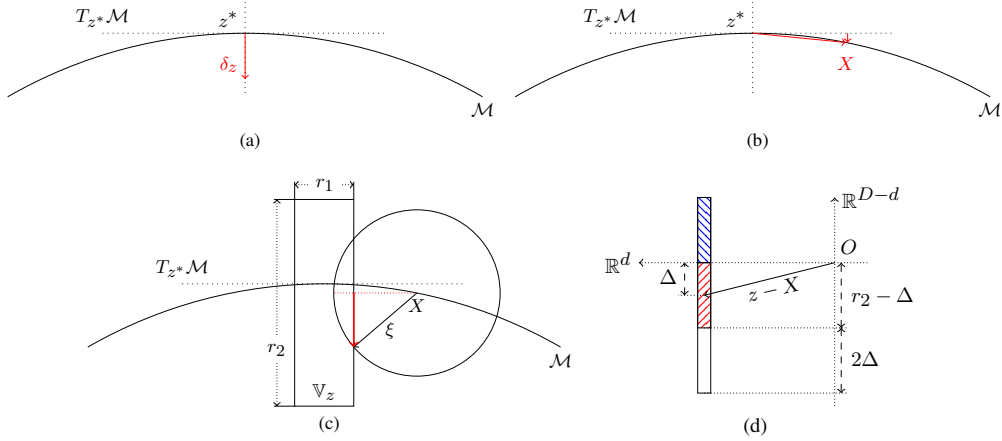
\begin{figure}[htbp]
        \centering
        \resizebox{.95\textwidth}{!}{
            \subfigure[]{
            \begin{tikzpicture}    
                \draw (0 ,0) node[below] {$\cM$} arc (60:120:8);
                \draw [dotted] (-6.4,1.0718) node[above] {$T_{z^*}\cM$}--(-4,1.0718) node[above left] {$z^*$}--(-1.6,1.0718);
                \draw [dotted] (-4,1.6) --(-4,0);
                \draw [red, ->] (-4,1.0718) -- (-4,0.3) node[above left] {$\delta_z$};
            \end{tikzpicture}
            }
            \subfigure[]{
            \begin{tikzpicture}    
                \draw (0 ,0) node[below] {$\cM$} arc (60:120:8);
                \draw [dotted] (-6.4,1.0718) node[above] {$T_{z^*}\cM$}--(-4,1.0718) node[above left] {$z^*$}--(-1.6,1.0718);
                \draw [dotted] (-4,1.6) --(-4,0);
                \draw[red, ->] (-4,1.0718) --(-2.431,0.9164) node[below = 4pt] {$X$};
                \draw[red, ->] (-2.4,1.0718) --(-2.4,0.9164);
            \end{tikzpicture}
            }
        }
        \resizebox{.475\textwidth}{!}{    
            \subfigure[]{
                \begin{tikzpicture}    
                    \draw (0 ,0) node[below] {$\cM$} arc (60:120:8);
                    \draw [dotted] (-6.4,1.0718) node[above] {$T_{z^*}\cM$}--(-4,1.0718) --(-1.6,1.0718);
                    \draw [dotted] (-4,-1) node[above] {$\bV_{z}$};
                    \draw (-4.5,-1) rectangle (-3.5,2.5);
                    \draw (-2.431,0.9164) node[below] {$X$} circle (1.4080);        
                    \draw[red,densely dotted] (-2.431,0.9164) --(-3.839,0.9164);
                    \draw[->] (-2.431,0.9164) --(-3.5,-0) node[above = 8pt, right = 12pt] {$\xi$};
                    \draw[densely dotted, <->] (-4.5,2.75) --(-3.5,2.75) node[left = 6pt,fill=white] {$r_1$};
                    \draw[densely dotted, <->] (-4.8,2.5) --(-4.8,-1) node[above = 21pt, fill=white] {$r_2$};
                    \draw[red, ->,thick] (-3.5,0.914) -- (-3.5,0);
                    \draw (-4.5,2.5) --(-4.85,2.5);        
                    \draw (-4.5,-1) --(-4.85,-1);     
                    \draw (-4.5,2.5) --(-4.5,2.8);
                    \draw (-3.5,2.5) --(-3.5,2.8);
                \end{tikzpicture}
                }
        }\resizebox{.32\textwidth}{!}{ 
                \subfigure[]{
                \begin{tikzpicture}    
                    \draw[densely dotted, ->] (0.5,0) node[above=7pt,right = -15pt]{$O$} -- (-3,0) node[left]{$\bR^{d}$};
                    \draw[densely dotted, ->] (0,-2.2) -- (0,1) node[right]{$\bR^{D-d}$};
                    \draw (-1.9,-2) rectangle (-2.1,1);
                    \draw[pattern=north west lines, pattern color=blue](-2.1,0) rectangle (-1.9,1);
                    \draw[pattern=north east lines, pattern color=red](-2.1,0) rectangle (-1.9,-1);
                    \draw[->] (0,0) -- (-2,-0.5) node[right=32pt, above = -4pt,rotate=13]{$z-X$};
                    \draw[densely dotted] (-2,-0.5) -- (-2.4,-0.5);        
                    \draw[dashed, <->] (-2.3,-0.5) --(-2.3,0) node[below left] {$\Delta$};
                    \draw[densely dotted] (-2,-1) -- (0.2,-1);
                    \draw[dashed, <->] (0.1,-1) --(0.1,0) node[below=15pt,right=1pt] {$r_2 - \Delta$};
                    \draw[densely dotted] (-2,-2) -- (0.2,-2);
                    \draw[dashed, <->] (0.1,-1) --(0.1,-2) node[above=15pt,right=1pt] {$2\Delta$};
                \end{tikzpicture}
                }
            }
        \caption{Illustration for the three parts of the error bound in \eqref{eq:decompose_mu_z}. (a) $\delta_z$, perpendicular to $T_{z^*}\cM$; (b) Projection of $X- z^*$, in a higher order than the length of $X- z^*$; (c, d) Projection of noise term, in two Cartesian-coordinate systems. A large area is canceled out because of symmetry.}
        \label{Fig:Illust3Parts}
    \end{figure}
    
    \begin{itemize}
        \item [(a)]$\Pi^-_{z^*}\delta_z$:
        
        \vspace{5pt}
        
        As $\delta_z \perp T_z\cM$, we have
        \begin{equation}
        \label{eq:Proof_Known_Dir_a}
            \Pi^-_{z^*}\delta_z=0.
        \end{equation}

        \item [(b)]$\bE_\nu\left(\Pi^\perp_{z^*}(X - z^*)|Y\in \bV_{z} \right)$:
        
        \vspace{5pt}
        
        Since $z^*$ and $X$ are exactly on $\cM$, from Jensen's inequality and Lemma \ref{Lemma:ReachCond} we have
        \begin{align*}
            \left\|\bE_\nu\left(\Pi^\perp_{z^*}(X - z^*)|Y\in \bV_{z} \right)\right\|_2
            &\leq \bE_\nu\left(\left\|\Pi^\perp_{z^*}(X - z^*)\right\|_2|Y\in \bV_{z} \right)\\
            &\leq \frac{1}{2\tau}\bE_\nu\left(\|X - z^*\|_2^2|Y\in \bV_{z} \right),
        \end{align*}
        where 
        \begin{align*}
            \|X - z^*\|_2^2 &= \|X - z + z - z^*\|_2^2 \\
            &\leq \|X-z\|_2^2 + \|z-z^*\|_2^2 \\
            &\leq C\sigma^2\log(1/\sigma).
        \end{align*}
        
        Hence,
        \begin{equation}
        \label{eq:Proof_Known_Dir_b}
            \left\|\bE_\nu\left(\Pi^\perp_{z^*}(X - z^*)|Y\in \bV_{z} \right)\right\|_2\leq \frac{C}{\tau}\sigma^2\log(1/\sigma).
        \end{equation}

        \item [(c)]$\bE_\nu\left(\Pi^\perp_{z^*}\xi|Y\in \bV_{z} \right)$:
        
        \vspace{5pt}
        
        Because 
        $$\bE_\nu\left(\Pi^\perp_{z^*}\xi|Y\in \bV_{z} \right)= \bE_\omega\left( \bE_{\phi}\left(\Pi^\perp_{z^*}\xi|X,~X+\xi\in \bV_{z} \right) \right),$$
        we evaluate the inner part $\bE_{\phi}\left(\Pi^\perp_{z^*}\xi|X,~X+\xi\in \bV_{z} \right)$ first.
        Assume the origin is transferred to $X$ as illustrated in Fig. \ref{Fig:Illust3Parts}(d). Now, 
        $$\bV_{z} =\cB_d(\Pi^-_{z^*} (z - X),r_1) \times \cB_{D-d}(\Pi^\perp_{z^*} (z - X),r_2)$$
        and there is a dislocation $\Delta = \|\Pi^\perp_{z^*} (z - X)\|_2$ in $\bR^{D-d}$, which is bounded by
        \begin{align*}
            \Delta \leq \|\Pi^\perp_{z^*} (z - z^*)\|_2 + \|\Pi^\perp_{z^*}( z^* - X) \|_2 \leq C\sigma.
        \end{align*}
        Let $\xi^\prime = \Pi^\perp_{z^*}\xi$; then, according to Proposition \ref{Prop:TruncateNormalConcentration}, we have
        \begin{equation}
        \label{eq:Proof_Known_Dir_c}
            \left\|\bE_{\phi}\left(\Pi^\perp_{z^*}\xi|X,~X+\xi\in \bV_{z} \right)\right\|_2
            = \|\bE\left(\xi^\prime|\xi^\prime \in \cB_{D-d}(a_\Delta,r_2)\right)\|_2
            \leq C\sigma^2,
        \end{equation}
        where $a_\Delta$ is the projection of $z-X$ onto $\bR^{D-d}$.
            
        \end{itemize}
    
        Combining the above result in  \eqref{eq:Proof_Known_Dir_a}, \eqref{eq:Proof_Known_Dir_b}, and \eqref{eq:Proof_Known_Dir_c}, for any $z$ such that $d(z,\cM)<C \sigma$, the corresponding $\mu_z^{\mathbb V}$ satisfies
        $$\|\mu_z^{\mathbb V} - z^*\|_2 \leq C \sigma^2 \log(1/\sigma),$$    
        for some constant $C$. Thus, the revised point $\mu_z^{\mathbb V}$ is $\cO(\sigma^2\log(1/\sigma))$-close to $\cM$.
    \end{proof}

\subsection{Contraction with estimated projection direction}
Usually, the projection matrix is unknown, but it can be estimated via many statistical methods. 
Assume $\widehat\Pi_{z^*}^\perp$ is an estimator for $\Pi_{z^*}^\perp$, whose bias is
\begin{equation}
\label{eq:dir_error1}
    \left\|\widehat\Pi_{z^*}^\perp - \Pi_{z^*}^\perp\right\|_F\leq c \sigma^\kappa.
\end{equation}
Based on this estimation, a similar region $\widehat{\bV}_{z}$ can be defined as
\begin{equation}
    \label{eq:V_z:2}
    \widehat{\bV}_{z} = \cB_d(z,r_1) \times \cB_{D-d}(z,r_2),
\end{equation}
where the first ball $\cB_d(z,r_1)$ is in the span space of $\widehat\Pi_{z^*}^-$ with a radius $r_1 = c\sigma$, and the second one is in the span space of $\widehat\Pi_{z^*}^\perp$ with a radius $r_2 = C\sigma\sqrt{\log(1/\sigma)}$. Then, an estimated version of $\mu_z^{\mathbb V}$ can be obtained:
\begin{equation}
    \label{eq:mu_z:2}
    \widehat \mu_z^{\mathbb V} = z + \widehat\Pi_{z^*}^\perp\bE_{Y\sim \nu}\left(Y-z|Y\in \widehat{\bV}_{z}\right),
\end{equation}
which is still closer to $\cM$. The error bound can be summarized as the following theorem:

\begin{theorem}    
    \label{Thm:ContractWithEstDir1}
   Consider a point $z$ such that $d(z,\cM)<C \sigma$. Let its neighborhood $\widehat{\bV}_{z}$ be defined as in \eqref{eq:V_z:2}, and the estimation error of $\widehat\Pi_{z^*}^\perp$ be bounded as in \eqref{eq:dir_error1}.
   The refined point $\widehat \mu_z^{\mathbb V}$ given by \eqref{eq:mu_z:2} satisfies
    $$d(\widehat\mu_z^{\mathbb V},\cM)\leq C\sigma^{1+\kappa}\sqrt{\log(1/\sigma)},$$
    for some constant $C$.
\end{theorem}

Such an estimator $\widehat\Pi_{z^*}^\perp$ can be obtained via classical dimension-reduction methods such as local PCA. Here we cite an error bound of local PCA estimators and implement the result of Theorem \ref{Thm:ContractWithEstDir1} in the following remark.

\begin{lemma}[Theorem 2.1 in \cite{yao2019manifold}]
    For a point $z$ such that $d(z,\cM)<C \sigma$, let $\widehat\Pi_{z}^\perp$ be the estimator of $\Pi_{z^*}^\perp$, obtained via local PCA with $r = C\sqrt{\sigma}$. The difference between $\widehat\Pi_{z}^\perp$ and $\Pi_{z^*}^\perp$ is bounded by 
    $$
    \|\widehat\Pi_{z}^\perp - \Pi_{z^*}^\perp\|_F \leq C\frac{r}{\tau}
    $$
    with high probability.
\end{lemma}

\begin{remark} With the PCA estimator $\widehat\Pi_{z^*}^\perp$ mentioned above, the distance between $\widehat \mu_z^{\mathbb V}$ and $\cM$ is bounded by
$$d(\widehat\mu_z^{\mathbb V},\cM)\leq C\sigma^{3/2}\sqrt{\log(1/\sigma)}$$
with high probability.
\end{remark}

\subsection{Contraction with estimated contraction direction}
\label{Sec:G(z)}
In the previous two cases, we attempted to move $z$ closer to $z^*$ in the direction of $\Pi_{z^*}^\perp$.
However, instead of estimating the entire projection matrix, finding an estimator in the main direction is sufficient and can be more accurate. Specifically, let the projection matrix onto $z^* - z$ be
$$U = {(z^* - z)(z^* - z)^T}/{\|z^* - z\|_2^2},$$
and, according to the discussion in Section \ref{section:Direction_Estimation}, there is one estimator 
$$\widetilde{U} = {(\mu_z^\bB - z)(\mu_z^\bB - z)^T}/{\|\mu_z^\bB - z\|_2^2},$$
whose error bound of $\widetilde{U}$ satisfies
\begin{equation}
\label{eq:dir_error2}
    \|\widetilde{U} - U\|_F \leq C\sigma\sqrt{\log(1/\sigma)}.
\end{equation}
A narrow region can be analogously constructed based on $\widetilde{U}$, namely
\begin{equation}
    \label{eq:V_z:3}
    \widehat{\bV}_{z} = \cB_{D-1}(z,r_1) \times \cB_1(z,r_2),
\end{equation}
where the second ball is actually an interval in the direction of $\widetilde{U}$ with $r_2 = C\sigma\sqrt{\log(1/\sigma)}$, and the first ball is in the span space of the complement of $\widetilde{U}$ in $\bR^D$ with $r_1 = c\sigma$. Similarly, $y$ can be refined by
\begin{equation}
    \label{eq:mu_z:3}
    \widehat \mu_z^\bV = z + \widetilde{U} \bE_{\nu}\left(Y-z|Y\in \widehat{\bV}_{z}\right),
\end{equation}
whose distance to $\cM$ can be bounded with the following theorem.

\begin{theorem}
    \label{Thm:ContractWithEstDir2}
     Consider a point $z$ such that $d(z,\cM)<C \sigma$. Let its neighborhood $\widehat{\bV}_{z}$ be defined as in \eqref{eq:V_z:3}, and the estimation error of $\widetilde{U}$ be bounded as in \eqref{eq:dir_error2}. The refined point $\widehat \mu_z^\bV $ given by \eqref{eq:mu_z:3} satisfies
    $$\|\widehat \mu_z^\bV-z^*\|_2\leq C\sigma^{2}\log(1/\sigma)$$
    for some constant $C$.
\end{theorem}
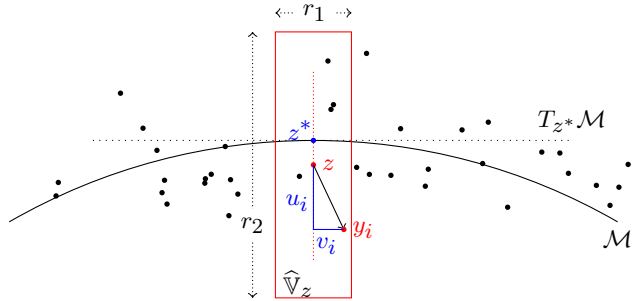
\begin{figure}[htbp]
        \centering
        \resizebox{0.6\textwidth}{!}{
        \begin{tikzpicture}
            \draw (0 ,0) node[below] {$\mathcal{M}$} arc (60:120:8);
            \draw [dotted] (-6.9,1.0718) --(-4,1.0718) --(-0.6,1.0718) node[above] {$T_{z^*}\mathcal{M}$};
            \draw (-4.2,-1.1) node[above] {$\widehat{\bV}_z$};
            \draw [red] (-4.5,-1) rectangle (-3.5,2.5);
            \fill [red] (-4,0.75) node[right]{$z$} circle (1pt);
            \draw[densely dotted, <->] (-4.5,2.75) --(-3.5,2.75) node[left = 6pt,fill=white] {$r_1$};
            \draw[densely dotted, <->] (-4.8,2.5) --(-4.8,-1) node[above = 21pt, fill=white] {$r_2$};
            
            \fill (-2.9811,0.6126) circle (1pt); \fill (0.1380,0.7562) circle (1pt); \fill (-5.1115,0.0816) circle (1pt); \fill (-5.0900,0.5550) circle (1pt);
            \fill (-1.4464,0.1926) circle (1pt); \fill (-1.7007,1.3119) circle (1pt); \fill (-5.9252,0.2342) circle (1pt); \fill (-4.9937,0.3653) circle (1pt);
            \fill (-0.1001,0.2172) circle (1pt); \fill (-0.7597,0.9058) circle (1pt); \fill (-3.8094,2.1171) circle (1pt); \fill (-6.0577,0.9415) circle (1pt);
            \fill (-4.1842,0.5993) circle (1pt); \fill (-1.7753,0.7675) circle (1pt); \fill (-0.9972,0.9182) circle (1pt); \fill (-6.5370,1.6932) circle (1pt);
            \fill (-3.7693,1.4801) circle (1pt); \fill (-5.3435,0.6265) circle (1pt); \fill (-3.4319,0.7174) circle (1pt); \fill (-5.9580,0.5476) circle (1pt);
            \fill (-2.8541,1.2264) circle (1pt); \fill (-0.2795,0.4803) circle (1pt); \fill (-0.6471,0.6389) circle (1pt); \fill (-3.2643,0.6250) circle (1pt);
            \fill (-2.4899,0.6930) circle (1pt); \fill (-2.0460,1.2104) circle (1pt); \fill (-3.3017,2.2170) circle (1pt); \fill (-2.5288,0.4695) circle (1pt);
            \fill (-5.1595,0.9918) circle (1pt); \fill (-7.3842,0.3415) circle (1pt); \fill (0.0170,0.4563) circle (1pt);
            \fill (-7.3544,0.5188) circle (1pt); \fill (-5.9907,0.3842) circle (1pt); \fill (-5.4281,0.5091) circle (1pt); \fill (-6.2409,1.2324) circle (1pt);
            \fill (-3.7367,1.5424) circle (1pt); \fill (-5.4203,0.5675) circle (1pt);
            
            \draw [red,densely dotted] (-4,-0.5) -- (-4,2);
            
            \draw [->] (-4,0.75) -- (-3.6,-0.1);
            \fill [red] (-3.6,-0.1) node[right]{$y_i$} circle (1pt);
            \draw [blue] (-3.6,-0.1) -- (-4,-0.1) node[right=5pt,below=0pt] {$v_i$} -- (-4,0.75) node[left=6pt,below=8pt] {$u_i$};
            \fill [blue] (-4,1.0718) node[above=4pt, left=-3pt]{$z^*$} circle (1pt);
        \end{tikzpicture}
        }
        \caption{Geometrical interpretation of $u_i$ and $v_i$ defined in \eqref{eq:def_u_i_v_i}: decomposing $y_i-z$ into its components; $u_i$ denotes the projection along $F(z)-z$, and $v_i$ represents the orthogonal component.}
        \label{Fig:Interpret_uv}
\end{figure}

For reasons similar to those discussed in Section \ref{Sec:F(z)}, a smooth estimator constructed with finite samples is needed. Recall that the continuous estimator for $U$ is
$$\widehat{U} = \frac{(F(z)-z)(F(z)-z)^T}{\|(F(z)-z)\|_2^2},$$
whose asymptomatic property is given in Theorem \ref{Thm:Angle:F(z)}. For a data point $y_i$, we define
\begin{equation}
    \label{eq:def_u_i_v_i}
    u_i = \widehat{U}(y_i-z),\quad v_i = y_i - z - u_i,
\end{equation}
which can be interpreted as the illustration in Fig. \ref{Fig:Interpret_uv}.
Let the contracted point of $z$ be
\begin{equation}
    \label{eq:def:G(z)}
    G(z) = \sum_{i}\beta_i(z) y_i,  
\end{equation}
with the weights given by
\begin{equation}\label{eq:def:weight_G}
    \begin{aligned}
        w_u(u_i) &= \left\{
        \begin{array}{cl}
            1, & \|u_i\|_2\leq \frac{r_2}{2} \\
            \left(1 - (\frac{2\|u_i\|_2-r_2}{r_2})^2\right)^k, & \|u_i\|_2\in (\frac{r_2}{2},r_2)\\
            0, & otherwise 
        \end{array}\right.,\\
        w_v(v_i) &= \left\{
        \begin{array}{cl}
            \left(1 - \frac{\|v_i\|_2^2}{r_1^2}\right)^k, & \|v_i\|_2\leq r_1 \\
            0, & otherwise 
        \end{array}\right.,\\
        \beta_i(z) = w_u(u_i)&w_v(v_i),
        \quad \tilde{\beta}(z) =  \sum\tilde{\beta_i}(z),
        \quad \beta_i(z) = \frac{\tilde{\beta_i}(y)}{\tilde{\beta}(z)},
    \end{aligned}
\end{equation}
with $k\geq 2$ being a fixed integer. It is clear that $G$ is a $C^2$-continuous map from $\mathbb{R}^D$ to $\mathbb{R}^D$. The estimation accuracy of $G(z)$ is summarized in the following theorem:
\begin{theorem}\label{Thm:appro G(z)}
    If the sample size $N = C_1\sigma^{-(d+3)}$, for a point $z$ such that $d(z,\cM) = \cO(\sigma)$, $G(z)$, as defined in \eqref{eq:def:G(z)}, provides an estimation of $z^*$, whose error can be bounded by 
    $$\|G(z)-z^*\|_2 \leq C_2\sigma^{2}{\log(1/\sigma)}$$
    with probability at least $1 - C_3\exp(-C_4\sigma^{-c})$, for some constant $c$, $C_1$, $C_2$, $C_3$, and $C_4$.
\end{theorem}

\section{Fit a smooth manifold}\label{sec:fit_mfd}
Up to this point, we have explicated the techniques for estimating the contraction direction and executing the contraction process for points proximal to $\cM$. In this section, we synthesize these two procedures to yield the ultimate smooth manifold estimator. The estimator is predicated upon a tubular neighborhood of $\cM$, denoted by $\Gamma = \{y: d(y, \cM) \leq C\sigma\}$, and manifests in two distinct incarnations, corresponding to the population and sample levels.

On the population level, we assume the distribution $\nu(y)$ is known, so that we can calculate all the expectations.
As mentioned in the introduction, estimating $\omega$ or $\cM$ with a known density function in the form of $\nu = \omega * \phi_{\sigma}$ is closely related to the singular deconvolution problem discussed in \cite{genovese2012manifold}. In contrast to their approach, our method uses geometrical structures to generate an estimate in the form of an image set, yielding a similar error bound. Formally, we have:
\begin{theorem}
    \label{Thm:out_manifold_1}
    Assume the density function $\nu(y)$ and region of interest $\Gamma$ are given. With the $\widehat{\mu}_y^{\mathbb{V}}$ defined in \eqref{eq:mu_z:3}, we let 
    $$\mathcal{S} = \{\widehat \mu_y^{\mathbb{V}}: ~ y\in\Gamma\}.$$
    Then, we have
    $$d_H(\mathcal{S},\cM)\leq C\sigma^2\log(1/\sigma)$$
    for some constant $C$.
\end{theorem}

When only the sample set $\mathcal{Y}$ is available, the function $G(y)$, as defined in \eqref{eq:def:G(z)}, can be used as an estimator of $\widehat{\mu}_y^{\mathbb{V}}$. First, $G(y)$ provides a good estimate of $y^*$ with high probability. Additionally, by definition, $G(\cdot)$ is a $C^2$-continuous mapping in $\mathbb{R}^D$. Hence, similar to the population case, the image set of $\Gamma$ under the mapping $G$ also has a good approximation property. Moreover, because of the smoothness of both $G$ and $\Gamma$, the output we obtain is also a smooth manifold.
Specifically, we have the following theorem:
\begin{theorem}
    \label{Thm:out_manifold_2}
    Assume the region $\Gamma$ is given. With the $G(y)$ defined in \eqref{eq:def:G(z)}, we let
    \begin{align}\label{output:G_Gamma}
    \widehat{\mathcal{S}} = G(\Gamma) = \{G(y): ~ y\in\Gamma\}.         
    \end{align}
    Then, $\widehat{\mathcal{S}}$ is a smooth sub-manifold in $\mathbb{R}^D$, and the following claims simultaneously hold for some constant $C$ with high probability:
    \begin{itemize}
        \item For any $x\in\cM$, $d(x,\widehat{\mathcal{S}})\leq C\sigma^2\log(1/\sigma)$;
        \item For any $s\in\widehat{\mathcal{S}}$, $d(x,\cM)\leq C\sigma^2\log(1/\sigma)$.
    \end{itemize}
    
\end{theorem}

The output manifold $\widehat{\mathcal{S}}$ furnishes a narrow tubular neighborhood of $\mathcal{M}$. By disregarding any anomalous points situated within a low-probability regime, we establish that the Hausdorff distance separating $\widehat{\mathcal{S}}$ and $\mathcal{M}$ scales as $\cO(\sigma^2\log(1/\sigma))$.
To further refine the intrinsic dimension of the manifold estimator to $d$, we introduce a partial solution in Theorem \ref{Thm:out_manifold_part_d} and a global solution in Theorem \ref{Thm:out_manifold_global_d}.

\begin{theorem}
    \label{Thm:out_manifold_part_d}
    For $x\in \mathcal{M}$, let $\widehat{\Pi}_x$ be the estimation of ${\Pi}_x$ as the one defined in \eqref{fy:yao2019}. Then there exists a constant $c>0$ such that
    $$
    \widehat{\mathcal{M}}_x =  \{y \in \Gamma \cap \cB_D(x,c\tau): \widehat{\Pi}_x^{\perp}(G(y) - y) = 0\}  
    $$
    is a $d$-dimensional manifold embedded in $\bR^D$. Meanwhile, for any point $y \in \widehat{\mathcal{M}}_x$,
    $$d(y, \mathcal{M}) \leq C\sigma^2\log(1/\sigma)$$
    for some constant $C$ with high probability.
\end{theorem}


Theorem \ref{Thm:out_manifold_part_d} provides a local solution, by guaranteeing that the function $\widehat{\Pi}_x^{\perp}(G(y) - y)$ has a constant rank $D-d$ through predetermined regions of interest and a fixed projection matrix. The resulting estimator is a piecewise $d$-dimensional manifold, which is more natural and smooth, but requires further manipulations to integrate the piecewise manifolds into an entirely smooth one. To avoid these manipulations, we assume there is a smooth initial manifold $\widetilde{\mathcal{M}}$ contained by $\Gamma$. Additionally, since $G$ is a $C^2$ continuous mapping in $\Gamma$,  we can assume that the Jacobi matrix of $G$ is bounded by $L_G$ and $\ell_G$, and the Hessian matrix of $G$ is bounded by $M_G$. Then a global estimator $\widetilde{\cM}$ can be obtained via the following theorem:

\begin{theorem}
    \label{Thm:out_manifold_global_d}
    Let $\widetilde{\cM} \subset \Gamma$ be a $d$-dimensional manifold with a positive reach $\tau_0$. Suppose that for each point $x\in \cM$, there exists a point $y$ such that $y^* = x$. Then, the estimator defined by 
    $
    \widehat{\cM} = G(\widetilde{\cM}) 
    $
    is also a $d$-dimensional manifold with the following conditions holding for some constant $c$ and $C$ with high probability: 
    \begin{itemize}
        \item[(I).] For any point $y \in \widehat{\cM}$, $d(y, \cM)$ is less than $C\sigma^2\log(1/\sigma)$;
        \item[(II).] For any point $x \in \cM$, $d(x, \widehat{\cM})$ is less than $C\sigma^2\log(1/\sigma)$;
        \item[(III).] The reach of $\widehat{\cM}$ is larger than a constant $\widehat{\tau} = \min\left\{c\sigma\tau_0, \  \frac{c\ell_G}{M_G + L_G}\right\}$.
    \end{itemize}
\end{theorem}



Notably, the estimator defined in Theorem \ref{Thm:out_manifold_global_d} requires an initial estimate $\widetilde{\cM}$, which can be obtained using the methods proposed in \cite{mohammed2017manifold, fefferman2018fitting, yao2019manifold,fefferman2021fitting}. In this paper, we also provide a defined strategy for reference.
\begin{proposition}
    \label{Prop:M_in_d_dim}
    Let $\widetilde{\cM}$ be a level set such that
    \begin{equation*}
        \widetilde{\cM} = \{y\in\Gamma:\Pi^* (F(y) - y)=0\},
    \end{equation*}
    where $\Pi^*$ is any arbitrary fixed projection matrix with rank $D-d$. Then, with high probability,
     $\widetilde{\cM}$ is a $d$-dimensional submanifold embedded in $\Gamma$, and $d_H(\widetilde{\cM},\cM)\leq C\sigma$.
\end{proposition}

In summary, we present two manifold estimators in the form of image sets and one in the form of level set, all satisfying the Hausdorff-distance condition under certain statistical conditions. Among them, the estimator proposed in Theorem \ref{Thm:out_manifold_2} is computationally simpler and more suitable for scenarios involving sample points, while the other estimators offer stronger theoretical guarantees for the geometric properties. As discussed in the introduction, prior works often employed level sets as manifold estimators, despite their inherent limitations: the existence of solutions to $f(x) = \bm{0}$, where $f(x)$ maps from $\mathbb{R}^D$ to $\mathbb{R}^D$, is not always evident. Thus the nonemptiness of the level sets is uncertain, requiring additional scrutiny. Furthermore, this approach lacks an explicit solution, making it difficult to obtain the projection of a given point onto $\widehat{\cM}$. Iterative solvers are necessary to approximate the projections, although their convergence remains unproven.

\section{Numerical study}
This section presents a comprehensive numerical investigation of the superior performance of our method (ysl23) in manifold fitting. The experiments are divided into three parts, each showcasing the advantages of  ysl23 from different perspectives.
\begin{itemize}
    \item We comprehensively demonstrate ysl23's effectiveness through various numerical visualizations, performance evaluations on diverse manifolds, and exploration of its asymptotic properties. The experiments confirm that the asymptotic behavior of ysl23 aligns with the main theorems presented in this paper as we increase the number of samples and reduce noise. Through this, we establish the reliability and validity of ysl23.
    \item  We compare ysl23 with three major manifold-fitting methods: yx19 \cite{yao2019manifold}, cf18 \cite{fefferman2018fitting}, and km17 \cite{mohammed2017manifold}, on two constant curvature manifolds and one inconstant curvature manifold. Their performance is evaluated using metrics such as the Hausdorff distance, average distance, and running time. The comparisons demonstrate that ysl23 outperforms the other methods in terms of both accuracy and efficiency.
    \item We apply ysl23 to a particularly challenging class of manifolds, the Calabi–Yau manifolds \cite{calabi2015kahler,yau1978ricci}, which have a complex structure and diverse shapes. We demonstrate the effectiveness of ysl23 by fitting Calabi–Yau manifolds and evaluating their performance by comparing the output with the underlying Calabi–Yau manifold. Through these experiments, we show that ysl23 can accurately fit the most complex manifolds, demonstrating its versatility and applicability in challenging scenarios.
\end{itemize}

To ensure reproducibility, we followed a standardized setup similar to \cite{mohammed2017manifold}. For each manifold $\cM$, The generation and evaluation of the output manifold are based on the following steps.
\begin{itemize}
    \item[1.]  Independently generate the sample set $\mathcal{Y}$ with size $N$ from the distribution $\nu$ defined in \eqref{eq:def:nu}, where $\sigma$ is predefined.      
    \item[2.] Generate another set of initial points $\mathcal{W} = \{w_1,...,w_{N_0}\}$ near the underlying manifold, satisfying $\sigma/2 \leq d(w_i,\cM)\leq 2\sigma$. 
    \item[3.] Project every point in $\mathcal{W}$ from each tested method to the output manifold, respectively. Denote the projection of $\mathcal{W}$ as $\widehat{\mathcal{W}}$.
    \item[4.] Evaluate the performance of all tested methods via the three measures:
      \begin{itemize}
        \item The supremum of the approximation error,
         $
          \max_{j}d(\widehat{w}_j, \cM),
         $
         calculated as an estimation of the Hausdorff distance between $\widehat{\cM}$ and $\cM$.
        \item The average of the approximation error,
         $
          \frac{1}{N_0}\sum_{j}d(\widehat{w}_j, \cM),
         $
         calculated as an estimation of the average distance between $\widehat{\cM}$ and $\cM$.
         \item The CPU time of the tested method.
\end{itemize}    
\end{itemize}

{\bf Implementation and code} The numerical study is conducted on a standard tower workstation with AMD ThreadRipper 3970X @4.5 GHz and 128GB of DDR4-3200mHz RAM. The operating system is Windows 10 Professional 64 Bit. The simulations are implemented with \texttt{Matlab R2023a}, which is chosen for its ability to perform parallel running conveniently and reliably. The detailed algorithm used in this paper can be found in the supplementary material, and the latest version of \texttt{Python} and \texttt{Matlab} implementation are available at \href{https://github.com/zhigang-yao/manifold-fitting}{https://github.com/zhigang-yao/manifold-fitting}.

\subsection{Numerical illustrations of ysl23}
Three different manifolds, including two constant-curvature manifolds - a circle embedded in $\bR^{2}$ and a sphere embedded in $\bR^3$ - and a manifold with negative curvature, namely a torus embedded in $\bR^3$, will be tested in this and the next subsection. A visualization of these simulated manifolds is presented in Figure \ref{Fig:simulated manifold}.
\begin{figure}[htbp]
    \centering
    \includegraphics[width = 1\linewidth, height = 0.23\linewidth]{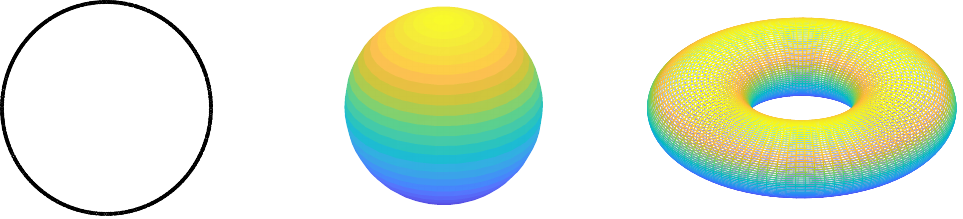}
    \caption{Manifolds employed in the numerical study. Left: a unit circle in $\mathbb{R}^2$; Middle: A unit sphere in $\mathbb{R}^3$; Right: a torus in $\mathbb{R}^3$. }\label{Fig:simulated manifold}
\end{figure}

\begin{algorithm}[ht]
{\color{black}
\caption{ysl23: Project $\mathcal{W}$ onto $\widetilde \cM$.}\label{alg:ysl23}
\raggedright Input: Initial points $\mathcal{W}$, noisy data $\mathcal{Y}$, three radius parameters $r_0$, $r_1$, and $r_2$.\\
\raggedright Output: Projection $\widehat{\mathcal{W}}$ of $\mathcal{W}$ onto $\widetilde \cM$.\\
\begin{itemize}
\item For each $w \in \mathcal{W}$:
\begin{itemize}
    \item[1.] Find the spherical neighborhood of $w$ with radius $r_0$, and denote the index of the samples in it as $I_w$.
    \item[2.] Calculate the weight function $\tilde{\alpha}_i(w)$ and $\alpha_i(w)$ for each  $i \in I_w$ as in (\ref{eq:def:weight_F}), then calculate $F(w)$ by (\ref{eq:def:F(z)}).
    \item[3.] Find the cylindrical neighborhood as in (\ref{eq:V_z:3}) with radius $r_1$ and $r_2$, and denote the index of the samples in it as $\widehat{I}_w$.
    \item[4.] Calculate the weight function $\tilde{\beta}_i(w)$ and $\beta_i(w)$ for each  $i \in \widehat{I}_w$ as in (\ref{eq:def:weight_G}), then calculate $G(w)$ by (\ref{eq:def:G(z)}).
    \item[5.] Obtain the output point as $\widehat{w} = G(w)$.
\end{itemize}
\end{itemize}
}
\end{algorithm}

\subsubsection{The fundamental procedure of ysl23}

Figure \ref{Fig:vis_alg} depicts a visualization of ysl23's steps using the circle as the underlying manifold. There are two simple steps in obtaining the final output for a given noisy point $w$. Firstly, the weighted means of a spherical neighborhood of $w$ are computed using (\ref{eq:def:F(z)}), which yields $F(w)$. The first step captures the crucial information about $w$, i.e., an approximation of the projected direction onto the underlying manifold. In the second step, the weighted means of a cylinder neighborhood of $F(w)$ are calculated to obtain the final output $G(w)$. The long axis of the cylinder is determined by the line connecting $w$ and $F(w)$. Notably, ysl23 requires no iteration or knowledge of the underlying manifold's dimension. Furthermore, ysl23 can map a noisy sample point not only approximately on the underlying manifold but also to its projection's proximity on the manifold, as demonstrated in panel (d) of Figure \ref{Fig:vis_alg}. As a summary, the detail of ysl23 can be found in Algorithm \ref{alg:ysl23}.
We always set the radius parameters as $r_0 = r_1 = 5\sigma/\lg(N)$ and $r_2 = 10\sigma\sqrt{\log(1/\sigma)}/\lg(N)$ in our experiment.   

\begin{figure}[ht]
    \centering
    \includegraphics[width = 1\linewidth, height = 0.24\linewidth]{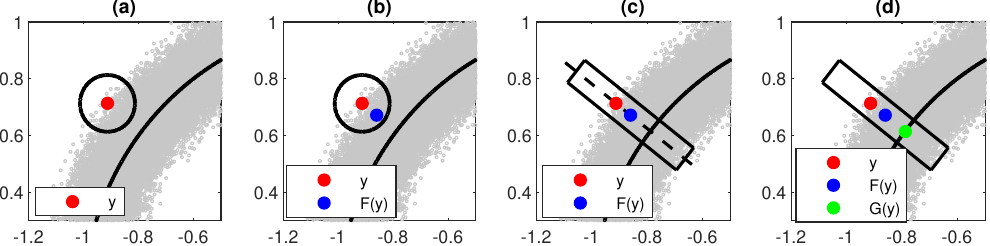}
    \caption{Visualization of ysl23's steps: (a) Locating the neighborhood of a noisy observation $w$. (b) Computing $F(w)$ defined in \eqref{eq:def:F(z)}. (c) Identifying the cylindrical neighborhood (points in the black rectangle) of $w$ based on $F(w)$. (d) Obtaining the output point $G(w)$ using \eqref{eq:def:G(z)}.}
    \label{Fig:vis_alg}
\end{figure}

\begin{figure}[ht]
    \centering
    \includegraphics[width = 0.9\linewidth, height = 0.34\linewidth]{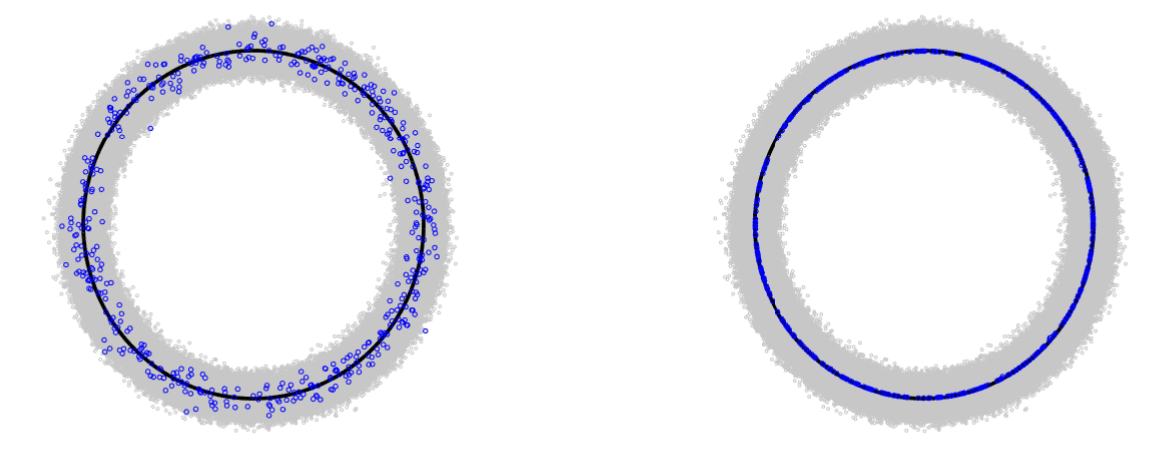}
    \caption{Assessing the performance of ysl23 in fitting the circle ($N = 5\times 10^4$, $N_0 = 100$, $\sigma = 0.06$): the left panel displays points in $\mathcal{W}$ surrounding the underlying manifold, while the right panel illustrates the corresponding points in $\widehat{\mathcal{W}}$.}\label{Fig:per_circle}
\end{figure}

The visualization of ysl23's performance for the circle case is shown in Figure \ref{Fig:per_circle}, and the result for the sphere and the torus case can be found in the supplementary material. In these tests, we set $N = 5\times 10^4$, $N_0 = 100$ for each case. The closer $\widehat{\mathcal{W}}$ are to the underlying manifold, the better it works. As can be observed from Figure \ref{Fig:per_circle}, the output points are significantly closer to the hidden manifold, clearly demonstrating the efficacy of ysl23. Similar phenomena, as shown in the supplementary material, can be observed for both sphere and torus cases.

\begin{figure}[ht]
    \centering
    \includegraphics[width = 1\linewidth, height = 0.6\linewidth]{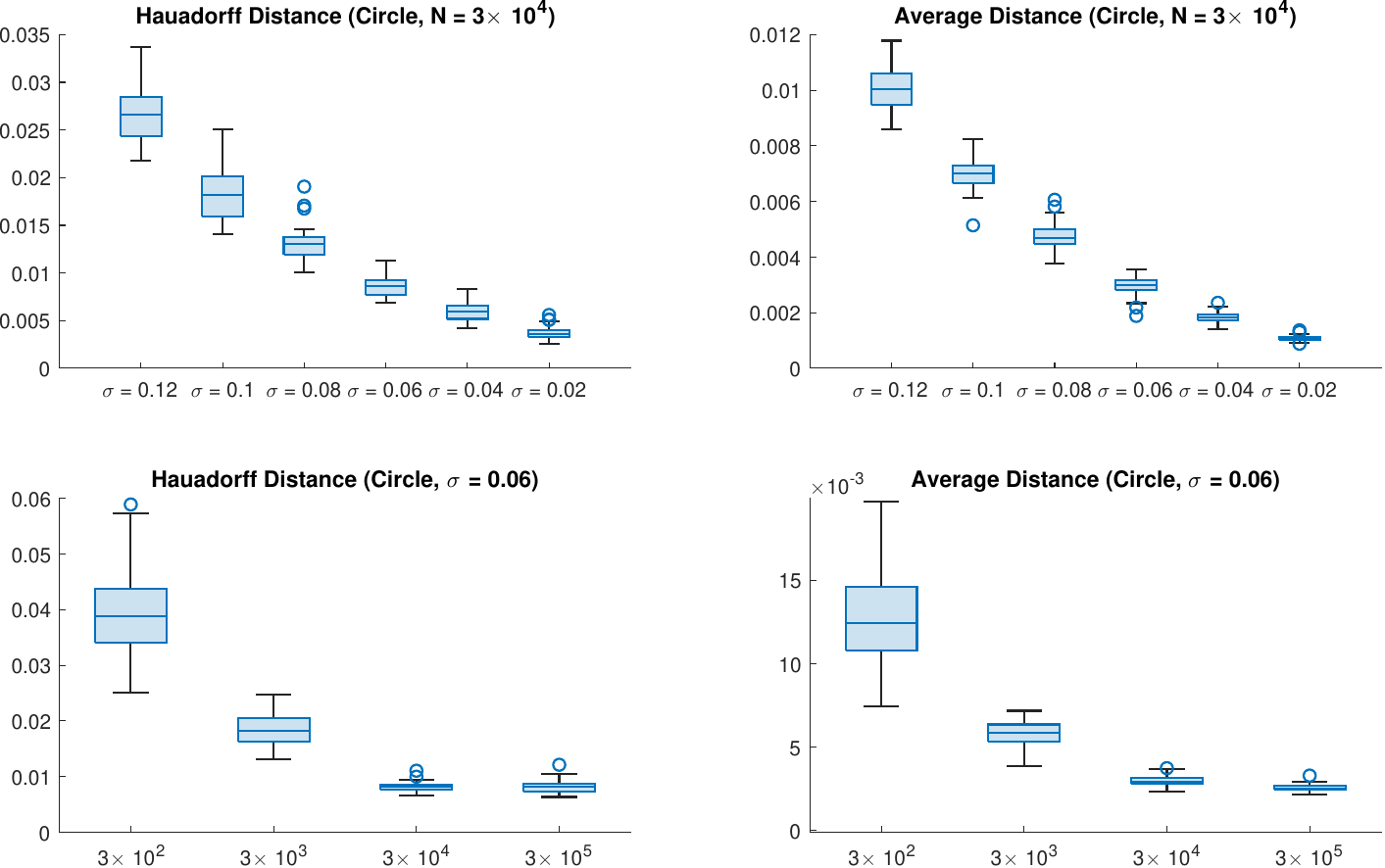}
    \caption{The asymptotic performance of ysl23 when fitting the circle. The top two figures show how the two distances change with $\sigma$, while the bottom two figure show how the two distances change with $N$.}\label{Fig:asym_circle}
\end{figure}

\subsubsection{Asymptotic analysis} 
To investigate the asymptotic properties of ysl23, we increased $N$ to simulate the case where it tends to infinity and decreased $\sigma$ to simulate the case where it tends to zero. Specifically, for the circle case, we considered $N \in \{3\times 10^2, 3\times 10^3, 3\times 10^4, 3\times 10^5\}$, and $\sigma \in \{0.12, 0.1, 0.08,0.06,0.04,0.02\}$. We started by fixing $N = 3\times 10^4$, $N_0 = 100$, and testing the performance of ysl23 with the change of $\sigma$. For each $\sigma$, we randomly selected 50 different $\mathcal{W}$ and executed ysl23 on each of them. The Hausdorff distances and average distances between the output manifold and the underlying manifold is shown at the top of Figure \ref{Fig:asym_circle}. It shows that the Hausdorff distance and average distance decrease at a quadratic rate as $\sigma$ decreases, which matches the upper bound of the error given in Section \ref{sec:fit_mfd}. We also observe that the average distance decreases more rapidly, demonstrating the global stability of ysl23. Similarly, we fixed $\sigma = 0.06$ to test the performance of ysl23 with the change of $N$. The Hausdorff distances and average distances between the output and hidden manifolds are shown at the bottom of Figure \ref{Fig:asym_circle}. It shows that, as $N$ increases, the Hausdorff distances and average distances both decrease significantly. 
This improvement can be attributed to two aspects. Firstly, with the increase of $N$, we can more accurately estimate the local geometry of the manifold. Secondly, the radius of the neighborhood in ysl23 is set to decrease with the increase of the sample size.  Hence, the neighborhood in ysl23 becomes closer to its center point while maintaining a sufficient number of points in the neighborhood. Similar results and phenomena, as shown in the supplementary material, can be observed for both sphere and torus cases.

\subsection{Comparison of other manifold fitting methods}
We performed ysl23, yx19, cf18, and km17 on the three aforementioned manifolds. The circles and spheres cases were combined since they both have constant curvature. The torus case was separately presented due to its inconstant curvature.

\begin{figure}[ht]
    \centering
    \includegraphics[width = 1\linewidth, height = 0.54\linewidth]{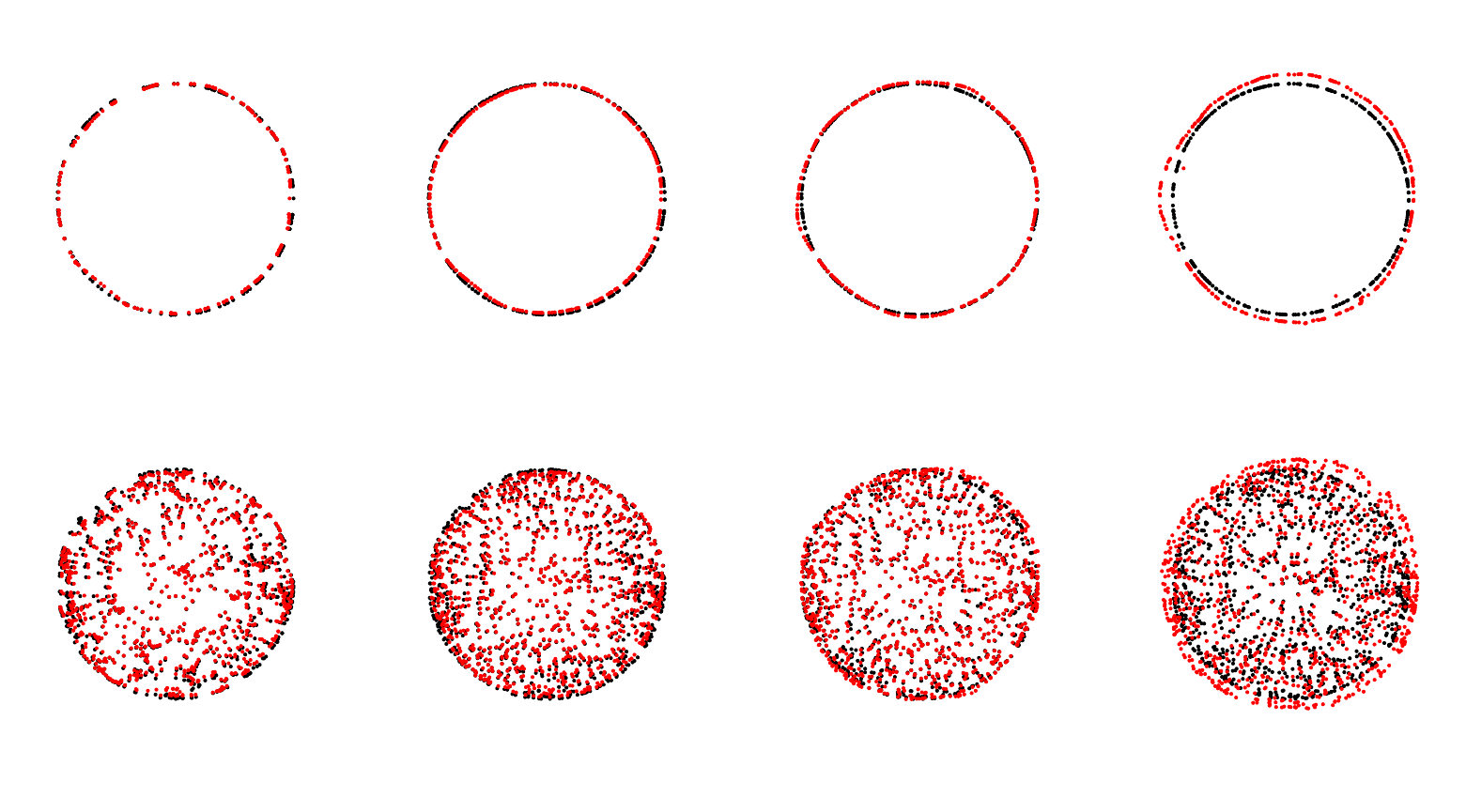}
    \caption{From left to right: the performance of ysl23, yx19, cf18, and km17 when fitting a circle (top, $N = 300$, $\sigma = 0.06$) and a sphere (bottom, $N = 1000$, $\sigma = 0.06$).}\label{Fig:cir_sph_com}
\end{figure}
 
\subsubsection{The fitting of the circle and sphere}
We set $N = N_0 = 300$ for the circle,  and $N = N_0 = 1000$ for the sphere. The radius of the neighborhood was set as $r = 2\sqrt{\sigma}$ for yx19, cf18, and km17. Figure \ref{Fig:cir_sph_com} displays the fitting results. The black and red dots correspond to $\widehat{\cM}$ and $\cM$, respectively. A higher degree of overlap between these two sets indicates a better fit. The first row presents the complete space for the circle embedded in $\mathbb{R}^2$, while the second row shows the view from the positive $z$-axis of the sphere embedded in $\mathbb{R}^3$. Notably, km17 demonstrates inferior performance compared with the other methods. Moreover, the estimated circles by cf18 exhibit two significant gaps, suggesting inaccuracies in the estimator for some local regions. The ysl23, as well as yx19, demonstrates the best performance.

We made an observation of interest when ysl23 successfully mapped the noisy samples to the proximity of the hidden manifold, but the sample distribution on the output manifold was slightly changed. This phenomenon occurred because the number of samples was not sufficient to represent the perturbation of the uniform distribution on the manifold. Because of this, our contraction strategy clustered the output points towards the denser regions on the input points. Fortunately, when the sample size is sufficiently large, ysl23 is able to ensure that the output points are approximately uniformly distributed on $\widehat{\cM}$ 
(see Figure \ref{Fig:scatter_more} in the supplementary material). 

\begin{figure}[ht]
    \centering
    \includegraphics[width = 1\linewidth, height = 0.5\linewidth]{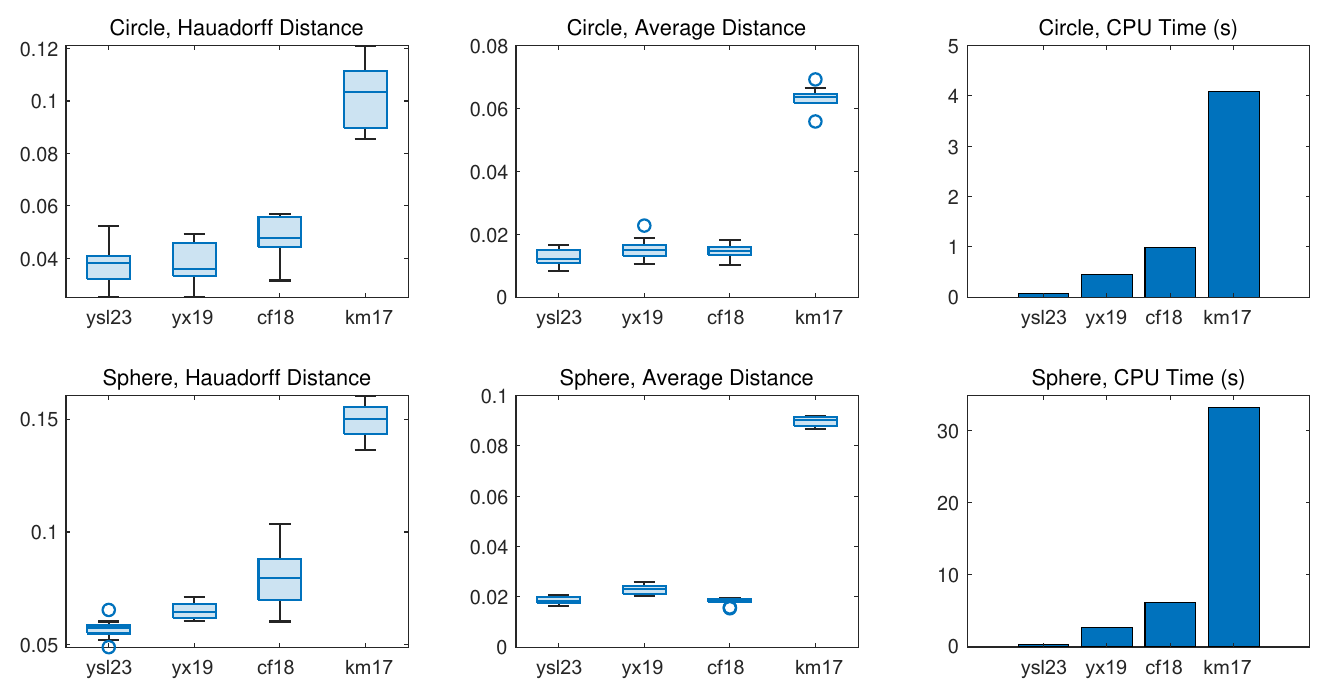}        
    \caption{The Hausdorff distance, average distance, and CPU time of fitting a circle (top, $N = 300$, $\sigma = 0.06$) and a sphere (bottom, $N = 1000$, $\sigma = 0.06$), using ysl23, yx19, cf18, and km17. }\label{Fig:box_circle_sphere}
\end{figure}

We repeated each method $10$ times and evaluated their effectiveness in Figure \ref{Fig:box_circle_sphere}.  We find that ysl23 and yx19 achieve slightly better results than cf18 in terms of the Hausdorff distance, while all three outperform km17 significantly. When evaluating the average distance, ysl23 and cf18 slightly outperform yx19, while all three show significant improvement over km17. Overall, ysl23 consistently ranks among the top across different metrics. In terms of computing time, ysl23 also stands out, with remarkably lower running times than those of the other three methods. Among them, yx19 is the most efficient, while km17 lags behind significantly.

\begin{figure}[ht]
    \centering
    \includegraphics[width = 1\linewidth, height = 0.6\linewidth]{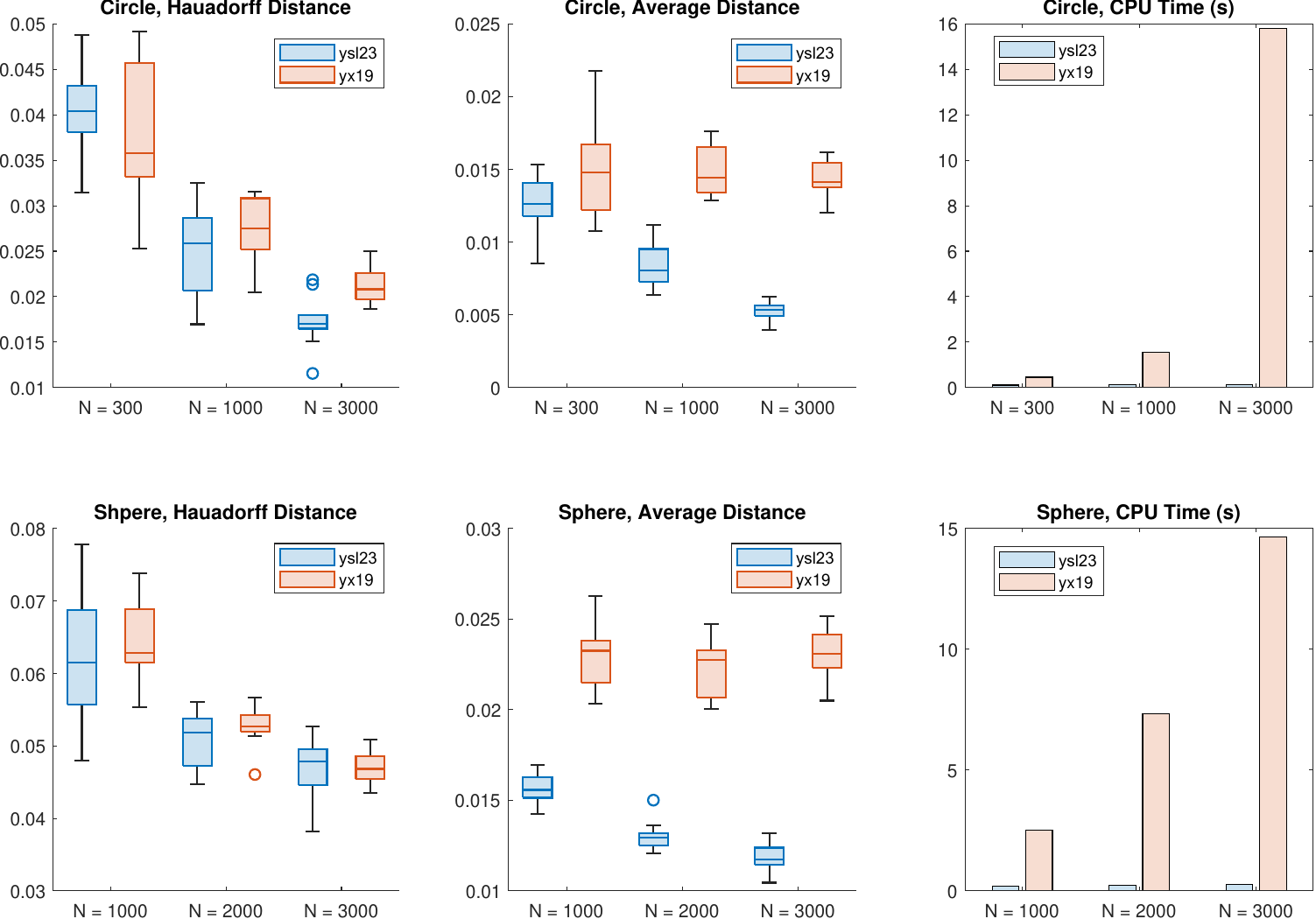}       
    \caption{The Hausdorff distance, average distance, and CPU time of fitting a circle (top, $\sigma = 0.06$) and a sphere (bottom, $\sigma = 0.06$) with increasing $N$, using ysl23 and yx19. }\label{Fig:box_com_asy_circle_sphere}
\end{figure}

We compared ysl23 and the well-performing yx19 by incrementally varying $N$to explore their performance dependence on it. For the circle case, we selected $N \in \{3\times 10^2,1\times 10^3, 3\times 10^3\}$, while for the sphere case, we selected $N \in \{1\times 10^2, 2\times 10^3, 3\times 10^3\}$. Results in terms of Hausdorff and average distance and running time are shown in Figure \ref{Fig:box_com_asy_circle_sphere}. The Hausdorff distance showed a significant decrease for both algorithms as $N$ increased. However, yx19 remained relatively constant with increasing $N$ when using the average distance, while ysl23 achieved a significant reduction. Additionally, ysl23 demonstrated a clear advantage in computational efficiency, with significantly shorter running times than yx19. For example, yx19 took over 10 seconds to terminate when $N$ reached 3000 in the presented examples, while ysl23 was completed in under 0.5 seconds.

\subsubsection{The fitting of the torus}
We set $N = 10^3$ for the torus case. The results, displayed in Figure \ref{Fig:box_com_torus}, show that ysl23 outperformed the other three methods in terms of the Hausdorff distance, average distance, and computing time. 
\begin{figure}[ht]
    \centering
    \includegraphics[width = 1\linewidth, height = 0.3\linewidth]{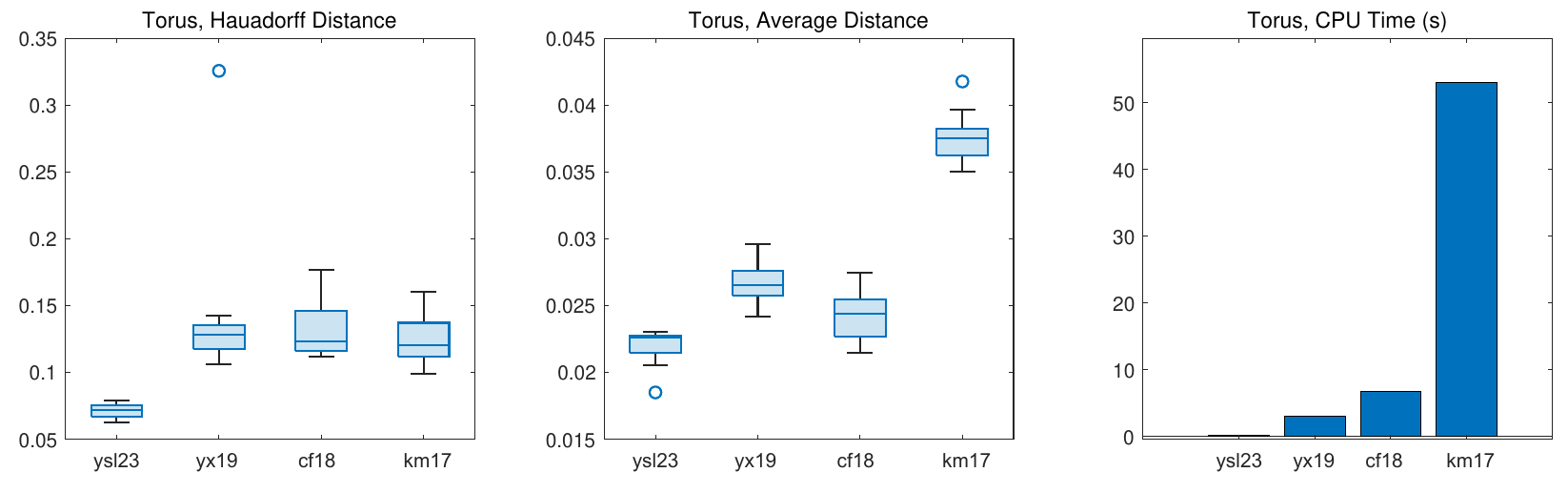}
    \caption{The Hausdorff distance, average distance, and CPU time of fitting a torus ($N = 1000$, $\sigma = 0.06$), using ysl23, yx19, cf18, and km17.}\label{Fig:box_com_torus}
\end{figure}
To evaluate the performance of ysl23 and yx19 on the torus, we set an increasing sample size of $N \in \{1000,2000,3000\}$ and compared their results.  Figure \ref{Fig:asy_com_torus} illustrates the results of both algorithms for each $N$. As $N$ increased, we observed a reduction in the distance for both algorithms. However, ysl23 consistently achieved a much lower distance than yx19, no matter which metric is used. Furthermore, ysl23 demonstrated a remarkable advantage in computational efficiency, completing the task with a significantly shorter running time than yx19. Specifically, in the presented examples, yx19 took over 10 seconds to terminate when $N$ reached 3000, while ysl23 finished in under 0.5 seconds.
\begin{figure}[ht]
    \centering
    \includegraphics[width = 1\linewidth, height = 0.3\linewidth]{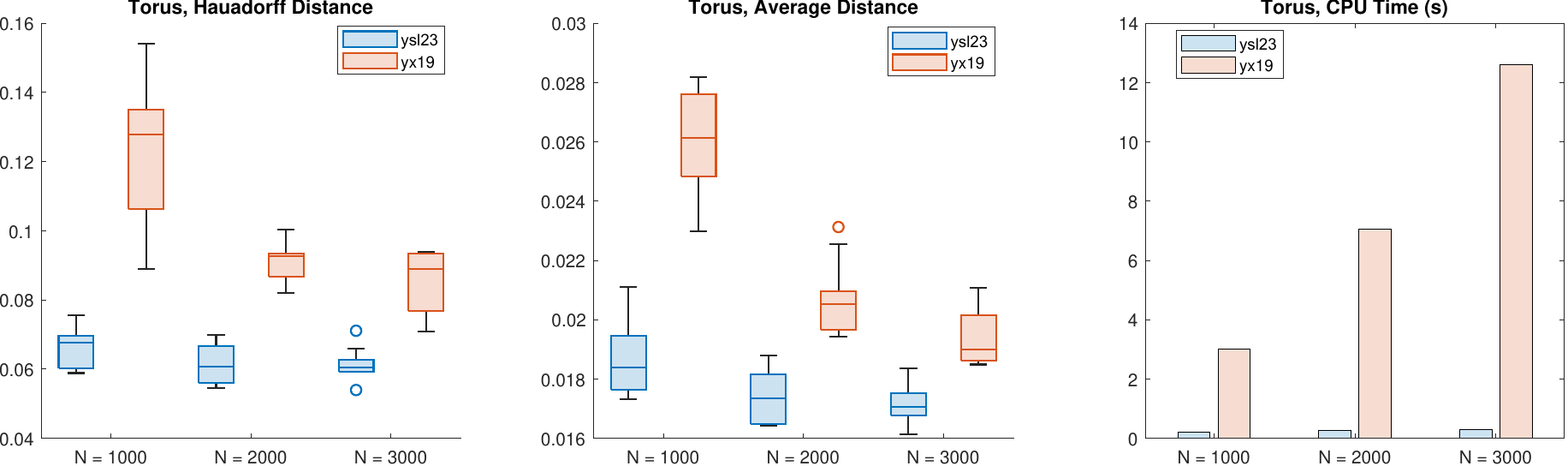}
    \caption{The Hausdorff distance, average distance, and CPU time of fitting a torus ($\sigma = 0.06$) with increasing $N$, using ysl23 and yx19.}\label{Fig:asy_com_torus}
\end{figure}

\subsection{Fitting of a Calabi–Yau manifold}
Calabi–Yau manifolds \cite{calabi2015kahler} are a class of compact, complex Kähler manifolds that possess a vanishing first Chern class. They are highly significant because they are Ricci-flat manifolds, which means that their Ricci curvature is zero at all points, aligning with the universe model of physicists. A simple example of a Calabi–Yau manifold is the Fermat quartic:
\begin{align}\label{cy3d}
x^4 + y^4 + z^4 + w^4 = 0, \quad (x, y, z, w) \in \mathbb{P}^3,
\end{align}
where $\mathbb{P}^3$ refers to the complex projective 3-space. To visualize it, we generate low-dimensional projections of the manifold by eliminating variables as in \cite{hanson1994construction}, dividing by $w^4$, and setting $\frac{z^4}{w^4}$ to be constant.  We then normalize the resulting inhomogeneous equation as
\begin{align}\label{cy_pj}
 x^4 + y^4 = 1, \quad x, y \in \mathbb{C}.
\end{align}
The resulting surface is embedded in 4D and can be projected to ordinary 3D space for display. The parametric representation of (\ref{cy_pj}) is
\begin{align}
    x(\theta, \, k_1) &= e^{2\pi ik_1/4}\cosh(\theta + \zeta i)^{2/4} \label{cy_x}\\ 
    y(\theta, \zeta, k_2) &= e^{2\pi ik_2/4}\sinh(\frac{\theta + \zeta i}{i})^{2/4}, \label{cy_y}
\end{align}
where the integer pair $(k_1,k_2)$ is selected by $0\le k_1,k_2\le 3$. Such $\{(x,y)\}$ can be seen as points in $\mathbb{R}^4$, denoted by $\{Re(x), Re(y), Im(x), Im(y)\}$. A natural 3D projection is
\[
(Re(x), Re(y), \cos(\psi)Im(x) + \sin(\psi)Im(y)),
\]
where $\psi$ is a parameter. The left panel of Figure \ref{Fig:scatter_cy} shows the surface plot of the 3D projection. 

\begin{figure}[ht]
    \centering
    \includegraphics[width = 1\linewidth, height = 0.28\linewidth]{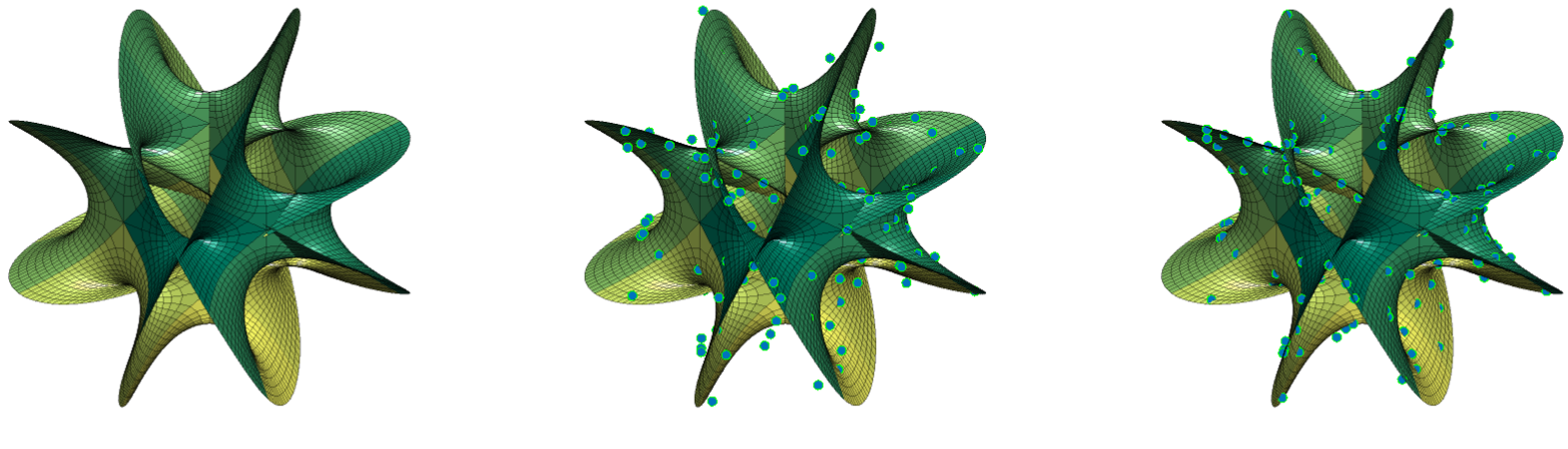}
    \caption{Performance of ysl23 when fitting the real projection of the Calabi–Yau manifold (\ref{cy3d}). The left panel illustrates the shape of the 3D projection. The middle panel shows some noisy points around the manifold, and the right panel shows the points on the output manifold.}\label{Fig:scatter_cy}
\end{figure}

We generated a set of points in (\ref{cy_x}) and (\ref{cy_y}) on a uniform grid $(\theta, \zeta)$, where $\theta$ is a sequence of numbers ranging from $-1.5$ to $1.5$ with a step size of $0.05$ between consecutive values, and $\zeta$ a sequence of numbers ranging from $0$ to $\pi/2$ with a step size of $1/640$ between consecutive values. In total, the dataset contains $N=313296$ samples with Gaussian noise added in $\mathbb{R}^4$. As shown in the middle panel of Figure \ref{Fig:scatter_cy}, the initial point distribution is not close to the manifold. However, after running ysl23, the output is significantly closer to it, as shown in the right panel of Figure \ref{Fig:scatter_cy}. This phenomenon indicates that ysl23 performs well in estimating complicated manifolds. It should be noted that we only applied ysl23 to this example without running other algorithms because the sample size would cause very long running times for other algorithms and would not yield usable results.

\begin{figure}[ht]
    \centering
    \includegraphics[width = 1\linewidth, height = 0.25\linewidth]{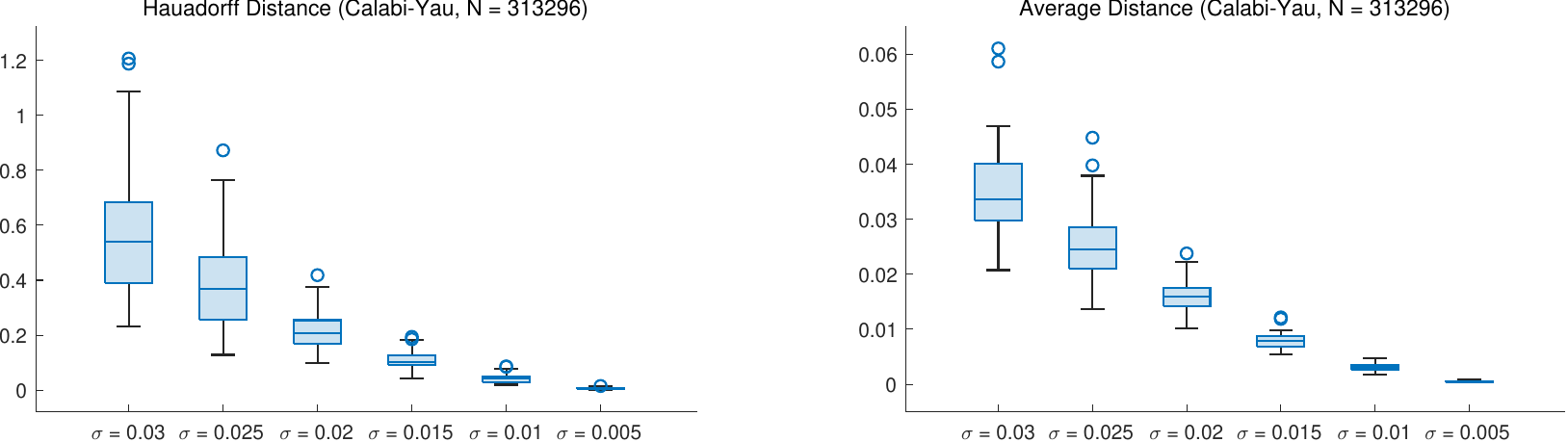}
    \caption{The asymptotic performance of ysl23 fitting the real projection of the Calabi–Yau manifold (\ref{cy3d}). The two panels show how the two distances change with $\sigma$.}\label{Fig:asy_cy}
\end{figure}

We also executed ysl23 with different $\sigma$. Specifically, we tested ysl23 with decreasing $\sigma \in \{0.03, 0.025, 0.02,0.015,0.01,0.005\}$. As we decrease $\sigma$, both the Hausdorff distance and average distance decrease at a quadratic rate, which matches Theorem \ref{Thm:out_manifold_global_d}.  These results further support the effectiveness and reliability of ysl23.

\section{Conclusion}
In this paper, the manifold-fitting problem is investigated by proposing a novel approach to construct a manifold estimator for the latent manifold in the presence of ambient-space noise. Our estimator achieves the best error rate, to our knowledge, with a sample size bounded by a polynomial of the standard deviation of the noise term, and preserves the smoothness of the latent manifold. The performance of the estimator is demonstrated through rigorous theoretical analysis and numerical experiments. Our method provides a reliable and efficient solution to the problem of manifold fitting from noisy observations, with potential applications in various fields, such as computer vision and machine learning.

Our approach uses a two-step local contraction strategy to obtain an output manifold with a significantly smaller error. First, we estimate the direction of contraction for a point around $\cM$ using a local average. Compared with previous methods that estimate the basis of the tangent space, our approach provides a significant advantage in terms of the error rate and facilitates the obtaining of better-contracted points. Next, we construct a hyper-cylinder, and the local average within it is regarded as the contracted point. This point is $\cO(\sigma^2\log(1/\sigma))$-close to $\cM$. Our hyper-cylinder has a length in a higher order of $\sigma$ than the width, which differs from the approach proposed in \cite{fefferman2021fitting}. This difference in order allows us to eschew their requirement of directly sampling from $\cM$.


We provide several methods to obtain the estimators of $\cM$. All of these estimators can roughly achieve a Hausdorff distance in the order of $\cO(\sigma^2\log(1/\sigma))$, with or without the high probability statement. Unlike in previous work, we achieve the state-of-the-art error bound by reducing the required sample size to $N=\cO(\sigma^{-(d+3)})$. Using image sets to generate estimators, our method is faster and more applicable to larger data sets. We also conduct comprehensive numerical experiments to validate our theoretical results and demonstrate that our algorithm not only achieves higher approximation accuracy but also consumes significantly less time and computational resources than other methods. These simulation results indicate the significant superiority of our approach in fitting the latent manifold, and suggest its potential in various applications.

Overall, our approach has demonstrated promising results in fitting smooth manifolds from ambient space, but nevertheless has some limitations that warrant further investigation. First, our current assumption that the observations are from the convolution of a uniform distribution on the manifold with a homogeneous Gaussian distribution may not capture the full complexity of real-world data. Therefore, future research could explore the effects of relaxing these assumptions. Second, while our theoretical results are promising, there is still scope for optimization because of the application of inequalities in the proof and the choice of weights in the two-step mapping. This limitation arises from the lack of an explicit expression for some integrations with respect to Gaussian distributions. We believe that further research addressing these limitations can lead to significant advancements in manifold-fitting methods, at both the theoretical and applied levels.

To conclude, we discuss potential avenues for further research. In the real world, data often exist on complicated manifolds, such as spheres, tori, and shape spaces, requiring specialized analysis methods. Our manifold-fitting algorithm projects data onto a low-dimensional manifold, allowing the use of other algorithms.
Firstly, our approach has wide-ranging implications for research involving the manifold hypothesis. For example, in GAN-based image-to-image translation, images are assumed to lie around a low-dimensional manifold. Incorporating our manifold-fitting method can significantly enhance the performance of the discriminator and improve the overall GAN model. Secondly, numerous statistical studies concentrate on non-Euclidean data originating from manifolds, including the principal nested spheres \cite{principal_nested_spheres} and the principal flows \cite{principal_flows}. As our method can fit smooth $d$-dimensional manifolds from ambient space, it provides a natural framework for generalizing statistical work on manifolds to ambient space. Additionally, our method can also aid in the analysis of Euclidean data by facilitating data clustering and simplifying subsequent objectives.
We believe that our approach will inspire further research in these areas.

\clearpage
\begin{appendix}

\section{Mathematical Preliminary}
We briefly review the basic concepts of topology and smooth manifolds essential for the study of manifold fitting; for further details, see, for example, \cite{LeeTM,LeeSM,LeeRM}.

\subsection{Topology}
\subsubsection{Topological Space}
Let $X$ be a set. A \bfemph{topology} on $X$ is a collection $\cT$ of subsets of $S$, called \bfemph{open subsets}, satisfying the following:
\begin{itemize}
    \item [(a)] $X$ and $\varnothing$ are open.
    \item [(b)] The union of any family of open sets is open.
    \item [(c)] The intersection of any finite family of open subsets is open.
\end{itemize}
A pair $(X, \mathcal{T})$ consisting of a set $X$ and a topology $\mathcal{T}$ on $X$ is called a \bfemph{topological space}. Usually, when the topology is understood, these details will be omitted, with only the statement that "$X$ is a topological space".

The most common examples of topological spaces, from which most of our examples of manifolds are built, are presented below.
\begin{Example}[Metric Spaces]    
    A metric space is a set $M$ endowed with a distance function (also called a metric) $d: M \times M \rightarrow \mathbb{R}$ (where $\mathbb{R}$ denotes the set of real numbers) satisfying the following properties for all $x, y, z \in M$ :
    \begin{itemize}
        \item [(a)] \textbf{Positivity}: $d(x, y) \geq 0$, with equality if and only if $x=y$.
        \item [(b)] \textbf{Symmetry}: $d(x, y)=d(y, x)$.
        \item [(c)] \textbf{Triangle inequality}: $d(x, z) \leq d(x, y)+d(y, z)$.
    \end{itemize}
    If $M$ is a metric space, $x \in M$, and $r>0$, the \bfemph{open ball of radius $\boldsymbol{r}$ around} $\boldsymbol{x}$ is the set
    $$
    B(x,r)=\{y \in M: d(x, y)<r\}.
    $$
    The metric topology on $M$ is defined by declaring a subset $S \subseteq M$ to be open if, for every point $x \in S$, there is some $r>0$ such that $B(x,r) \subseteq S$.
\end{Example}
\begin{Example}[Euclidean Spaces]
    For integer $n \geq 1$, the set $\mathbb{R}^n$ of ordered $n$-tuples of real numbers is called \bfemph{$\boldsymbol{n}$-dimensional Euclidean space}.
    We let a point in $\mathbb{R}^n$ be denoted by $\left(x^{(1)}, \cdots, x^{(n)}\right)$ or $\vec x$. The numbers $x^{(i)}$ are called the \bfemph{$\boldsymbol{i}$-th components or coordinates} of $\vec x$.
    For $\vec x \in \mathbb{R}^n$, the \bfemph{Euclidean norm} of $\vec x$ is the nonnegative real number
    $$
    \|\vec x\|_2=\sqrt{\left(x^{(1)}\right)^2+\cdots+\left(x^{(n)}\right)^2},
    $$
    and, for $\vec x, \vec y \in \mathbb{R}^n$, the Euclidean distance function is defined by
    $$
    d(\vec x, \vec y)=\|\vec x-\vec y\|_2 .
    $$
    This distance function turns $\mathbb{R}^n$ into a complete metric space. The resulting metric topology on $\mathbb{R}^n$ is called the Euclidean topology.
\end{Example}

For the purposes of manifold theory, arbitrary topological spaces are too general. To avoid pathological situations arising when there are not enough open subsets of $X$, we often restrict our attention to \bfemph{Hausdorff space}.
\begin{definition}[Hausdorff space]
    A topological space $X$ is said to be a Hausdorff space if, for every pair of distinct points $p, q \in X$, there exist disjoint open subsets $U, V \subseteq X$ such that $p \in U$ and $q \in V$. 
\end{definition}

There are numerous essential concepts in topology concerning \bfemph{maps}, and these will be introduced next. Let $X$ and $Y$ be two topological spaces, and $F: X \rightarrow Y$ be a map between them.
\begin{itemize}
    \item $F$ is \bfemph{continuous} if, for every open subset $U \subseteq Y$, the preimage $F^{-1}(U)$ is open in $X$.
    \item If $F$ is a continuous bijective map with continuous inverse, it is called a \bfemph{homeomorphism}. If there exists a homeomorphism from $X$ to $Y$, we say that $X$ and $Y$ are \bfemph{homeomorphic}.
    \item A continuous map $F$ is said to be a \bfemph{local homeomorphism} if every point $p \in X$ has a neighborhood $U \subseteq X$ such that $F(U)$ is open in $Y$ and $F$ restricts to a homeomorphism from $U$ to $F(U)$.
    \item $F$ is said to be a \bfemph{closed map} if, for each closed subset $K \subseteq X$, the image set $F(K)$ is closed in $Y$, and an \bfemph{open map} if, for each open subset $U \subseteq X$, the image set $F(U)$ is open in $Y$. It is a \bfemph{quotient map} if it is surjective and $V \subseteq Y$ is open if and only if $F^{-1}(V)$ is open.
\end{itemize}
Furthermore, for a continuous map $F$, which is either open or closed, the following rules apply:
\begin{itemize}
    \item [(a)] If $F$ is surjective, it is a \emph{quotient map}.
    \item [(b)] If $F$ is injective, it is a \emph{topological embedding}.
    \item [(c)] If $F$ is bijective, it is a \emph{homeomorphism}.
\end{itemize}

For maps between metric spaces, there are several useful variants of continuity, especially in the case of compact spaces. Assume $\left(M_1, d_1\right)$ and $\left(M_2, d_2\right)$ are metric spaces, and $F: M_1 \rightarrow M_2$ is a map. Then, $F$ is said to be \bfemph{uniformly continuous} if, for every $\epsilon>0$, there exists $\delta>0$ such that, for all $x, y \in M_1, d_1(x, y)<\delta$ implies $d_2(F(x), F(y))<\epsilon$. It is said to be \bfemph{Lipschitz continuous} if there is a constant $C$ such that $d_2(F(x), F(y)) \leq C d_1(x, y)$ for all $x, y \in M_1$. Any such $C$ is called a \bfemph{(globally) Lipschitz constant} for $\boldsymbol{F}$. We say that $F$ is \bfemph{locally Lipschitz continuous} if every point $x \in M_1$ has a neighborhood on which $F$ is Lipschitz continuous.

\subsubsection{Bases and countability}
Suppose $X$ is merely a set, and $\mathcal{B}$ is a collection of subsets of $X$ satisfying the following conditions:
\begin{itemize}
    \item [(a)] $X=\bigcup_{B \in \mathcal{B}} B$.
    \item [(b)] If $B_1, B_2 \in \mathcal{B}$ and $x \in B_1 \cap B_2$, then there exists $B_3 \in \mathcal{B}$ such that $x \in B_3 \subseteq$ $B_1 \cap B_2$.
\end{itemize}
Then, the collection of all unions of elements of $\mathcal{B}$ is a topology on $X$, called the \bfemph{topology generated by $\boldsymbol{\mathcal{B}}$}, and $\mathcal{B}$ is a \bfemph{basis} for this topology.

A set is said to be \bfemph{countably infinite} if it admits a bijection with the set of positive integers, and \bfemph{countable} if it is finite or countably infinite. A topological space $X$ is said to be \bfemph{first-countable} if there is a countable neighborhood basis at each point, and \bfemph{second-countable} if there is a countable basis for its topology. Since a countable basis for $X$ contains a countable neighborhood basis at each point, second-countability implies first-countability.

\subsubsection{Subspaces and Products}
If $X$ is a topological space and $S \subseteq X$ is an arbitrary subset, we define the \bfemph{subspace topology} (or \bfemph{relative topology}) on $S$ by declaring a subset $U \subseteq S$ to be open in $S$ if and only if there exists an open subset $V \subseteq X$ such that $U=V \cap S$. A subset of $S$ that is open or closed in the subspace topology is sometimes said to be \bfemph{relatively open} or \bfemph{relatively closed} in $S$, to make it clear that we do not mean open or closed as a subset of $X$. Any subset of $X$ endowed with the subspace topology is said to be a \bfemph{subspace} of $X$.

If $X$ and $Y$ are topological spaces, a continuous injective map $F: X \rightarrow Y$ is called a \bfemph{topological embedding} if it is a homeomorphism onto its image $F(X) \subseteq Y$ in the subspace topology.

If $X_1, \cdots, X_k$ are (finitely many) sets, their \bfemph{Cartesian product} is the set $X_1 \times \cdots \times X_k$ consisting of all ordered $k$-tuples of the form $\left(\vec x_1, \cdots, \vec x_k\right)$ with $\vec x_i \in X_i$ for each $i$.

Suppose $X_1, \cdots, X_k$ are topological spaces. The collection of all subsets of $X_1 \times$ $\cdots \times X_k$ of the form $U_1 \times \cdots \times U_k$, where each $U_i$ is open in $X_i$, forms a basis for a topology on $X_1 \times \cdots \times X_k$, called the \bfemph{product topology}. Endowed with this topology, a finite product of topological spaces is called a \bfemph{product space}. Any open subset of the form $U_1 \times \cdots \times U_k \subseteq X_1 \times \cdots \times X_k$, where each $U_i$ is open in $X_i$, is called a \bfemph{product open subset}.

\subsubsection{Connectedness and Compactness}
A topological space $X$ is said to be \bfemph{disconnected} if it has two disjoint nonempty open subsets whose union is $X$, and it is \bfemph{connected} otherwise. Equivalently, $X$ is connected if and only if the only subsets of $X$ that are both open and closed are $\varnothing$ and $X$ itself. If $X$ is any topological space, a \bfemph{connected subset} of $X$ is a subset that is a connected space when endowed with the subspace topology.

Closely related to connectedness is \bfemph{path connectedness}. If $X$ is a topological space and $p, q \in X$, a \bfemph{path} in $X$ from $p$ to $q$ is a continuous map $f: I \rightarrow X$ (where $I=[0,1]$ ) such that $f(0)=p$ and $f(1)=q$. If for every pair of points $p, q \in X$ there exists a path in $X$ from $p$ to $q$, then $X$ is said to be \bfemph{path-connected}.

A topological space $X$ is said to be \bfemph{compact} if every open cover of $X$ has a finite subcover. A \bfemph{compact subset} of a topological space is one that is a compact space in the subspace topology. For example, it is a consequence of the Heine–Borel theorem that a subset of $\mathbb{R}^n$ is compact if and only if it is closed and bounded. We list some of the properties of compactness as follows.

\begin{itemize}
    \item If $F: X \rightarrow Y$ is continuous and $X$ is compact, then $F(X)$ is compact.
    \item If $X$ is compact and $f: X \rightarrow \mathbb{R}$ is continuous, then $f$ is bounded and attains its maximum and minimum values on $X$.
    \item Any union of finitely many compact subspaces of $X$ is compact.
    \item If $X$ is Hausdorff and $K$ and $L$ are disjoint compact subsets of $X$, then there exist disjoint open subsets $U, V \subseteq X$ such that $K \subseteq U$ and $L \subseteq V$.
    \item Every closed subset of a compact space is compact.
    \item Every compact subset of a Hausdorff space is closed.
    \item Every compact subset of a metric space is bounded.
    \item Every finite product of compact spaces is compact.
    \item Every quotient of a compact space is compact.
\end{itemize}

\subsection{Smooth Manifold}
\subsubsection{Topological Manifolds}
A $\boldsymbol{d}$\bfemph{-dimensional topological manifold} (or simply a $\boldsymbol{d}$\bfemph{-manifold}) is a second-countable Hausdorff topological space that is \bfemph{locally Euclidean of dimension $\boldsymbol{d}$}, which means every point has a neighborhood homeomorphic to an open subset of $\mathbb{R}^d$. Given a $d$-manifold $\cM$, a \bfemph{coordinate chart} for $\cM$ is a pair $(U, \varphi)$, where $U \subseteq \cM$ is an open set and $\varphi: U \rightarrow \widetilde{U}$ is a homeomorphism from $U$ to an open subset $\widetilde{U} \subseteq \mathbb{R}^d$. If $p \in \cM$ and $(U, \varphi)$ is a chart such that $p \in U$, we say that $(U, \varphi)$ is a \bfemph{chart containing} $p$.

On occasion, we may need to consider manifolds with boundaries. A $\boldsymbol{d}$\bfemph{-dimensional topological manifold with boundary} a is a second-countable Hausdorff topological space in which every point has a neighborhood homeomorphic either to an open subset of $\mathbb{R}^d$ or to an open subset of the half space of $\bR^d$. The corresponding concepts are slightly different with the manifolds without boundaries. For consistency, in the following sections a \emph{manifold} without further qualification is always assumed to be a manifold without a boundary.

\subsubsection{Smooth Manifolds}
Briefly speaking, \bfemph{smooth manifolds} are topological manifolds endowed with an extra structure that allows us to differentiate functions and maps. To introduce the smooth structure, we first recall the smoothness of a map $F:U\to\bR^k$. When $U$ is an open subset of $\bR^d$, $F$ is said to be \bfemph{smooth} (or $C^\infty$), and all of its component functions have continuous partial derivatives of all orders.
More generally, when the domain $U$ is an arbitrary subset of $\bR^d$, not necessarily open, $F$ is said to be smooth if, for each $x \in U,$ $F$ has a smooth extension to a neighborhood of $x$ in $\mathbb{R}^n$. A \bfemph{diffeomorphism} is a bijective smooth map whose inverse is also smooth.

If $\cM$ is a topological $d$-manifold, then two coordinate charts $(U, \varphi),(V, \psi)$ for $\cM$ are said to be \bfemph{smoothly compatible} if both of the \bfemph{transition maps} $\psi \circ \varphi^{-1}$ and $\varphi \circ \psi^{-1}$ are smooth where they are defined (on $\varphi(U \cap V)$ and $\psi(U \cap V)$, respectively). Since these maps are inverses of each other, it follows that both transition maps are in fact diffeomorphisms. An \bfemph{atlas} for $\cM$ is a collection of coordinate charts whose domains cover $\cM$. It is called a \bfemph{smooth atlas} if any two charts in the atlas are smoothly compatible. A \bfemph{smooth structure} on $\cM$ is a smooth atlas that is maximal, which means it is not properly contained in any larger smooth atlas. A smooth manifold is a topological manifold endowed with a specific smooth structure. If $\cM$ is a set, a smooth manifold structure on $\cM$ is a second-countable, Hausdorff, locally Euclidean topology together with a smooth structure, making it a smooth manifold.
If $\cM$ is a smooth $d$-manifold and $W \subseteq \cM$ is an open subset, then $W$ has a natural smooth structure consisting of all smooth charts $(U, \varphi)$ for $\cM$ such that $U \subseteq W$, and so every open subset of a smooth $d$-manifold is a smooth $d$ manifold in a natural way.

Suppose $\cM$ and $\cN$ are smooth manifolds. A map $F: \cM \rightarrow \cN$ is said to be \bfemph{smooth} if, for every $p \in \cM$, there exist smooth charts $(U, \varphi)$ for $\cM$ containing $p$ and $(V, \psi)$ for $\cN$ containing $F(p)$ such that $F(U) \subseteq V$ and the composite map $\psi \circ F \circ \varphi^{-1}$ is smooth from $\varphi(U)$ to $\psi(V)$. In particular, if $\cN$ is an open subset of $\mathbb{R}^k$ or $\mathbb{R}_{+}^k$ with its standard smooth structure, we can take $\psi$ to be the identity map of $\cN$, and then smoothness of $F$ simply means that each point of $\cM$ is contained in the domain of a chart $(U, \varphi)$ such that $F \circ \varphi^{-1}$ is smooth. It is a clear and direct consequence of the definition that identity maps, constant maps, and compositions of smooth maps are all smooth. A map $F: \cM \rightarrow \cN$ is said to be a \bfemph{diffeomorphism} if it is smooth and bijective and $F^{-1}: \cN \rightarrow \cM$ is also smooth.

We let $C^{\infty}(\cM, \cN)$ denote the set of all smooth maps from $\cM$ to $\cN$, and $C^{\infty}(\cM)$ the vector space of all smooth functions from $\cM$ to $\mathbb{R}$. For every function $f: M \rightarrow \mathbb{R}$ or $\mathbb{R}^k$, we define the support of $f$, denoted by $supp~f$, as the closure of the set $\{x \in \cM: f(x) \neq 0\}$. If $A \subseteq \cM$ is a closed subset and $U \subseteq \cM$ is an open subset containing $A$, then a \bfemph{smooth bump function} for $A$ supported in $U$ is a smooth function $f: \cM \rightarrow \mathbb{R}$ satisfying $0 \leq f(x) \leq 1$ for all $x \in M,\left.f\right|_A \equiv 1$, and $supp~f \subset U$. Such smooth bump functions always exist.

There are various equivalent approaches to define tangent vectors on $\cM$. The most convenient one is via the following definition: for every point $p \in \cM$, a \bfemph{tangent vector} at $p$ is a linear map $v: C^{\infty}(\cM) \rightarrow \mathbb{R}$ that is a derivation at $p$, which means that, for all $f, g \in C^{\infty}(\cM)$, $v$ satisfies the product rule
$$
v(f g)=f(p) v g+g(p) v f .
$$
The set of all tangent vectors at $p$ is denoted by $T_p \cM$ and called the \bfemph{tangent space} at $p$.

Suppose $\cM$ is $d$-dimensional and $\varphi: U \rightarrow \widetilde{U} \subseteq \mathbb{R}^d$ is a smooth coordinate chart on some open subset $U \subseteq \cM$. Writing the coordinate functions of $\varphi$ as $\left(x^{(1)}, \cdots, x^{(n)}\right)$, we define the coordinate vectors $\partial /\left.\partial x^{(1)}\right|_p, \cdots, \partial /\left.\partial x^{(n)}\right|_p$ by
$$
\left.\frac{\partial}{\partial x^{(i)}}\right|_p f=\left.\frac{\partial}{\partial x^{(i)}}\right|_{\varphi(p)}\left(f \circ \varphi^{-1}\right) .
$$
These vectors form a basis for $T_p \cM$, which therefore has dimension $d$. Thus, once a smooth coordinate chart has been chosen, every tangent vector $\vec v \in T_p M$ can be written uniquely in the form
$$
\vec v=v^{(1)} \partial /\left.\partial x^{(1)}\right|_p+\cdots+ v^{(n)} \partial /\left.\partial x^{(n)}\right|_p.
$$


If $F: \cM \rightarrow \cN$ is a smooth map and $p$ is any point in $\cM$, we define a linear map $d F_p: T_p \cM \rightarrow T_{F(p)} \cN$, called the \bfemph{differential of $\boldsymbol{F}$ at $\boldsymbol{p}$}, with
$$
d F_p(v) f=v(f \circ F), \quad v \in T_p \cM .
$$
Once we have chosen local coordinates $\left(x^{(i)}\right)$ for $\cM$ and $\left(y^{(j)}\right)$ for $\cN$, we find, by unwinding the definitions, that the coordinate representation of the differential map is given by the \bfemph{Jacobian matrix} of the coordinate representation of $F$, which is its matrix of first-order partial derivatives:
$$
d F_p\left(\left.v^{(i)} \frac{\partial}{\partial x^{(i)}}\right|_p\right)=\left.\frac{\partial \widetilde{F}^{(j)}}{\partial x^{(i)}}(p) v^{(i)} \frac{\partial}{\partial y^{(j)}}\right|_{F(p)} .
$$

\subsubsection{Submanifolds}

The theory system of submanifolds is established on the inverse function theorem and its corollaries.

\begin{theorem}[\bfemph{Inverse Function Theorem for Manifolds}, Thm. 4.5 of \cite{LeeSM}]
    Suppose $\cM$ and $\cN$ are smooth manifolds and $F: \cM \rightarrow \cN$ is a smooth map. If the linear map $d F_p$ is invertible at some point $p \in \cM$, then there exist connected neighborhoods $U_0$ of $p$ and $V_0$ of $F(p)$ such that $\left.F\right|_{U_0}: U_0 \rightarrow V_0$ is a diffeomorphism.
\end{theorem}
The most useful consequence of the inverse function theorem is that a smooth map $F: \cM \rightarrow \cN$ is said to have \bfemph{constant rank} if the linear map $d F_p$ has the same rank at every point $p \in \cM$.
\begin{theorem}[\bfemph{Rank Theorem}, Thm. 4.12 of \cite{LeeSM}]
    Suppose $\cM$ and $\cN$ are smooth manifolds of dimensions $m$ and $n$, respectively, and $F: \cM \rightarrow \cN$ is a smooth map with constant rank $r$. For each $p \in \cM$ there exist smooth charts $(U, \varphi)$ for $\cM$ centered at $p$ and $(V, \psi)$ for $\cN$ centered at $F(p)$ such that $F(U) \subseteq V$, in which $F$ has a coordinate representation of the form
    $$
    \widetilde{F}\left(x^{(1)}, \cdots, x^{(r)}, x^{(r+1)}, \cdots, x^{(m)}\right)=\left(x^{(1)}, \cdots, x^{(r)}, 0, \cdots, 0\right)
    $$
\end{theorem}

The most important types of constant-rank maps are listed below. In all of these definitions, $\cM$ and $\cN$ are smooth manifolds, and $F: \cM \rightarrow \cN$ is a smooth map.
\begin{itemize}
    \item $F$ is a \bfemph{submersion} if its differential is surjective at each point, or equivalently if it has constant rank equal to $\operatorname{dim} \cN$.
    \item $F$ is an \bfemph{immersion} if its differential is injective at each point, or equivalently if it has constant rank equal to $\operatorname{dim} \cM$.
    \item $F$ is a \bfemph{local diffeomorphism} if every point $p \in \cM$ has a neighborhood $U$ such that $\left.F\right|_U$ is a diffeomorphism onto an open subset of $\cN$, or equivalently if $F$ is both a submersion and an immersion.
    \item $F$ is a \bfemph{smooth embedding} if it is an injective immersion that is also a topological embedding (a homeomorphism onto its image, endowed with the subspace topology).
\end{itemize}

\begin{remark}[Prop. 5.5 of \cite{LeeSM}]
    If $\cM$ is a smooth manifold, then an embedded submanifold $\cN \subseteq \cM$ is properly embedded if and only if it is a closed subset of $\cM$.
\end{remark}

Most submanifolds are presented in the following manner. Suppose $\Phi: \cM \rightarrow \cN$ is any map. Every subset of the form $\Phi^{-1}(\{y\}) \subseteq \cM$ for some $y \in \cN$ is called a \bfemph{level set} of $\Phi$, or the \bfemph{fiber} of $\Phi$ over $y$. The simpler notation $\Phi^{-1}(y)$ is also used for a level set when there is no likelihood of ambiguity. Let the \bfemph{codimension of $\cN$} be the difference $\operatorname{dim}\cN - \operatorname{dim}\cM$.

\begin{theorem}[\bfemph{Constant-Rank Level Set Theorem}, Thm. 5.12 of \cite{LeeSM}]
    Suppose $\cM$ and $\cN$ are smooth manifolds, and $\Phi: \cM \rightarrow \cN$ is a smooth map with constant rank $r$. Every level set of $\Phi$ is a properly embedded submanifold of codimension $r$ in $\cM$
\end{theorem}

\begin{corollary}[\bfemph{Submersion Level Set Theorem}, Cor. 5.13 of \cite{LeeSM}]
     Suppose $\cM$ and $N$ are smooth manifolds, and $\Phi: \cM \rightarrow \cN$ is a smooth submersion. Every level set of $\Phi$ is a properly embedded submanifold of $\cM$, whose codimension is equal to $\operatorname{dim} N$.
\end{corollary}

In fact, a map does not have to be a submersion, or even to have constant rank, for its level sets to be embedded submanifolds. If $\Phi: \cM \rightarrow \cN$ is a smooth map, a point $p \in \cM$ is called a \bfemph{regular point} of $\Phi$ if the linear map $d \Phi_p: T_p \cM \rightarrow T_{\Phi(p)} \cN$ is surjective, and $p$ is called a \bfemph{critical point} of $\Phi$ if it is not. A point $c \in \cN$ is called a \bfemph{regular value} of $\Phi$ if every point of $\Phi^{-1}(c)$ is a regular point of $\Phi$, and a \bfemph{critical value} otherwise. A level set $\Phi^{-1}(c)$ is called a \bfemph{regular level set} of $\Phi$ if $c$ is a regular value of $\Phi$.

\begin{corollary}[\bfemph{Regular Level Set Theorem}, Cor. 5.14 of \cite{LeeSM}]
    Let $\cM$ and $\cN$ be smooth manifolds, and let $\Phi: \cM \rightarrow \cN$ be a smooth map. Every regular level set of $\Phi$ is a properly embedded submanifold of $\cM$ whose codimension is equal to $\operatorname{dim} \cN$.
\end{corollary}

\subsection{Riemannian manifold}

There are many important geometric concepts in Euclidean space, such as length and angle, which are derived from inner product. To extend these geometric ideas to abstract smooth manifolds, we need a structure that amounts to a smoothly varying choice of inner product on each tangent space.

Let $\cM$ be a smooth manifold. A \bfemph{Riemannian metric} on $\cM$ is a collection of inner products, whose element at $p\cM$ is an inner product $g_p: T_p\cM\times T_p\cM \to \bR$ that varies smoothly with respect to $p$. A \bfemph{Riemannian manifold} is a pair $(\cM, g)$, where $\cM$ is a smooth manifold and $g$ is a specific choice of Riemannian metric on $\cM$. If $\cM$ is understood to be endowed with a specific Riemannian metric, a conventional statement often used is “$\cM$ is a Riemannian manifold.” In the following sections, we assume $(\cM, g)$ is an oriented Riemannian $d$-manifold.

Another important construction provided by a metric on an oriented manifold is a canonical volume form. For $(\cM, g)$, there is a unique $d$-form $d V_g$ on $\cM$, called the \bfemph{Riemannian volume form}, characterized by
$$
d V_g=\sqrt{\operatorname{det}\left(g_{i j}\right)} d x^{(1)} \wedge \cdots \wedge d x^{(d)} ,
$$
where the $d x^{(i)}$ are $1$-forms from any oriented local coordinates. Here, $\operatorname{det}\left(g_{i j}\right)$ is the absolute value of the determinant of the matrix representation of the metric tensor on the manifold.
The Riemannian volume form allows us to integrate functions on an oriented Riemannian manifold. Let $f$ be a continuous, compactly supported real-valued function on $(\cM, g)$. Then, $f d V_g$ is a compactly supported $d$-form. Therefore, the integral $\int_\cM f d V_g$ makes sense, and we define it as the \bfemph{integral of $\boldsymbol{f}$ over $\boldsymbol{\cM}$}. Similarly, we can define probability measures on $\cM$, and if $\cM$ is compact, the \bfemph{volume} of $\boldsymbol{\cM}$ can be evaluated as
$$
\operatorname{Vol}(\cM)=\int_\cM d V_g=\int_\cM 1 d V_g.
$$

A \bfemph{curve} in $\cM$ usually means a \bfemph{parametrized curve}, namely a continuous map $\gamma: I \rightarrow \cM$, where $I \subseteq \mathbb{R}$ is some interval. To say that $\gamma$ is a \bfemph{smooth curve} is to say that it is smooth as a map from $I$ to $M$. A smooth curve $\gamma: I \rightarrow \cM$ has a \bfemph{well-defined velocity} $\gamma^{\prime}(t) \in T_{\gamma(t)} \cM$ for each $t \in I$. We say that $\gamma$ is a \bfemph{regular curve} if $\gamma^{\prime}(t) \neq 0$ for $t \in I$. This implies that the image of $\gamma$ has no “corners” or “kinks.” For brevity, we refer to a \bfemph{piecewise regular curve segment} $\gamma:[a, b] \rightarrow \cM$ as an \bfemph{admissible curve}, and any partition $\left(a_0, \cdots, a_k\right)$ such that $\left.\gamma\right|_{\left[a_{i-1}, a_i\right]}$ is smooth for each $i$ as an \bfemph{admissible partition} for $\gamma$. If $\gamma$ is an admissible curve, we define \bfemph{the length of $\boldsymbol{\gamma}$} as
$$
L_g(\gamma)=\int_a^b\left|\gamma^{\prime}(t)\right|_g d t.
$$

The \bfemph{speed} of $\gamma$ at any time $t \in I$ is defined as the scalar $\left|\gamma^{\prime}(t)\right|$. We say that $\gamma$ is a \bfemph{unit-speed curve} if $\left|\gamma^{\prime}(t)\right|=1$ for all $t$, and a \bfemph{constant-speed curve} if $\left|\gamma^{\prime}(t)\right|$ is constant. If $\gamma:[a, b] \rightarrow M$ is a unit-speed admissible curve, then its \bfemph{arc-length function} has the simple form $s(t)=t-a$. For this reason, a unit-speed admissible curve whose parameter interval is of the form $[0, b]$ is said to be \bfemph{parametrized by arc-length}.

For each pair of points $p, q \in \cM$, we define the \bfemph{Riemannian distance} from $p$ to $q$, denoted by $d_\cM(p, q)$, as the infimum of the lengths of all admissible curves from $p$ to $q$.
When $\cM$ is connected, we say an admissible curve $\gamma$ is a \bfemph{minimizing curve} if and only if $L_g(\gamma)$ is equal to the distance between its endpoints. A unit-speed minimizing curve is also called a \bfemph{geodesic}. Thus, we use \bfemph{geodesic distance} and Riemannian distance interchangeably.

\begin{theorem}[Existence and Uniqueness of Geodesics, Thm 4.27 of \cite{LeeRM}]
    For every $p \in \cM, w \in T_p \cM$, and $t_0 \in \mathbb{R}$, there exist an open interval $I \subseteq \mathbb{R}$ containing $t_0$ and a geodesic $\gamma: I \rightarrow \cM$ satisfying $\gamma\left(t_0\right)=p$ and $\gamma^{\prime}\left(t_0\right)=w$. Any two such geodesics agree on their common domain.
\end{theorem}
A geodesic $\gamma: I \rightarrow \cM$ is said to be maximal if it cannot be extended to a geodesic on a larger interval. A geodesic segment is a geodesic whose domain is a compact interval. For each $p \in \cM$ , the \bfemph{(restricted) exponential map at $\boldsymbol{p}$}, denoted by $\exp_p$, is defined by
$$
\exp_p (v)=\gamma_v(1),
$$
where $v\in T_p\cM$ and $\gamma_v$ are the unique geodesic with initial location $\gamma_v(0) = p$ and $\gamma^\prime_v = v$. The exponential map is a diffeomorphism in a neighborhood of the tangent space. Similarly, we define the \bfemph{logarithm map} $\log_p$ as the inverse of $\exp_p$. The injectivity radius of $\cM$ at $p$, denoted by $\operatorname{inj}(p)$, is the supremum of all $r>0$ such that $\exp _p$ is a diffeomorphism from $\cB(0,r) \subseteq T_p \cM$ onto its image.

\subsection{Other concepts}
\begin{definition}
    [Normal matrices] A matrix square matrix $A$ is normal when $A A^{*}=A^{*} A$, where $A^{*}$ is its conjugate-transpose. This is equivalent to saying that there exists a unitary matrix $U$ such that $U A U^{*}$ is diagonal (and the diagonal elements are precisely the eigenvalues of $A$). Every Hermitian and every unitary matrix is normal.
\end{definition}

\begin{definition}
    [Trace norm] The trace norm is defined for every $A$ by
    $$
    \|A\|_F^{2}:=\operatorname{Tr}\left(A A^{*}\right)=\operatorname{Tr}\left(A^{*} A\right)=\sum_{1 \leq i, j \leq n}\left|A_{i, j}\right|^{2} .
    $$
    This is also known as the Frobenius, Schur, or Hilbert–Schmidt norm.
\end{definition}

\begin{definition}
    [Principal angles] Suppose $\mathcal A$ and $\mathcal B$ are two vector spaces; we call each
    $$\theta_i(\mathcal A,\mathcal B) = \arccos(\lambda_i(\mathcal A,\mathcal B))$$
    the $i$-th principal angle between $\mathcal A$ and $\mathcal B$, where $\lambda_i(\mathcal A,\mathcal B)$ is the $i$-th largest eigenvalue of $\mathcal A^T \mathcal B$.  Let $\Theta(\mathcal A,\mathcal B)$ denote the diagonal matrix whose $i$-th diagonal entry is $\theta_i(\mathcal A,\mathcal B)$, and let $\sin \Theta(\mathcal A,\mathcal B)$ be performed entrywise, i.e.,
    $$\sin \Theta(\mathcal A,\mathcal B) := diag\left(\sin \theta_i (\mathcal A,\mathcal B)\right). $$
\end{definition}

\section{Proof omitted from the main text}
\subsection{Some useful lemmas and corollaries}
\begin{lemma}[Chernoff bound]
    The generic Chernoff bound for a random variable $X$ is attained by applying Markov's inequality to $e^{tX}$. For every $t>0$, there is
    $$\bP(X \geq a)=\bP\left(e^{t X} \geq e^{t a}\right) \leq \frac{\bE\left(e^{t X}\right)}{e^{t a}}.$$
    Since the inequality holds for every $t>0$, we have
    $$\bP(X \geq a)\leq \inf_{t>0}\frac{\bE\left(e^{t X}\right)}{e^{t a}}.$$
\end{lemma}  
\begin{corollary}\label{col:Chi_Chernoff_bound}
    Let $\xi\sim N(0,\sigma^2 I_D)$ be a $D$-dimensional normal random vector with mean $0$ and covariance matrix $\sigma^2 I_D$. According to the Chernoff bound,
    $$
    \bP(\|\xi\|_2\geq t) \leq \left(\frac{t^2}{D\sigma^2}\exp\left(1 - \frac{t^2}{D\sigma^2}\right)\right)^{D/2}
    $$
    for $t\geq \sqrt{D}\sigma$.
\end{corollary}
\begin{corollary}\label{col:Bi_Chernoff_bound}
    Let $n\sim Bino(N,p)$ be a binomial random variable with size $N$ and probability $p$. According to the Chernoff bound,
    $$
    \bP\left(\frac{n}{N} \geq p+\epsilon\right) \leq \exp\left\{-N \mathcal{D}_{KL}\left(p+\epsilon\|p\right)\right\},
    $$
    $$
    \bP\left(\frac{n}{N} \leq p-\epsilon\right) \leq \exp\left\{-N \mathcal{D}_{KL}\left(p-\epsilon\|p\right)\right\},
    $$
    for $\epsilon>0$, where
    $$
    \mathcal{D}_{KL}(a\|b) = a\log(\frac{a}{b}) + (1-a)\log(\frac{1-a}{1-b})
    $$
    denotes the Kullback–Leibler divergence between Bernoulli distributions $Be(a)$ and $Be(b)$.
\end{corollary}

\begin{lemma}
\label{Thm:Concentration_with_Sample}
    Assume there is a sequence of observed points $\{y_i\}_{i=1}^n$, with a series of weights $W(y_1),\cdots,W(y_1)$. Let the local moving weighted average be
    $$
        \widehat{\mu}_n = \frac{\sum_{i=1}^n W(y_i) y_i}{\sum_{i=1}^n W(y_i)}.
    $$
    Then, if $\{y:W(y)>0\} \subset \cB_D(z,r)$,
    $$
        \sqrt{n} \left(\widehat{\mu}_n - \widehat{\mu}_w\right) \overset{d}{\to} N\left(0,\frac{\Sigma}{\bE(W)^2}\right),
    $$
    with $\Sigma\leq r^2 I_D$ and  $\widehat{\mu}_w = {\bE(WY)}/{\bE(W)}$.
\end{lemma}
\begin{proof}
    According to the central limit theorem and the law of large numbers,
    $$\frac{\sum_{i=1}^n w_i}{n}\overset{a.s.}{\to}\bE(W),$$
    $$\sqrt{n}\left(\frac{\sum_{i=1}^n w_i y_i}{n} - \bE(WY)\right)\overset{d}{\to} N(0,\Sigma),$$
    where $\Sigma\leq r^2 I_D$. Thus, 
    $$\sqrt{n} \left(\widehat{\mu}_n - \frac{\bE(WY)}{\bE(W)}\right) \overset{d}{\to} N\left(0,\frac{\Sigma}{\bE(W)^2}\right).$$
\end{proof}
\begin{corollary}
\label{col:Concentrate_with_n}
    In the case of $n = CD\sigma^{-3}$, with $\sigma$ sufficiently small,
    $$\bP(\|\widehat{\mu}_n - \widehat{\mu}_w\|_2\leq c\sigma^2) \geq 1 - C_1\sigma^{c_1-1}\exp\left(-C_2 \sigma^{c_1-1}\right),$$
    for some constant $C_1$, $C_2$, and any $c_1\in (0,1)$.
\end{corollary}
\begin{proof}
    According to Corollary \ref{col:Chi_Chernoff_bound}, when $\sigma$ is sufficiently small,
    \begin{align*}
        \bP(\|\widehat{\mu}_n &- \widehat{\mu}_w\|_2\leq c\sigma^2) \geq \bP( \frac{r}{\sqrt{n}}\sqrt{\chi}\leq c\sigma^2)\\
        &\geq 1 - \left(\frac{c}{D}\frac{n\sigma^2}{\log(1/\sigma)}\exp\left\{1 - \frac{c}{D}\frac{n\sigma^2}{\log(1/\sigma)}\right\}\right)^{D/2},
    \end{align*}
    for $n\geq \frac{D}{c}\sigma^{-2}\log(1/\sigma)$. Thus, in the case of $n = CD\sigma^{-3}$, the probability is close to 1.
\end{proof}

\subsection{Proof of content in Section 2}
\subsubsection{Proof of Lemma \ref{Lemma:prob_in_a_ball}}
\begin{proof}
    Recall that, in our model, $Y = X + \xi$, with $X\sim\omega(\cM)$ and $\xi\sim N(0,\sigma^2 I_D)$.
    We first check the Chernoff bound for the noise term, which is
    \begin{align*}
        \bP(\|\xi\|_2 \geq c_1r)
        &\leq \left(\frac{c_1^2r^2}{D\sigma^2}\exp\left\{1 - \frac{c_1^2r^2}{D\sigma^2}\right\}\right)^{D/2}\\
        &= \left(\frac{c_1^2 C^2 2d}{D} \log(1/\sigma) \exp\left\{1 - \frac{c_1^2 C^2 2d}{D} \log(1/\sigma)\right\}\right)^{D/2}\\
        & = c_2 \left(\log(1/\sigma)\right)^{D/2} \sigma^{c_1^2C^2d}\\
        &\leq c_2 r^d,
    \end{align*}
    where the first inequality comes from the Chernoff bound, while the last one occurs because $\sigma$ is sufficiently small.
    
    Then, for $\bP(Y\in \cB_D(z,r))$, on one hand,
    \begin{align*}
        \bP(Y\in \cB_D(z,r)) &\geq \bP(\|\xi\|_2\leq c_1r)\bP(X\in \cM\cap \cB_D(z,(1-c_1)r))\\
        &\geq (1-c_2 r^d) \frac{vol (\cM\cap \cB_D(z,(1-c_1)r))}{vol (\cM)}\\
        &\geq c_3 r^d.
    \end{align*}
    On the other,
    \begin{align*}
        \bP(Y\in \cB_D(z,r))&=\bP(X\in \cM\cap \cB_D(z,C_2r),\|Y-z\|_2\leq r) \\
        &\quad + \bP(X\notin \cM\cap \cB_D(z,C_2r),\|Y-z\|_2\leq r)\\
        &\leq\bP(X\in \cM\cap \cB_D(z,C_2r))+\bP(\|\xi\|_2\geq (C_2-1) r)\\
        &\leq \frac{vol (\cM\cap \cB_D(z,(1-c_1)r))}{vol (\cM)} + c_4 r^d\\
        &\leq c_5 r^d.
    \end{align*}

    Therefore, $\bP(Y\in \cB_D(z,r)) = c r^d$ for some constant $c$.
\end{proof}
\subsubsection{Proof of Corollary \ref{Col:Local_sample_size}}
\begin{proof}
    The number of points $n$ can be viewed as a binomial random variable with size $N$ and probability parameter $p = cr^d$. For any $c_1 \in (0,1)$, according to Corollary \ref{col:Bi_Chernoff_bound}, 
    \begin{align*}
        \bP\left(\frac{n}{N} \leq (1-c_1)p\right)
        &\leq \exp\left\{-N\left((1-c_1)p\log(1-c_1)\right)\right\} \\
        &\quad \times \exp\left\{-N\left((1-(1-c_1)p)\log(\frac{1-(1-c_1)p}{1-p})\right)\right\}\\
        &\leq \exp\left( -C_1\sigma^{-3}\right),
    \end{align*}
    \begin{align*}
        \bP\left(\frac{n}{N} \geq (1+c_1)p\right)
        &\leq \exp\left\{-N\left((1+c_1)p\log(1+c_1)\right)\right\} \\
        &\quad \times \exp\left\{-N\left((1-(1+c_1)p)\log(\frac{1-(1+c_1)p}{1-p})\right)\right\}\\
        &\leq \exp\left( -C_2\sigma^{-3}\right).
    \end{align*}
    Therefore,
    $$\bP(C_3 D\sigma^{-3}\leq n \leq C_4 D\sigma^{-3})\geq 1 - 2\exp\left( -C_5\sigma^{-3}\right).$$
    When $\sigma$ is sufficiently small, the probability will be close to $1$.
\end{proof}

\subsubsection{Proof of Proposition \ref{Prop:TruncateNormalConcentration}}
\begin{figure}[htbp]
    \centering
    \hspace{20pt}
    \subfigure[]{
        \begin{tikzpicture}
        \draw [->] (0,-2.5) -- (0,2.5) node[right]{$\bR^{D-1}$};        
        \draw [->] (-2,0) -- (3.5,0) node[above right]{$\bR^{1}$};
        \draw [->] (0.5176,0) -- (1.9316,1.414) node[below=15pt, left=2pt]{$r$};
        \draw (0.5176, 0) circle (2);
        \draw[pattern=north west lines, pattern color=blue] (0 ,-1.93185) arc (-75:75:2);
        \draw[pattern=north east lines, pattern color=red] (0 ,1.93185) arc (105:255:2); 
        \fill (0,0) node[below=8pt, left=2pt,fill=white]{$O$} circle (1pt);
        \fill (0.5176,0) node[below=2pt,fill=white]{$\Delta$} circle (1pt);
    \end{tikzpicture}
    \hspace{40pt}
    }\subfigure[]{
        \begin{tikzpicture}
        \draw [->] (0,-2.5) -- (0,2.5) node[right]{$\bR^{D-1}$};        
        \draw [->] (-2,0) -- (3.5,0) node[above right]{$\bR^{1}$};
        \draw (0.5176, 0) circle (2);
        \draw [densely dashed] (0, 0) circle (1.4824);
        \draw (0 ,-1.93185) arc (-75:75:2);
        \draw (0 ,1.93185) arc (105:255:2); 
        \draw [orange] (-1.4824,-2) rectangle (2.5176,2) node[below left]{$V_1$};
        \draw [cyan] (-1.0482,-1.0482) rectangle (1.0482,1.0482) node[below left]{$V_2$};
        \draw [magenta] (-0.8964,-1.414) rectangle (1.9316,1.414) node[below left]{$V_3$};
    \end{tikzpicture}
    }
    
    \caption{Illustration of the integral region in the proof of Proposition \ref{Prop:TruncateNormalConcentration}: (a) The region of calculating the conditional expectation $\bE(\xi | \xi \in \cB_D(\Delta U, r))$, where the two shaded parts cancel each other out; (b) Three multidimensional cubes designed for bounding the expectation.}
    \label{Fig:IllustTruncated}
\end{figure}
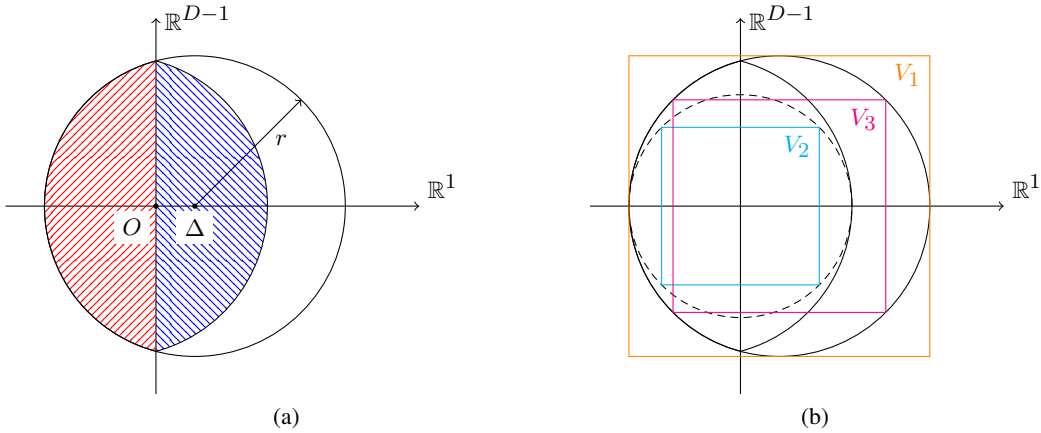
\begin{proof}
    Without loss of generality, we adjust the Cartesian-coordinate system such that $z = (\Delta,0,\cdots,0)$ and $\xi = (\xi^{(1)},\cdots,\xi^{(D)})$, with $\Delta = \|z\|_2\leq C_1\sigma$ for some constant $C_1$. As illustrated in Fig.\ref{Fig:IllustTruncated}, a large part of the calculating region is canceled, and the expectations can be bounded through three integrations on multidimensional cubes. That is, 
    \begin{align*}
        V_1 &= [\Delta-r,~\Delta+r]\times[-r,~r]^{D-1},\\
        V_2 &= [\frac{\Delta-r}{\sqrt{D}},~\frac{r-\Delta}{\sqrt{D}}]^{D},\\
        V_3 &= [\Delta-\frac{r}{\sqrt{D}},~\Delta+\frac{r}{\sqrt{D}}]\times[-\frac{r}{\sqrt{D}},~\frac{r}{\sqrt{D}}]^{D-1}.
    \end{align*}
    To bound the distance between $\bE(\xi|\xi\in\cB_D(z,r))$ and the origin, let 
    \begin{align*}
        &I_1 = \int_{V_1} |\xi^{(1)}| (2\pi\sigma^2)^{-\frac{D}{2}} \exp\{-\frac{\|\xi\|_2^2}{2\sigma^2}\} d\xi,\\
        &I_2 = \int_{V_2} |\xi^{(1)}| (2\pi\sigma^2)^{-\frac{D}{2}} \exp\{-\frac{\|\xi\|_2^2}{2\sigma^2}\} d\xi, \\
        &I_3 = \int_{V_3}(2\pi\sigma^2)^{-\frac{D}{2}} \exp\{-\frac{\|\xi\|_2^2}{2\sigma^2}\} d\xi.
    \end{align*}
    Because of the symmetry,
    \begin{align*}
        \|\bE(\xi|\xi\in\cB_D(z,r))\|_2
        &= \frac{\int_{\cB_D(z,r)} \xi^{(1)} (2\pi\sigma^2)^{-\frac{D}{2}} \exp\{-\frac{\|\xi\|_2^2}{2\sigma^2}\} d\xi}{\int_{\cB_D(z,r)} (2\pi\sigma^2)^{-\frac{D}{2}} \exp\{-\frac{\|\xi\|_2^2}{2\sigma^2}\} d\xi}\\
        &\leq \frac{I_1 - I_2}{I_3}.
    \end{align*}
    For simplicity, denote
    \begin{align*}
        {2^{D-1}}(2\pi\sigma^2)^{\frac{D}{2}}I_1 
        &= \int_{\Delta - r}^{\Delta + r} |t| \exp\{ - \frac{t^2}{2\sigma^2}\}dt\left(\int_{0}^{r}\exp\{-\frac{s^2}{2\sigma^2}\}ds\right)^{D-1} \\
        &:= I_1^+ + I_1^-,\\
        {2^{D-1}}(2\pi\sigma^2)^{\frac{D}{2}}I_2 
        &= \int_{\frac{\Delta-r}{\sqrt{D}}}^{\frac{r-\Delta}{\sqrt{D}}} |t| \exp\{ - \frac{t^2}{2\sigma^2}\}dt\left(\int_{0}^{\frac{r-\Delta}{\sqrt{D}}}\exp\{-\frac{s^2}{2\sigma^2}\}ds\right)^{D-1}\\
        &:= I_2^+ + I_2^-.
    \end{align*}
    Meanwhile, let 
    $$a = \int_{0}^{\frac{r-\Delta}{\sqrt{D}}} t \exp\{ - \frac{t^2}{2\sigma^2}\}dt, \quad \delta_a = \int_{\frac{r-\Delta}{\sqrt{D}}}^{r + \Delta} t \exp\{ - \frac{t^2}{2\sigma^2}\}dt,
    $$
    $$b = \int_{0}^{\frac{r-\Delta}{\sqrt{D}}}\exp\{-\frac{s^2}{2\sigma^2}\}ds, \quad \delta_b = \int_{\frac{r-\Delta}{\sqrt{D}}}^{r}\exp\{-\frac{s^2}{2\sigma^2}\}ds.
    $$
    Then, there is
    $$
    a = \sigma^2\left(1-\exp\{-\frac{Cr^2}{2\sigma^2}\}\right)<\sigma^2, \quad  b = \sigma(\Phi(C) - \Phi(0)) <C\sigma,
    $$
    $$\delta_a = \sigma^2\left(\exp\left\{-\frac{(\frac{r-\Delta}{\sqrt{D}})^2}{2\sigma^2}\right\}-\exp\left\{-\frac{(r + \Delta)^2}{2\sigma^2}\right\}\right)<C\sigma^3,
    $$
    $$
   \delta_b <(r- \frac{r-\Delta}{\sqrt{D}})\exp\left\{-\frac{(\frac{r-\Delta}{\sqrt{D}})^2}{2\sigma^2}\right\} = C\sigma^2\log(1/\sigma)<C\sigma.
    $$
    Furthermore, we can obtain
    \begin{align*}
        I_1^+ - I_2^+ & = (a + \delta_a)(b + \delta_b)^{D-1}  - ab^{D-1}\\ 
        &= a(b + \delta_b)^{D-1} - ab^{D-1} + \delta_a(b + \delta_b)^{D-1}\\ 
        &= \delta_a(b + \delta_b)^{D-1} + a((b + \delta_b)^{D-1} - b^{D-1})\\
        &< C\sigma^{D+2} + a\delta_b\left\{(b+\delta_b)^{D-2} + (b+\delta_b)^{D-3}b + \cdots + (b+\delta_b)b^{D-3} + b^{D-2} \right\}\\
        &< C\sigma^{D+2}.
    \end{align*}
    Thus, $I_1 - I_2 < C\sigma^{-D}(I_1^+ - I_2^+)=C\sigma^2$.
    
    Additionally, it is clear that $I_3>C$, and hence,
    $$\|\bE(\xi|\xi\in\cB_D(z,r))\|_2\leq C\sigma^2.$$

\end{proof}

\subsubsection{Proof of Lemma \ref{Lemma:PartManifold}}
\begin{proof}
    Let $\nu_c(y) =\int_{\cM\backslash\cM_R} \phi_\sigma(y-x)\omega(x)dx$; then, according to the model setting,
    $$\nu_c(y) =  \nu(y) - \nu_R(y) \geq 0.$$
    The probability measure within $\cB_D(z,r)$ is proportional to $\nu(y)$, which can be expressed as 
    \begin{align*}
        \tilde{\nu}(y) &= \frac{\nu(y)}{\int_{\cB_D(z,r)} \nu(y) dy}\\
        &= \frac{\nu_R(y) + \nu_c(y)}{\int_{\cB_D(z,r)} \nu_R(y) dy + \int_{\cB_D(z,r)} \nu_c(y) dy}
    \end{align*}
    for $y \in \cB_D(z,r)$. Then, the relative difference between $\tilde{\nu}(y)$ and $\tilde{\nu}_R(y)$ can be evaluated as 
    \begin{align*}
        \left|\tilde{\nu}(y) - \tilde{\nu}_R(y)\right|
        &=\left|\frac{\nu_R(y) + \nu_c(y)}{\int_{\cB_D(z,r)} \nu_R(y) dy + \int_{\cB_D(z,r)} \nu_c(y) dy} - \frac{\nu_R(y)}{\int_{\cB_D(z,r)} \nu_R(y) dy}\right|\\
        &\leq \left| \frac{\nu_c(y)}{\nu_R(y)} - \frac{\int_{\cB_D(z,r)}\nu_c(y)dy}{\int_{\cB_D(z,r)}\nu_R(y)dy}\right| \tilde{\nu}_R(y).
    \end{align*}
    
    Let $R = r + C_1\sigma\sqrt{(d+\eta)\log(1/\sigma)}$ and $R^\prime = r + C_1\sigma < R$; then, 
    \begin{align*}
        \frac{\nu_c(y)}{\nu_R(y)} &= \frac{\int_{\cM\backslash\cM_R} \phi_\sigma(y-x)\omega(x)dy}{ \int_{\cM_R} \phi_\sigma(y-x)\omega(x)dy}\\
        &\leq \frac{\phi_\sigma(R-r)\omega(\cM)}{\phi_\sigma(R^\prime-r)\omega(\cM\cap \cB_D(y,R^\prime))}\\
        &\leq C\frac{\sigma^{d+\eta}V}{vol(\cM\cap \cB_D(y,R^\prime))}\\
        &\leq C\frac{\sigma^{d+\eta}}{(R^\prime)^d}.
    \end{align*}
    Therefore,
    $$\left|\tilde{\nu}(y) - \tilde{\nu}_R(y)\right| \leq C\sigma^\eta \tilde{\nu}_R(y).$$
\end{proof}

\subsection{Proof of content in Section 3}
\subsubsection{Proof of Proposition \ref{Prop:Angle-Cord}}
\begin{proof}
    Recall that $z^*$ is the origin, and $z - z^*$ is the $(d+1)$-th direction in the Cartesian-coordinate system. Then,
    \begin{align*}
        \mu_z^\mathbb{B} &= (\mu^{(1)},\cdots,\mu^{(d)},\mu^{(d+1)},\mu^{(d+2)},\cdots,\mu^{(D)}),\\
        z &= (0,\cdots,0,\Delta,0,\cdots,0),
    \end{align*}
    where $\Delta = \|z-z^*\|\leq c\sigma$. The angle between $\mu_z^\mathbb{B} - z$ and $z^*-z$ can be represented by its sine as follows:
    \begin{align*}
        \sin^2(\Theta(\mu_z^\mathbb{B} - z,\ z^*-z)) &= 1 - \cos^2(\Theta(\mu_z^\mathbb{B} - z,\ z^*-z))\\
        &= 1 - \left(\frac{(\mu_z^\mathbb{B} - z)\cdot(z^*-z)}{\|\mu_z^\mathbb{B} - z\|_2\|z^*-z\|_2}\right)^2\\
        &= \frac{\sum_{i\neq d+1} (\mu^{(i)})^2}{\sum_{i\neq d+1} (\mu^{(i)})^2 + (\mu^{(d+1)}-\Delta)^2}
    \end{align*}
    If
    \begin{equation}        
        \left\{
            \begin{array}{cll}
                |\mu^{(i)}-\Delta| &\geq c_1\sigma, \quad &\text{for } i = d+1;\\
                |\mu^{(i)}|   &\leq c_2 \sigma^2{\sqrt{\log(1/\sigma)}}, \quad &\text{for } i \neq d+1.\\
            \end{array}
        \right.
        \label{eq:Angle-Cord_appendix}
    \end{equation}
    then
    \begin{align*}        
        \sin^2(\Theta(\mu_z^\mathbb{B} - z,\ z^*-z)) 
        \leq \frac{(D-1)C_2^2\sigma^4\log(1/\sigma)}{(D-1)C_2^2\sigma^4\log(1/\sigma) + C_1^2\sigma^2}
        \leq C\sigma^2\log(1/\sigma)
    \end{align*}
    for some constant $C$.
    
    In other words, equation \eqref{eq:Angle-Cord_appendix} is sufficient for $\sin(\Theta(\mu_z^\mathbb{B} - z,\ z^*-z))\leq C\sigma\sqrt{\log(1/\sigma)}$.
\end{proof}

\subsubsection{Proof of Lemma \ref{Lemma:LocalGeneralDiff}}
\begin{proof}
    To prove Lemma \ref{Lemma:LocalGeneralDiff}, the following propositions are needed.

    \begin{proposition}
        \label{Prop:LocalChart}
        Assume there is a mapping $\Psi: \bD \to \cM_R$ that satisfies, for any point $z = (z_1,\cdots,z_d,0,\cdots,0) \in \bD$,
        $$\Psi(z) = (z_1,\cdots,z_d,\psi(z_1,\cdots,z_d)).$$
    \end{proposition}
    
    \begin{proposition}
        \label{Prop:LocalChartJacobian}
        Since the approximation error of the tangent space for the manifold localization is in quadratic order,
        $$
        d(\Psi(z)) = (1+g(z))dz
        $$
        with $|g(z)| < C\|z\|_2$.    
    \end{proposition}

    Let $\delta_z = \Psi(z) - z$ for $z\in\bD$. Then, there is
    \begin{align*}
        \nu_R(y) &= \int_{\cM_R} \phi_\sigma(y - x) \omega(x)dx\\
        &= \int_{\cM_R} \phi_\sigma(y - \Psi(z))\omega(\Psi(z))d\Psi(z)\\
        &= \frac{1}{V}\int_{\bD}\phi_\sigma(y - z - \delta_z) (1+g(z))dz\\
        &= \frac{1}{V}\int_{\bD}\phi_\sigma(y - z - \delta_z) dz + \frac{1}{V}\int_{\bD}\phi_\sigma(y - z - \delta_z) g(z)dz.
    \end{align*}
    The difference between $\nu_{\bD}(y)$ and the first term of $\nu_R(y)$ can be expressed as follows:
    \begin{align*}
        \Bigl|\nu_{\bD}(y) &- \frac{1}{V}\int_{\bD}\phi_\sigma(y - z - \delta_z) dz \Bigl|
        = \left|\frac{1}{V}\int_{\bD} \phi_\sigma(y - z) - \phi_\sigma(y - z - \delta_z) dz\right|\\
        &\leq \frac{1}{V}\int_{\bD}\phi_\sigma(y - z) \left| 1 - \exp\left\{ - \frac{\|x-z-\delta_z\|_2^2 - \|x-z\|_2^2}{2\sigma^2}\right\}\right| dz\\
        &\leq \frac{1}{V}\int_{\bD}\phi_\sigma(y - z) \frac{\left|\|x-z-\delta_z\|_2^2 - \|x-z\|_2^2\right|}{2\sigma^2} dz \\
        &\leq \frac{1}{V}\int_{\bD}\phi_\sigma(y - z) \frac{\left\|\delta_z\right\|_2^2}{2\sigma^2} dz\\
        &\leq C \sigma\nu_{\bD}(y),
    \end{align*}
    for some constant $C$. Moreover, the second term of $\nu_R(y)$ is of higher order:
    \begin{align*}
        \left|\frac{1}{V}\int_{\bD}\phi_\sigma(y - z - \delta_z) g(z)dz\right|
        &\leq \frac{1}{V}\int_{\bD}\phi_\sigma(y - z - \delta_z) \|z\|_2dz\\
        &\leq  C\sigma{\sqrt{\log(1/\sigma)}}\nu_{\bD}(y).
    \end{align*}
\end{proof}

\subsubsection{Proof of Lemma \ref{Lemma:LocalDisc}}
\begin{proof}
    Assume the manifold can be regard as $\bD = T_{z^*}\cM\cap\cB_D(y,R)$ locally with 
    $$R = r + C\sigma\sqrt{\log(1/\sigma)}\gg r.$$
    We still investigate the conditional expectation within $\cB_D(z,r)$, where we use $\nu_\bD(y)$ to denote the density function of $y$ and $\widetilde{\nu}_\bD(y)$ to denote its normalized version in $\cB_D(z,r)$. Similarly, we let $z^*$ be the origin, $z - z^*$ be the $(d+1)$-th direction, and $\mu_z^\mathbb{B} = (\mu^{(1)},\cdots,\mu^{(D)})$. Then, for the $i$-th direction,
    $$
        \mu^{(i)} = \int_{\cB_D(z,r)} y^{(i)} \widetilde{\nu}_\bD(y)\,dy = \frac{\int_{\cB_D(z,r)} y^{(i)} {\nu}_\bD(y)\,dy}{\int_{\cB_D(z,r)}{\nu}_\bD(y)\,dy}.
    $$
    
    \vspace{10pt}
    {\noindent\underline{We first prove that $\mu^{(i)} = 0$ for $i \neq d+1$:}}
    \vspace{5pt}
    
    Since the ball $\cB_D(z,r)$ is centrosymmetric for $i\neq d+1$, there exists a mapping $h_i$ for each $i\neq d+1$ such that, for any $y = (y^{(1)},\cdots,y^{(i)},\cdots, y^{(D)}) \in \cB_D(z,r)$,
    $$h_i: (y^{(1)},\cdots,y^{(i)},\cdots, y^{(D)}) \mapsto (y^{(1)},\cdots,-y^{(i)},\cdots, y^{(D)}),$$
    and $h_i(y)\in \cB_D(z,r)$. That is, for all $y\in\cB_D(z,r)$, $h_i(y)$ is its mirror with respect to the $i$-th direction, and
    $$y\in\cB_D(y_r) \Leftrightarrow h_i(y)\in \cB_D(z,r), \quad \text{for } i\neq d+1,$$
    $$x \in \bD \Rightarrow h_i(x) \in \bD,  \quad \text{for } i= 1,\cdots,d,$$
    $$x \in \bD \Rightarrow h_i(x) = x,  \quad \text{for } i= d+1,\cdots,D.$$
    
    Let $\cB_i^+$ and $\cB_i^-$ be two hemispheres such that
    $$
        \cB_i^+ = \left\{y\in\cB_D(z,r): y^{(i)}>0\right\}, \quad \cB_i^- = \left\{y\in\cB_D(z,r): y^{(i)}<0\right\}.
    $$
    Then,
    \begin{align*}
        \mu^{(i)} 
        &= \int_{\cB_i^+} y^{(i)} \widetilde{\nu}_\bD(y) dy + \int_{\cB_i^-} y^{(i)} \widetilde{\nu}_\bD(y) dy\\
        &= \int_{\cB_i^+} y^{(i)} \widetilde{\nu}_\bD(y) dy + \int_{\cB_i^+} (h_i(y))^{(i)} \widetilde{\nu}_\bD(h_i(y)) d(h_i(y))\\
        &= \int_{\cB_i^+} y^{(i)} (\widetilde{\nu}_\bD(y) -  \widetilde{\nu}_\bD(h_i(y)))dy
    \end{align*}
    To show $\mu^{(i)} = 0$, it is sufficient to show $\widetilde{\nu}_\bD(y) =  \widetilde{\nu}_\bD(h_i(y))$ or ${\nu}_\bD(y) =  {\nu}_\bD(h_i(y))$. Recall that
    $$
        \nu_{\bD}(y) = \int_{\bD} \phi_\sigma(y-x)\omega(x) dx,
    $$
    and
    $$
        \|y-x\|_2 = \|h_i(y)- h_i(x)\|_2, \quad \|h_i(y)-x\|_2 = \|y- h_i(x)\|_2.
    $$
    Therefore, for $i = 1,\cdots,d$,
    \begin{align*}
        {\nu}_\bD(h_i(y)) &= \int_{\bD}  \phi_\sigma(h_i(y) - x) \omega(x) dx\\
        &= \int_{\bD}  \phi_\sigma(y - h_i(x)) \omega(h_i(x)) dh_i(x)\\
        &= \nu_\bD(y),
    \end{align*}
    and, for $i = d+2,\cdots,D$,    
    \begin{align*}
        {\nu}_\bD(y) &= \int_{\bD}  \phi_\sigma(y - x) \omega(x) dx\\
        &= \int_{\bD}  \phi_\sigma(h_i(y) - h_i(x)) \omega(h_i(x)) dh_i(x)\\
        &= \int_{\bD}  \phi_\sigma(h_i(y) - x) \omega(x) dx\\
        &= \nu_\bD(h_i(y)).
    \end{align*}
    Thus,
    $$
        \mu^{(i)} = 0, \quad \text{for } i \neq d+1.
    $$
    
    \vspace{10pt}
    {\noindent\underline{For $i = d+1$, we need to bound $|\Delta - \mu^{(d+1)}|$ below:}}
    \vspace{5pt}
    
    According to Lemma \ref{Lemma:prob_in_a_ball} and our model setting,
    \begin{align*}
        |\Delta - \mu^{(d+1)}| 
        &= \left|\frac{\int_{\cB_D(z,r)} \int_\bD (\Delta -y^{(d+1)}) \phi_\sigma(y-x) \omega(x)\,dx\,dy}{\int_{\cB_D(z,r)} \nu_\bD(y)\,dy}\right|\\
        &= Cr^{-d} \left|\int_{\cB_D(z,r)} \int_\bD (\Delta -y^{(d+1)}) \phi_\sigma(y-x) \omega(x)\,dx\,dy\right|.
    \end{align*}
    If we express the numerator in the form of elements in the Cartesian coordinates,
    \begin{align*}
        &\left|\int_{\cB_D(z,r)} \int_\bD (\Delta -y^{(d+1)}) \phi_\sigma(y-x) \omega(x)\,dx\,dy\right|\\
        &= \left|\int_{\cB_D(z,r)} (\Delta -y^{(d+1)})\phi_\sigma(y^{(d+1)}) \int_\bD  \prod_{j=1}^d \phi_\sigma(y^{(j)}-x^{(j)}) \omega(x) \,dx\,\prod_{j=d+2}^D \phi_\sigma(y^{(j)})\,dy\right|\\
        &\geq C\left|\int_{\cB_D(z,r)} (\Delta -y^{(d+1)})\phi_\sigma(y^{(d+1)}) \prod_{j=d+2}^D \phi_\sigma(y^{(j)})\,dy\right|\\
        &\geq C\int_0^\Delta t(\phi_\sigma(t-\Delta) - \phi_\sigma(t+\Delta))\,dt~\int_{\sum_{j\neq d+1} (y^{(j)})^2\leq r^2 - \Delta^2} \phi_\sigma(y^{(j)})\,dy\\
        &:= CI_1 I_2,
    \end{align*}
    where the last inequality is the result of the cropping of the integral area (similar to that in Fig. \ref{Fig:IllustTruncated}), while the first inequality stems from the fact that
    \begin{equation}
    \label{eq:RemoveDisc_append}
        \begin{aligned}
        \int_\bD  \prod_{j=1}^d \phi_\sigma(y^{(j)}-x^{(j)}) \omega(x) \,dx 
        &\geq \bP\left(\|\xi^\prime\|_2\geq (R-r)| \xi^\prime\sim N(0,\sigma^2 I_d)\right)\\
        &\geq 1 - c \sigma^C\approx 1.
        \end{aligned}
    \end{equation}
    
    If we let $p = (y^{(1)},\cdots,y^{(d)})$ and $q = (y^{(d+2)},\cdots,y^{(D)})$, with $\Delta = C_0\sigma$ and $r \geq \Delta + c_0\sigma$, we have
    $$
    I_1 = \int_{-\Delta}^\Delta t \phi_\sigma(t-\Delta)\,dt
        = \frac{C_0\sqrt{\pi}Erf(C_0) - 2(e^{-C_0^2}-1)}{\sqrt{2\pi}}\sigma,
    $$
    \begin{align*}
        I_2 
        &= \int_{\|p\|^2+\|q\|_2^2\leq r^2 - \Delta^2} \phi_\sigma(q)\,dp\,dq\\
        &\geq \int_{\|p\|^2+\|q\|_2^2\leq c_0^2\sigma^2} \phi_\sigma(q)\,dp\,dq\\
        &= C\sigma^{-(D-d-1)}\int_0^{c_0\sigma} (c_0^2 \sigma^2 - s^2)^{\frac{d}{2}}s^{D-d-2}\exp\left(-\frac{s^2}{2\sigma^2}\right)\,ds\\
        &\geq c\sigma^d.
    \end{align*}
    In other words, when $C_0>0$ and $c_0>0$, we have $I_1\geq c \sigma$, $I_2 \geq c\sigma^d$, and thus
    $$|\Delta - \mu^{(d+1)}|\geq C r^{-d} I_1I_2\geq c \sigma.$$
    
    \vspace{10pt}
    {\noindent\underline{Combining all the results above, we have}}
    $$
    \left\{
    \begin{aligned}
            &|\Delta - \mu_\bD^{(d+1)}| \geq c\sigma\\
            &|\mu_\bD^{(i)}| = 0 , \quad \text{for } i \neq d+1
    \end{aligned}
    \right..
    $$
\end{proof}
\subsubsection{Proof of Theorem \ref{Thm:Angle:F(z)}}
\begin{proof}
    The proof is based on the framework of Lemma \ref{Thm:Concentration_with_Sample} and Corollary \ref{Col:Local_sample_size}. We first provide an estimation of the local sample size, and then show the equivalent property between $\mu^\bB_z$ and $\bE(W(Y)Y)/\bE(W(Y))$.
    
    For simplicity, we let the collection of observation that falls in $\cB_D(z,r_0)$ be $\{y_i\}_{i=1}^n$, with size $n$. According to Corollary \ref{Col:Local_sample_size}, if $N = Cr_0^{-d}\sigma^{-3}$,
    $$\bP(C_3 D\sigma^{-3}\leq n \leq C_4 D\sigma^{-3})\geq 1 - 2\exp\left( -C_5\sigma^{-3}\right).$$
    In the definition of $F(z)$, the weight function is constructed as
    $$W(y) = \left(1 - \frac{\|z-y\|_2^2}{r_0^2}\right)^k.$$
    To obtain the asymptotic distribution, we need to evaluate $\bE(W(Y)Y)$ and $\bE(W(Y))$. Same with the proof of Theorem \ref{Thm:AngleCondition}, we only need to work on $\nu_\bD(y)$ under the same setting of Cartesian coordinates, which means, $z^*$ is the origin, $z - z^*$ is the $(d+1)$-th direction in the Cartesian-coordinate system, and 
    $$z = (\underbrace{0,\cdots,0}_d,\Delta, \underbrace{0,\cdots 0}_{D-d-1}).$$
    If we define $p = (y^{(1)},\cdots,y^{(d)})$, $t=y^{(d+1)}$, and $q = (y^{(d+2)},\cdots,y^{(D)})$, and assume $\eta \in \bR^{2k}$ being an auxiliary vector, there is
    \begin{align*}
        &\bE(W(Y)) = \frac{\int_{\cB_D(z,r_0)}\int_\bD W(y)\phi_\sigma(y-x)\omega(x)\,dx\,dy}{\int_{\cB_D(z,r_0)}\int_\bD \phi_\sigma(y-x)\omega(x)\,dx\,dy}\\
        \approx& cr_0^{-d} \int_{\|p\|_2^2+(t-\Delta)^2+\|q\|_2^2\leq r_0}W(y)\phi_\sigma(t-\Delta)\phi_\sigma(q)\,dp\,dt\,dq\\
        =& cr_0^{-(d+2k)} \int_{\|p\|_2^2+(t-\Delta)^2+\|q\|_2^2\leq r_0}(r_0^2 - \|p\|_2^2-(t-\Delta)^2-\|q\|_2^2)^{k}\\
        &\hspace{180pt}\times\phi_\sigma(t-\Delta)\phi_\sigma(q)\,dp\,dt\,dq\\
        =& cr_0^{-(d+2k)} \int_{\|p\|_2^2+(t-\Delta)^2+\|q\|_2^2+\|\eta\|_2^2\leq r_0}\phi_\sigma(t-\Delta)\phi_\sigma(q)\,d\eta\,dp\,dt\,dq\\
        =& c.
    \end{align*}
    Meanwhile, the $i$-th element of $\bE(W(Y)Y)$ can be expressed as
    \begin{align*}
         &(\bE(W(Y)Y))^{(i)} = \frac{\int_{\cB_D(z,r_0)}\int_\bD W(y) y^{(i)} \phi_\sigma(y-x)\omega(x)\,dx\,dy}{\int_{\cB_D(z,r_0)}\int_\bD \phi_\sigma(y-x)\omega(x)\,dx\,dy}\\
        \approx& cr_0^{-d} \int_{\|p\|_2^2+(t-\Delta)^2+\|q\|_2^2\leq r_0} W(y)y^{(i)} \phi_\sigma(t-\Delta)\phi_\sigma(q)\,dp\,dt\,dq\\
        =& cr_0^{-(d+2k)} \int_{\|p\|_2^2+(t-\Delta)^2+\|q\|_2^2+\|\eta\|_2^2\leq r_0} y^{(i)}\phi_\sigma(t-\Delta)\phi_\sigma(q)\,d\eta\,dp\,dt\,dq
    \end{align*}
    where the two approximation marks are the result of \eqref{eq:RemoveDisc_append}. By introducing the auxiliary vector $\eta$, these two expectations can be viewed as the analogy of our manifold-fitting model in a higher-dimension case where the dimensionalities of the ambient space and latent manifold are $D+2k$ and $d+2k$, respectively. 
    
    Hence, let $\widehat{\mu}_w = \bE(W(Y)Y)/\bE(W(Y))$; then, according to Theorem \ref{Thm:AngleCondition},
    \begin{equation*}        
        \left\{
            \begin{array}{cl}
                |\widehat{\mu}_w^{(d+1)}-\Delta| &\geq c_1\sigma\\
                |\widehat{\mu}_w^{(i)}|   &\leq c_2 \sigma^2, \quad \text{for } i \neq d+1\\
            \end{array}
        \right..
    \end{equation*}
    Combining the result with Corollary \ref{Col:Local_sample_size} and Corollary \ref{col:Concentrate_with_n}, if the total sample size $N = Cr_0^{-d}\sigma^{-3}$, 
    $$\bP(\|F(z) - \widehat{\mu}_w\|_2\leq c\sigma^2) \geq 1 - C_1\sigma^{c_1-1}\exp\left(-C_2 \sigma^{c_1-1}\right),$$
    for some constant $C_1$, $C_2$, and any $c_3\in(0,1)$, and thus
    $$\sin \{\Theta\left(F(z) - z, \ z^*-z\right)\} \leq C_1\sigma\sqrt{\log(1/\sigma)}$$
    with probability at least $1 - C_2\exp(-C_3\sigma^{-c})$, for some constant $c$, $C_1$, $C_2$, and $C_3$.
\end{proof}
\subsection{Proof of content in Section 4}
\subsubsection{Proof of Theorem \ref{Thm:ContractWithEstDir1}}
\begin{proof}
    Assume $\widehat\Pi_{z^*}^\perp$ satisfies 
    $$\|\widehat\Pi_{z^*}^\perp - \Pi_{z^*}^\perp\|_F\leq c \sigma^\kappa,$$
    and the region $\widehat{\bV}_{z}$ is constructed correspondingly. As in the proof of Theorem \ref{Thm:ContractWithKnownDir}, $\widehat \mu_z^{\mathbb V}$ can be written as
    \begin{align*}
        \widehat \mu_z^{\mathbb V} 
        &= z + \widehat\Pi_{z^*}^\perp\bE_{Y\sim \nu}\left(Y-z|Y\in \widehat{\bV}_{z}\right)\\
        &= z^* + \widehat\Pi^-_{z^*}\delta_z + \bE_\nu\left(\widehat\Pi^\perp_{z^*}(X - z^*)|Y\in \widehat{\bV}_{z} \right) + \bE_\nu\left(\widehat\Pi^\perp_{z^*}\xi|Y\in \widehat{\bV}_{z} \right),
    \end{align*}
    which also can be divided into three parts. According to Lemma \ref{Lemma:PartManifold}, we can assume $\|X- z^*\|_2 \leq C\sigma\sqrt{\log(1/\sigma)}$ for some constant $C$. Let $\delta_z = z - z^*$, and the three parts can be evaluated as follows:
    \begin{itemize}
        \item [(a)] $\widehat\Pi^-_{z^*}\delta_z$:
        \vspace{5pt}
        
            The norm of $\widehat\Pi^-_{z^*}\delta_z$ is upper bounded as
            \begin{align*}
                \|\widehat\Pi^-_{z^*}\delta_z\|_2
                & = \| \Pi^-_{z^*}\delta_z + (\widehat\Pi^-_{z^*} - \Pi^-_{z^*})\delta_z \|_2\\ 
                &\leq \|\Pi^-_{z^*}\delta_z\|_2+\|(\widehat\Pi^-_{z^*} - \Pi^-_{z^*})\|_F\|\delta_z\|_2\\
                &\leq 0 + c\sigma^\kappa \sigma\\
                &\leq C\sigma^{1+\kappa},
            \end{align*}
            for some constant $C$.
        \vspace{5pt}
        \item [(b)] $\bE_\nu\left(\widehat\Pi^\perp_{z^*}(X - z^*)|Y\in \widehat{\bV}_{z} \right)$:
        \vspace{5pt}
        
            From Jensen's inequality,
            $$\left\|\bE_\nu\left(\widehat\Pi^\perp_{z^*}(X - z^*)|Y\in \widehat{\bV}_{z} \right)\right\|_2
            \leq \bE_\nu\left(\left\|\widehat\Pi^\perp_{z^*}(X - z^*)\right\|_2|Y\in \widehat{\bV}_{z} \right).$$
            
            Since $z^*$ and $X$ are exactly on $\cM$, according to Lemma \ref{Lemma:ReachCond},
            \begin{align*}
                \left\|\widehat\Pi^\perp_{z^*}(X - z^*)\right\|_2
                &= \left\|\Pi^\perp_{z^*}(X - z^*) + (\widehat\Pi^\perp_{z^*} - \Pi^\perp_{z^*})(X - z^*)\right\|_2\\
                &\leq \frac{1}{2\tau}\|X - z^*\|^2_2 + \sigma^\kappa\|X - z^*\|_2,
            \end{align*}
            where 
            \begin{align*}
                \|X - z^*\|^2 &= \|X - z + z - z^*\|_2^2 \\
                &\leq \|X-z\|^2_2 + \|z-z^*\|^2_2 \\
                &\leq C\sigma^2\log(1/\sigma),
            \end{align*}
            and thus
            $$\bE_\nu\left(\widehat\Pi^\perp_{z^*}(X - z^*)|Y\in \widehat{\bV}_{z} \right) \leq C\sigma^{1+\kappa}\sqrt{\log(1/\sigma)}.$$
        \item [(c)] $\bE_\nu\left(\widehat\Pi^\perp_{z^*}\xi|Y\in \widehat{\bV}_{z} \right)$:
        \vspace{5pt}
        
            The dislocation $\Delta$ can be evaluated as follows:
            \begin{align*}
                \Delta &= \|\widehat\Pi^\perp_{z^*} (z - X)\|_2\\
                &\leq \|\widehat\Pi^\perp_{z^*} (z - z^*)\|_2 + \|\widehat\Pi^\perp_{z^*}( z^* - X) \|_2\\
                &\leq \|z - z^*\|_2 + C\sigma^{1+\kappa}\sqrt{\log(1/\sigma)}\\
                &\leq C\sigma.
            \end{align*}
            Thus, if we let $\xi^\prime = \widehat\Pi^\perp_{z^*}\xi$, according to Proposition \ref{Prop:TruncateNormalConcentration},
            \begin{align*}
                \left\|\bE_\nu\left(\widehat\Pi^\perp_{z^*}\xi|Y\in \widehat{\bV}_{z} \right)\right\|_2
                &\leq \bE_\omega\left(\left\|\bE_{\phi}\left[\Pi^\perp_{z^*}\xi|X,~X+\xi\in \bV_{z} \right]\right\|_2\right)\\
                &\leq \bE_\omega\left(\|\bE\left(\xi^\prime|\xi^\prime \in \cB_{D-d}(a_\Delta,r_2)\right)\|_2\right)\\
                &\leq C\sigma^2
            \end{align*}
           for some constant $C$.
    \end{itemize}
    Therefore, 
    $$\|\widehat\mu_z^{\mathbb V}-z^*\|_2\leq C\sigma^{1+\kappa}\sqrt{\log(1/\sigma)},$$    
    for some constant $C$, and $\widehat\mu_z^{\mathbb V}$ is $\cO(\sigma^{1+\kappa}\sqrt{\log(1/\sigma)})$ to $\cM$.
\end{proof}

\subsubsection{Proof of Theorem \ref{Thm:ContractWithEstDir2}}
\begin{proof}
    Assume $U$ is the projection matrix onto $z-z^*$, and $\widetilde{U}$ is its estimation via $\mu_z^\bB - z$ such that $\|U - \widetilde{U}\|_F\leq C\sigma \sqrt{\log(1/\sigma)}$ for some constant $C$. Let $U^-$ be the complement in $\bR^D$; then, $\widehat \mu_z^\bV$ can be rewritten as follows:
    \begin{align*}
        \widehat \mu_z^\bV &= z + \widetilde{U} \bE_{\nu}\left(Y-z|Y\in \widehat{\bV}_{z}\right)\\
        &=z^* + \widetilde{U}^- \delta_z + \bE_\nu\left(\widetilde{U}(X - z^*)|Y\in \widehat{\bV}_{z} \right) + \bE_\nu\left(\widetilde{U} \xi|Y\in \widehat{\bV}_{z} \right),
    \end{align*}
    which also can be divided into three parts. According to Lemma \ref{Lemma:PartManifold}, we can assume $\|X- z^*\|_2 \leq C\sigma\sqrt{\log(1/\sigma)}$ for some constant $C$. Let $\delta_z = z - z^*$. The three parts can be evaluated as follows:
    \begin{itemize}
        \item [(a)] $\widetilde{U}^- \delta_z$:
        \vspace{5pt}
        
        As $\delta_z$ is orthogonal to the base of $U^-$,
        \begin{align*}
            \|\widetilde{U}^- \delta_z\|_2
            \leq \|U^- \delta_z\|_2 + \|U^- -\widetilde{U}^-\|_F\|\delta_z\|_2
            \leq C\sigma^2\sqrt{\log(1/\sigma)}.
        \end{align*}
        \item [(b)] $\bE_\nu\left(\widetilde{U} (X - z^*)|Y\in \widehat{\bV}_{z} \right)$:
        \vspace{5pt}
        
        Using Jensen's inequality,
        $$\left\|\bE_\nu\left(\widetilde{U}(X - z^*)|Y\in \widehat{\bV}_{z} \right)\right\|_2
        \leq \bE_\nu\left(\left\|\widetilde{U}(X - z^*)\right\|_2|Y\in \widehat{\bV}_{z} \right).$$
        
        Since $U$ is one direction of $\Pi^\perp_{z^*}$,
        \begin{align*}
            \| U (X - z^*)\|_2\leq\|\Pi^\perp_{z^*}(X - z^*)\|_2.
        \end{align*}
        
        As $z^*$ and $X$ are exactly on $\cM$, according to Lemma \ref{Lemma:ReachCond},
            \begin{align*}
                \left\|\widetilde{U}(X - z^*)\right\|
                &= \left\|U(X - z^*) + (\widetilde{U} - U)(X - z^*)\right\|_2\\
                &\leq \left\|\Pi^\perp_{z^*}(X - z^*)\right\|_2 + \|\widetilde{U} - U\|_F\left\|X - z^*\right\|_2\\
                &\leq \frac{1}{2\tau}\|X - z^*\|^2_2 + \sigma\|X - z^*\|_2\\
                &\leq C\sigma^2{\log(1/\sigma)}.
            \end{align*}
        \item [(c)] $\bE_\nu\left(\widetilde{U} \xi|Y\in \widehat{\bV}_{z} \right)$:
        \vspace{5pt}
        
         The dislocation $\Delta$ can be evaluated as
            \begin{align*}
                \Delta &= \|\widetilde{U} (z - X)\|_2\\
                &\leq \|\widetilde{U} (z - z^*)\|_2 + \|\widetilde{U}(z^* - X) \|_2\\
                &\leq \|z - z^*\|_2 + C\sigma^{2}\sqrt{\log(1/\sigma)}\\
                &\leq C\sigma.
            \end{align*}
            Thus, if we let $\xi^\prime = \widetilde{U}\xi$, according to Proposition \ref{Prop:TruncateNormalConcentration},
            \begin{align*}
                \left\|\bE_\nu\left(\widetilde{U}\xi|Y\in \widehat{\bV}_{z} \right)\right\|_2
                &\leq \bE_\omega\left(\left\|\bE_{\phi}\left(\widetilde{U}\xi|X,~X+\xi\in \bV_{z} \right)\right\|_2\right)\\
                &\leq \bE_\omega\left(\|\bE\left(\xi^\prime|\xi^\prime \in \cB_{D-d}(a_\Delta,r_2)\right)\|_2\right)\\
                &\leq C\sigma^2
            \end{align*}
           for some constant $C$.
    \end{itemize}
    Therefore,
    $$\|\widehat \mu_z^\bV-z^*\|_2\leq C\sigma^2{\log(1/\sigma)},$$
    for some constant $C$.
\end{proof}
\subsubsection{Proof of Theorem \ref{Thm:appro G(z)}}
\begin{proof}
    Let $n$ be the number of samples falling in $\widehat{\bV}_z$. According to Lemma \ref{Thm:Concentration_with_Sample},
    $$
    \sqrt{n}\left(G(z) -\widehat{\mu}_w\right) \overset{d}{\to} N(0,\Sigma),$$
    where $\Sigma\leq r^2 I_D$ and 
    \begin{align*}
        \widehat\mu_w &= \frac{\bE(\beta(Y)Y)}{\bE(\beta(Y))}= \frac{\int y\beta(y)\nu(y)\,dy}{\int \beta(y)\nu(y)\,dy}\\
        &=\frac{\int y w_u(\widehat{U}(y-z)) w_v((I_D - \widehat{U})(y-z))\nu(y)\,dy}{\int w_u(\widehat{U}(y-z)) w_v((I_D - \widehat{U})(y-z))\nu(y)\,dy}.
    \end{align*}
    To obtain the asymptomatic property of $G(z)$, we need to investigate $\widehat{\mu}_w$ first. For simplicity, we define two more expectations:
    \begin{align*}
        \mu_w &=\frac{\int y w_u({U}(y-z)) w_v((I_D - {U})(y-z))\nu(y)\,dy}{\int w_u({U}(y-z)) w_v((I_D - U)(y-z))\nu(y)\,dy}\\
        &=:\frac{\int y\beta^*(y)\nu(y)\,dy}{\int \beta^*(y)\nu(y)\,dy},\\
        \mu_{w,\bD} &= \frac{\int y\beta^*(y)\nu_{\bD}(y)\,dy}{\int \beta^*(y)\nu_{\bD}(y)\,dy}.
    \end{align*}
    In what follows, we will show that $\widehat\mu_w$ is $\cO(\sigma^2\log(1/\sigma))$-close to $\cM$ with high probability. Since $\|\mu_{w,\bD} - \mu_{w}\|_2\leq C\sigma^2\log(1/\sigma)$, from Lemma \ref{Lemma:LocalGeneralDiff}, we only need to show that $\|\widehat\mu_w - \mu_{w}\|_2\leq C\sigma^2\log(1/\sigma)$ with high probability and $\mu_{w,\bD}$ is $\cO(\sigma^2)$-close to $\cM$. 
    
    \vspace{10pt}
    {\noindent \underline{Bound of $\|\widehat\mu_w - \mu_{w}\|_2$}:}
    \vspace{5pt}
    
    According to Theorem \ref{Thm:Angle:F(z)}, $\|\widehat U - U\|_F\leq C_1\sigma\sqrt{\log(1/\sigma)}$ with probability at least $1 - C_2\exp(-C_3\sigma^{-c})$, and the first derivatives of $w_u$ and $w_v$ are both upper bounded by a constant $C$. We have
    \begin{align*}
        |\widehat{W}_u - W_u|
        &=:|w_u(\widehat{U}(y-z)) - w_u({U}(y-z))|\\
        &\leq C\|\widehat U - U\|_F\|y-z\|_2\\
        &\leq C_4\sigma^2\log(1/\sigma),
    \end{align*}
    \begin{align*}
        |\widehat{W}_v - W_v|
        &=:|w_v((I_D-\widehat{U})(y-z)) - w_v((I_D-{U})(y-z))|\\
        &\leq C\|\widehat U - U\|_F\|y-z\|_2\\
        &\leq C_5\sigma^2\log(1/\sigma),
    \end{align*}
    and thus,
    \begin{align*}
        |\beta^*(y)-\beta(y)|
        &= |W_u W_v - \widehat{W}_u\widehat{W}_v|\\
        &= |W_u W_v - W_u\widehat{W}_v + W_u\widehat{W}_v - \widehat{W}_u\widehat{W}_v|\\
        &\leq W_u|\widehat{W}_v - W_v| + \widehat{W}_v|\widehat{W}_u - W_u|\\
        &\leq C_6\sigma^2\log(1/\sigma),
    \end{align*}
    where the last inequality is the result of both $W_u$ and $\widehat{W}_v$ being in the interval $[0,1]$. Then,
    \begin{align*}
        \|\widehat\mu_w - \mu_{w}\|_2
        &=\left\|\frac{\int y\beta(y)\nu(y)\,dy}{\int \beta(y)\nu(y)\,dy} - \frac{\int y\beta^*(y)\nu(y)\,dy}{\int \beta^*(y)\nu(y)\,dy}\right\|_2\\
        &\leq \left\|\frac{\int y(\beta(y)-\beta^*(y))\nu(y)\,dy}{\int \beta(y)\nu(y)\,dy}\right\|_2 \\
        &\quad + \left\|\frac{\int y\beta^*(y)\nu(y)\,dy}{\int \beta^*(y)\nu(y)\,dy}\right\|_2 \left|\frac{\int y(\beta(y)-\beta^*(y))\nu(y)\,dy}{\int \beta(y)\nu(y)\,dy}\right|\\
        &\leq C_6\sigma^2\log(1/\sigma)\frac{1+\|\mu_w - z^*\|_2}{\bE(\beta(Y))}\\
        &\leq C_7\sigma^2\log(1/\sigma),
    \end{align*}
with probability at least $1 - C_2\exp(-C_3\sigma^{-c})$.

\vspace{10pt}
{\noindent \underline{Property of $\mu_{w,\bD}$}:}
\vspace{5pt}

As in the proof of Section 3, we let $z^*$ be the origin and $z - z^*$ be the $(d+1)$-th direction in the Cartesian-coordinate system. We also let $p = (y^{(1)},\cdots,y^{(d)})$, $t = y^{(d+1)}$, and $q = (y^{(d+2)},\cdots,y^{(D)})$. With $U$ the same as before, let $\|u\| = \|U(y-z)\| = |t - \Delta|$, $\|v\| = \|p + q\|_2$. Assume $\mu_{w,\bD} = (\mu^{(1)},\cdots,\mu^{(D)})$; then, the $i$-th element of $\mu_{w,\bD}$, i.e., $\mu^{(i)}$, can be expressed as
\begin{align*}
\frac{\int_{\|p\|_2^2 +\|q\|_2^2\leq r_1^2} \int_{(t-\Delta)^2\leq r_2^2} y^{(i)}w_u(|t-\Delta|) (r_1^2 - \|p\|_2^2- \|q\|_2^2)^k \phi_\sigma(t)\phi_\sigma(q)\,dt\,dp\,dq}{\int_{\|p\|_2^2 +\|q\|_2^2\leq r_1^2} \int_{(t-\Delta)^2\leq r_2^2} w_u(|t-\Delta|) (r_1^2 - \|p\|_2^2- \|q\|_2^2)^k \phi_\sigma(t)\phi_\sigma(q)\,dt\,dp\,dq}.
\end{align*}
For $i\neq d+1$:
\begin{align*}
    \mu^{(i)} &\approx \frac{\int_{\|p\|_2^2 +\|q\|_2^2\leq r_1^2} y^{(i)} (r_1^2 - \|p\|_2^2- \|q\|_2^2)^k\phi_\sigma(q)\,dp\,dq}{\int_{\|p\|_2^2 +\|q\|_2^2\leq r_1^2}(r_1^2 - \|p\|_2^2- \|q\|_2^2)^k\phi_\sigma(q)\,dp\,dq}\\
    &=\frac{\int_{\|p\|_2^2 +\|q\|_2^2+\|\eta\|_2^2\leq r_1^2} y^{(i)}\phi_\sigma(q)\,d\eta\,dp\,dq}{\int_{\|p\|_2^2 +\|q\|_2^2+\|\eta\|_2^2\leq r_1^2}\phi_\sigma(q)\,d\eta\,dp\,dq}\\
    &= 0,
\end{align*}
where $\eta \in \bR^{2k}$ is an auxiliary vector making the above conditional expectation an analogy of Lemma 3.4 in $D+2k-1$-dimensional space.\\
For $i = d+1$, we assume $r_2 = C\sigma\sqrt{\log(1/\sigma)} > 2\Delta$. We have
\begin{align*}
    \mu^{(d+1)} &\approx \frac{\int_{\Delta- r_2}^{\Delta+ r_2} tw_u(|t-\Delta|) \phi_\sigma(t)\,dt}{\int_{\Delta- r_2}^{\Delta+ r_2}w_u(|t-\Delta|) \phi_\sigma(t)\,dt}\\
    &\leq C \int_{\Delta- r_2}^{\Delta+ r_2} tw_u(|t-\Delta|) \phi_\sigma(t)\,dt\\
    &= C \int_{0}^{\Delta+ r_2} t[w_u(|t-\Delta|)-w_u(|t+\Delta|)] \phi_\sigma(t)\,dt\\
    &= C \int_{r_2/2 - \Delta}^{\Delta+ r_2} t[w_u(|t-\Delta|)-w_u(|t+\Delta|)] \phi_\sigma(t)\,dt\\
    &\leq C \int_{r_2/2 - \Delta}^{\Delta+ r_2} t \phi_\sigma(t)\,dt\\
    &\leq C\sigma^2.
\end{align*}
Therefore, $\|\mu_{w,\bD}- z^*\|_2\leq C\sigma^2$.

\vspace{10pt}
{\noindent Combining all the results above, we have}
\vspace{3pt}
$$\|\widehat\mu_w - z^*\|_2\leq \|\widehat\mu_w - \mu_{w}\|_2+\|\mu_{w} - \mu_{w,\bD}\|_2+\|\mu_{w,\bD}- z^*\|_2\leq C\sigma^2\log(1/\sigma),$$
with probability at least $1 - C_2\exp(-C_3\sigma^{-c})$.

According to Corollary \ref{Col:Local_sample_size} and Corollary \ref{col:Concentrate_with_n}, if the sample size $N= C_1r_1^{-d}\sigma^{-3}$, 
$$\|G(z)-z^*\|_2\leq C_2\sigma^2\log(1/\sigma)$$
with probability at least $1 - C_2\exp(-C_3\sigma^{-c})$, for some constant $c$, $C_1$, $C_2$, and $C_3$.

\end{proof}

\subsection{Proof of content in Section 5}
\subsubsection{Proof of Theorem \ref{Thm:out_manifold_1}}
\begin{proof}
    To show $d_H(\mathcal{S},\cM)\leq C\sigma^2\log(1/\sigma)$ is equivalent to showing that
    \begin{equation*}
        \left\{
        \begin{aligned}
            &d(s,\cM) \leq C\sigma^2\log(1/\sigma),~\text{for all } s \in \mathcal{S},\\
            &d(x,\mathcal{S}) \leq C\sigma^2\log(1/\sigma),~\text{for all } x \in \cM.
        \end{aligned}\right.
    \end{equation*}

    The first condition is clear. For any $s\in\mathcal{S}$, there exists a $y_s\in \Gamma$ such that $s = \widehat \mu_{y_s}^{\mathbb{V}}$. Then, according to Theorem \ref{Thm:ContractWithEstDir2}, 
    \begin{equation}
        \label{eq:proof_dHS_1}
        d(s,\cM) \leq \|\mu_{y_s}^{\mathbb{V}}-y_s^*\|_2\leq C\sigma^2\log(1/\sigma).
    \end{equation}

    For the second inequality, let $x$ be an arbitrary point on $\cM$. Then, there exists a point $y_x \in \Gamma$ such that $x$ is its projection on $\cM$. Hence, from Theorem \ref{Thm:ContractWithEstDir2} again, 
    \begin{equation}
        \label{eq:proof_dHS_2}
        d(x,\mathcal{S}) \leq \|x - \mu_{y_x}^{\mathbb{V}}\|_2\leq C\sigma^2\log(1/\sigma).
    \end{equation}
    
    Because \eqref{eq:proof_dHS_1} and \eqref{eq:proof_dHS_2} hold for any $s \in \mathcal{S}$ and $x\in\cM$, we complete the proof.
\end{proof}

\subsubsection{Proof of Theorem \ref{Thm:out_manifold_2}}
\begin{proof}
From the smoothness of $\Gamma$ and $G$, it is evident that $\widehat{\mathcal{S}}$ becomes a smooth manifold. 

For any $s\in\widehat{\mathcal{S}}$, there exists a $y_s\in \Gamma$ such that $s = G(y_s)$. Then, according to Theorem \ref{Thm:appro G(z)}, 
    \begin{equation}
        \label{eq:proof_dHS_appro_1}
        d(s,\cM) \leq \|G(y_s)-y_s^*\|_2\leq C\sigma^2\log(1/\sigma),
    \end{equation}
with a high probability.
For the second inequality, let $x$ be an arbitrary point on $\cM$. Then, there exists a point $y_x \in \Gamma$ such that $x$ is its projection on $\cM$. Hence, from Theorem \ref{Thm:appro G(z)} again, 
    \begin{equation}
        \label{eq:proof_dHS_appro_2}
        d(x,\mathcal{S}) \leq \|x - G(y_x)\|_2\leq C\sigma^2\log(1/\sigma)
    \end{equation}
with a high probability. Thus the proof is completed.
\end{proof}

\subsubsection{Proof of Theorem \ref{Thm:out_manifold_part_d}}
\begin{proof}
By fixing the projection matrix $\widehat{\Pi}_x^{\perp}$ within a neighbour, the function defining $\widehat{\cM}_x$ is a smooth map with constant rank $D-d$, and thus, according to the Constant-Rank Level-Set Theorem, $ \widehat{\cM}_x $ is a properly embedded submanifold of dimension $d$ in $\bR^D$.

To show the distance, let $y$ be an arbitrary point on $\widehat{\cM}_x$. Then there is
$$
\widehat{\Pi}_x^{\perp}(G(y) - y) =  \widehat{\Pi}_x^{\perp}(G(y) - y^* - (y - y^*) ) = 0,
$$
where $y^*$ is the projection of $y$ onto $\cM$. Thus, there is
$$
\begin{aligned}
    \|\widehat{\Pi}_x^{\perp}(y - y^*)\|_2 
    &= \|\widehat{\Pi}_x^{\perp}(G(y) - y^*)\|_2 \\
    &\leq \|G(y) - y^*\|_2\\
    &\leq C\sigma^2(\log(1/\sigma)),
\end{aligned}
$$
with high probability. Since $y\in \cB_D(x,c\tau)$, there exists $c_1\in (0,1)$ such that $\|\widehat{\Pi}_x^{\perp} - {\Pi}_{y^*}^{\perp}\|\leq c_1$ with high probability. Hence, 
$$
\begin{aligned}
    \|\widehat{\Pi}_x^{\perp}(y - y^*)\|_2 
    &\geq \left|\|{\Pi}_{y^*}^{\perp}(y - y^*)\|_2 - \|({\Pi}_{y^*}^{\perp} - \widehat{\Pi}_x^{\perp})(y - y^*)\|_2 \right|\\
    &\geq |1-c_1| \|y-y^*\|_2\\
    &\geq c \|y-y^*\|_2.
\end{aligned}
$$
Therefore, for any $ y\in \widehat{\cM}_x$
$$d(y, \mathcal{M}) = \|y-y^*\|\leq C\sigma^2\log(1/\sigma)$$
with high probability.

\end{proof}

\subsubsection{Proof of Theorem \ref{Thm:out_manifold_global_d}}
\begin{proof}
The proof of (I) and (II) is exactly the same as the proof in Theorem \ref{Thm:out_manifold_2}. To reveal (III), Let $a, b\in \widehat{\cM}$, and $a\neq b$.
When $\|a-b\|_2 \geq c\sigma\tau_0$,  $\|a - b\|_2^2/d(b, T_{a}\widehat{\cM}) \geq c\sigma\tau_0$ is clearly true since $\|a-b\|_2 \geq d(b, T_a\widehat{\cM})$. Hence, we assume that $\|a-b\|_2 < c\sigma\tau_0$. We further denote $a_0 = G^{-1}(a) \in \widetilde{\cM}$ and $b_0 = G^{-1}(b) \in \widetilde{\cM}$. 

Let $J_G$ denote the Jacobi matrix of $G$, then $J_G(a_0)$ is a linear mapping from $T_{a_0}\widetilde{\cM}$ to $T_{a}\widehat{\cM}$. Consider a local chart of $\Gamma$ at $T_{a_0}\Gamma$, then the natural projection from $\widehat{\cM} \cap \cB(a_0, \|b_0-a_0\|_2 )$ to $T_{a_0}\widehat{\cM} \cap \cB(a_0, \|b_0-a_0\|_2 )$ is an invertible mapping. Denote the inverse mapping of the natural projection as $\phi$, and then there exists $\eta_{b_0}\in T_{a_0}\widehat{\cM}$ such that $\phi(0) = a_0$, and $\phi(\eta_{b_0}) = b_0$. Since $\|a-b\|_2 < c\sigma\tau_0$, there exists $0<c<C$ such that
$$
c\|a_0-b_0\|_2 \le \|\eta_{b_0} - \eta_{a_0}\|_2 = \|\eta_{b_0}\|_2 \le C\|a_0-b_0\|_2.
$$
Using the Taylor expansion of G at $a_0$, there is
\begin{align*}
    d(b, T_{a}\widehat{\cM}) &\leq\|b - J_G(a_0)\eta_{b_0} - G(a_0)\|_2\\
    &= \|G(\phi(\eta_{b_0})) - J_G(a_0)\eta_{b_0} - G(a_0)\|_2\\
    & = \|G(\eta_{b_0}) - J_G(a_0)\eta_{b_0} - G((a_0))\|_2 + \|G(\eta_{b_0})-G(\phi(\eta_{b_0}))\|_2\\
    &\leq \|H_G(z_1)\|_2\|\eta_{b_0}\|_2^2 + \|J_G(z_2)\|_2\|\eta_{b_0}- \phi(\eta_{b_0})\|_2\\
    &\leq C(M_G + L_G)\|\eta_{b_0}\|_2^2\\
    &\leq C(M_G + L_G)\|a_0-b_0\|_2.
\end{align*}
Here, $H_G$ is the Hessian matrix of $G$, $M_G$ and $L_G$ are the upper bound of $\|H_G\|_2$ and $\|J_G\|_2$. Moreover, 
$$
\|a_0-b_0\|_2 \le \frac{1}{\ell_G} \|G(a_0) - G(b_0)\|_2 = \frac{1}{\ell_G} \|G(a_0) - G(b_0)\|_2,
$$
where $\ell_G$ is the lower bound of $J_G$. Hence, we have
$$
d(b, T_{a}\widehat{\cM}) \leq C\frac{M_G + L_G}{\ell_G} \|a-b\|^2_2.
$$
Finally, the reach of $\widehat{\cM}$ can be bounded below as
$$
 {\reach}(\widehat{\cM}) \geq  \min\left\{c\sigma\tau_0, \  c\frac{\ell_G}{M_G + L_G}\right\}.
$$
\end{proof}

\subsubsection{Proof of Proposition \ref{Prop:M_in_d_dim}}
\begin{proof}
    Since $\widetilde{\cM} \subset \Gamma$, it is clear that $d_H(\widetilde{\cM},\cM)\leq C\sigma$. In the following section, we show that $\operatorname{dim} \widetilde{\cM} = d$.
    
    Recall that $F(y) = \sum_i \alpha_i(y) y_i$, with $\sum_i \alpha_i(y) = 1$. Let $H(y) = F(y)-y$; then, we have 
    \begin{align*}
        H(y)
        &= \sum_i \alpha_i(y) y_i - y\\
        &= \sum_i \alpha_i(y) (y_i - y).
    \end{align*}
    According to Lemma 17 and Theorem 18 in \cite{yao2019manifold}, for any unit norm direction vector $v\in\bR^D$,
    $$\|\partial_v H(y) - v\|_2 \leq Cr_0,$$
    with high probability. In the case of $\sigma$ being sufficiently small, the Jacobian matrix of $H$, denoted by $J_H$, is full rank. For any fixed arbitrary rank $D-d$ projection matrix $\Pi^*$,
    $$
        \Pi^*H: \bR^D\to \bR^D,
    $$
    $$
        J_{\Pi^*H} = \Pi^*J_H.
    $$
    In other words, $\Pi^*H$ is a smooth map with constant rank $D-d$, and thus, according to the Constant-Rank Level-Set Theorem, $\widetilde{\cM} = \{y\in\Gamma: \Pi^*H(y) = 0\}$ is a properly embedded submanifold of co-dimension $D-d$ in $\Gamma$. Therefore, $\operatorname{dim} \widetilde{\cM} = d$.
\end{proof}

\section{Supplement to the simulation}

\begin{figure}[ht]
    \centering
    \includegraphics[width = 0.9\linewidth, height = 0.34\linewidth]{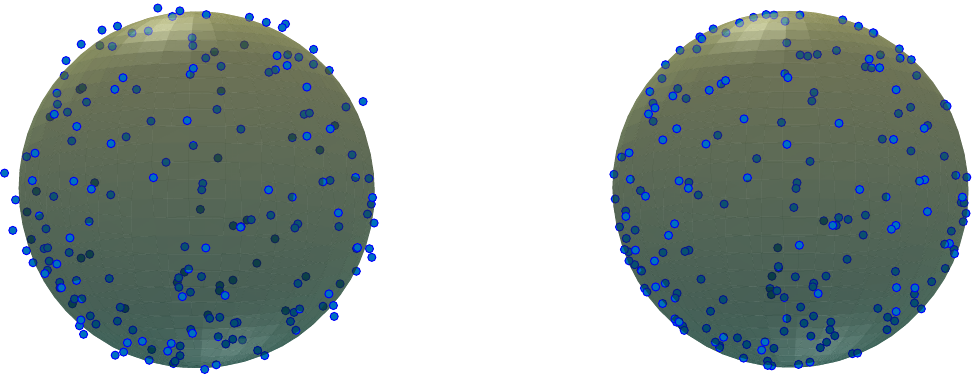}
    \caption{Assessing the performance of ysl23 in fitting the sphere ($N = 5\times 10^4$, $N_0 = 100$, $\sigma = 0.06$): the left panel displays points in $\mathcal{W}$ surrounding the underlying manifold, while the right panel illustrates the corresponding points in $\widehat{\mathcal{W}}$.}\label{Fig:per_sphere}
\end{figure}

\begin{figure}[ht]
    \centering
    \includegraphics[width = 0.85\linewidth, height = 0.26\linewidth]{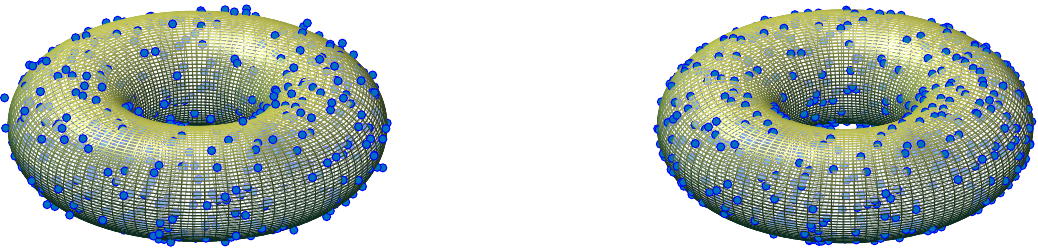}
    \caption{Assessing the performance of ysl23 in fitting the torus ($N = 5\times 10^4$, $N_0 = 100$, $\sigma = 0.06$): the left panel displays points in $\mathcal{W}$ surrounding the underlying manifold, while the right panel illustrates the corresponding points in $\widehat{\mathcal{W}}$.}\label{Fig:per_torus}
\end{figure}

\begin{figure}[ht]
    \centering
    \includegraphics[width = 1\linewidth, height = 0.5\linewidth]{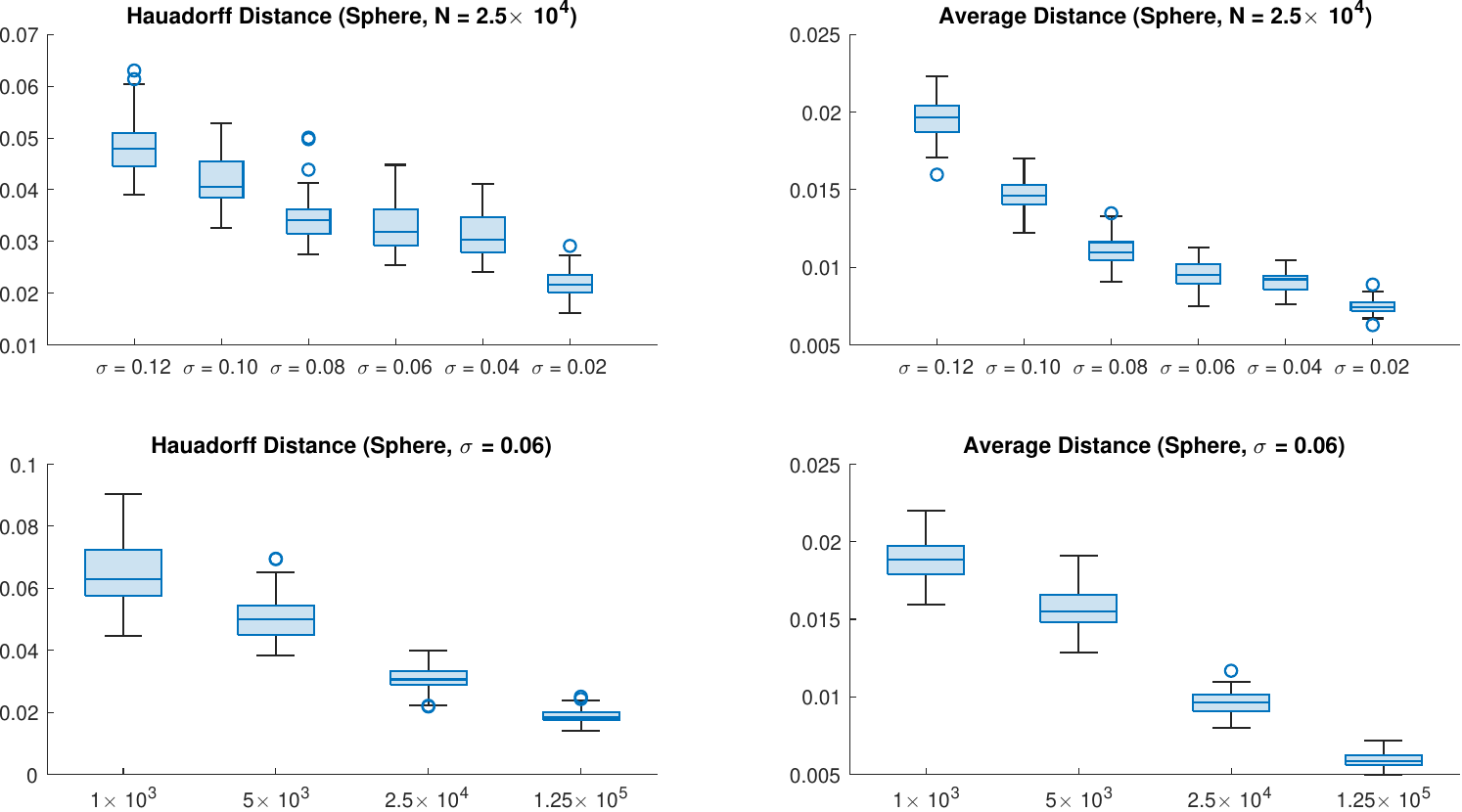}
    \caption{The asymptotic performance of ysl23 when fitting a sphere. The top two figures show how the two distance change with $\sigma$, while the bottom two figures show how the distances change with $N$.}\label{Fig:asym_sphere}
\end{figure}

\begin{figure}[ht]
    \centering
    \includegraphics[width = 1\linewidth, height = 0.5\linewidth]{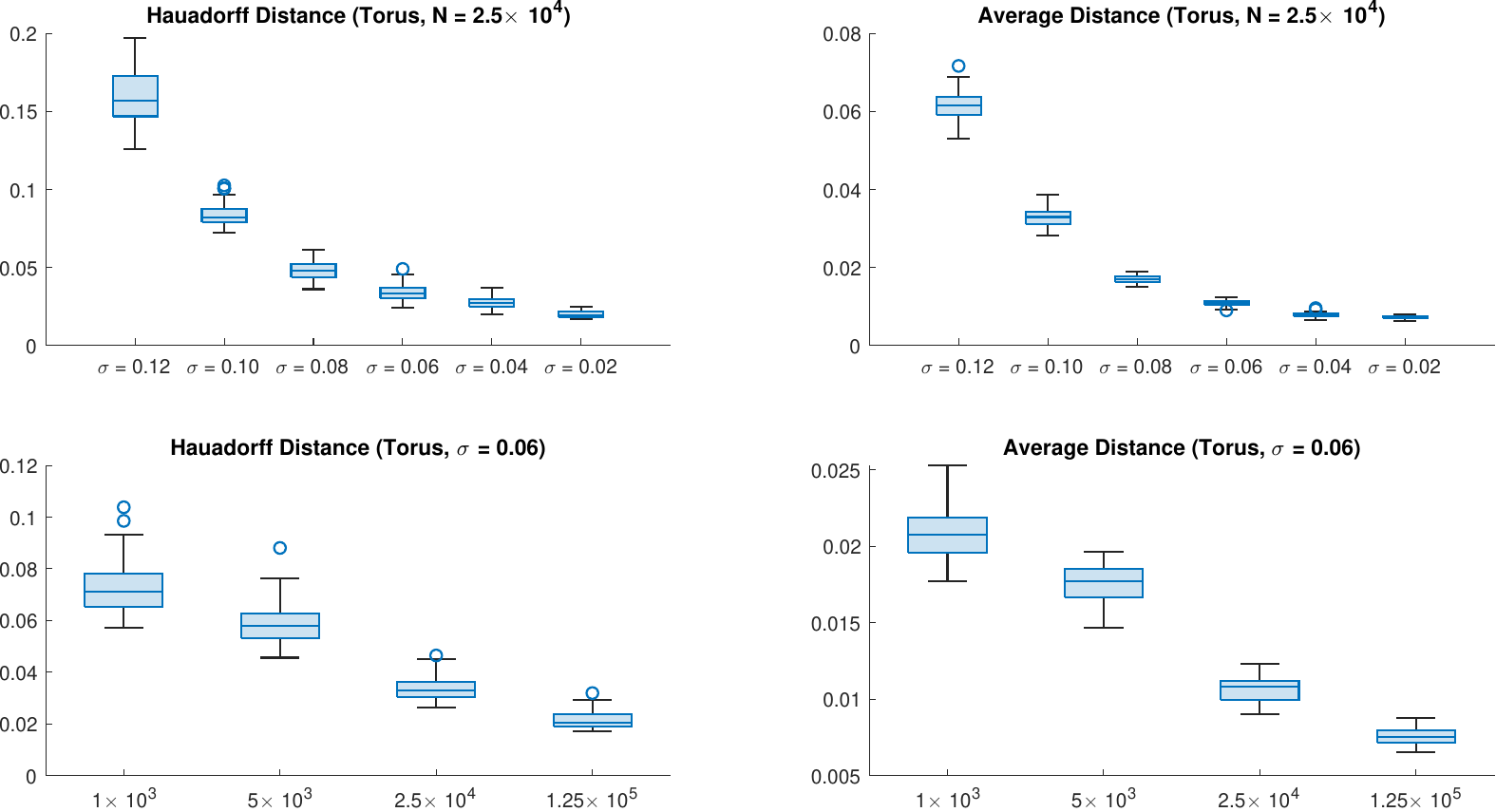}
    \caption{The asymptotic performance of ysl23 when fitting a torus. The top two diagrams show how the two distance change with $\sigma$, while the bottom two figures show how the distances change with $N$.}\label{Fig:asym_torus}
\end{figure}

\begin{figure}[ht]
    \centering
    \includegraphics[width = 1\linewidth, height = 0.7\linewidth]{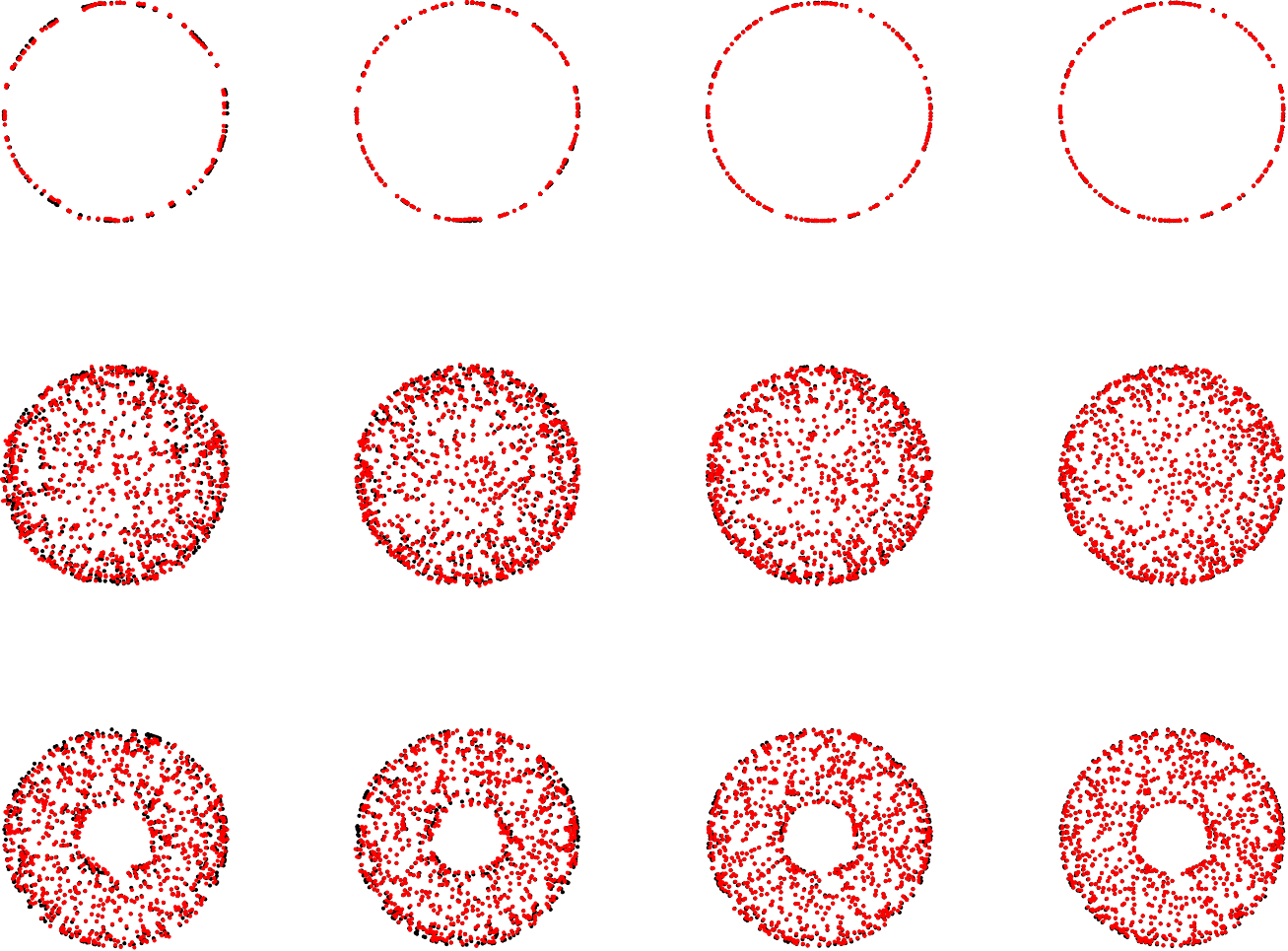}
    \caption{The performance of ysl23 with increasing $N$. Top row, from left to right:  $N = 3\times 10^2$, $3\times 10^3$, $3\times 10^4$, $3\times 10^5$. Middle row, from left to right: $N = 1\times 10^3$, $5\times 10^3$, $2.5\times 10^4$, $1.25\times 10^5$. Bottom row, from left to right: $N = 1\times 10^3$, $5\times 10^3$, $2.5\times 10^4$, $1.25\times 10^5$. It can be observed that for each example, as the number of samples increases, the distribution of $\mathcal{W}$ output by ysl23 becomes more uniform.}\label{Fig:scatter_more}
\end{figure}

\end{appendix}
\clearpage

\begin{supplement}
\stitle{Supplementary material for ``Manifold Fitting: an Invitation to Statistics"}
\slink[doi]{COMPLETED BY THE TYPESETTER}
\sdatatype{.pdf}
\sdescription{We include all materials omitted from the main text. }
\end{supplement}
\bibliographystyle{imsart-number} 
\bibliography{references}
\end{document}